\documentclass[11pt]{article}
\usepackage[utf8]{inputenc}
\usepackage{geometry}
\usepackage{amsthm,amssymb,amsmath,mathtools}

\usepackage[spaces,hyphens]{url}

\usepackage{hyperref}
\usepackage{blindtext}
\usepackage{titlesec}
\usepackage{bbm}
\hypersetup{
  colorlinks   = true, %Colours links instead of ugly boxes
  urlcolor     = blue, %Colour for external hyperlinks
  linkcolor    = black, %Colour of internal links
  citecolor   = red %Colour of citations
}
\usepackage{pdfpages}
\usepackage[all]{hypcap}
\usepackage[font=small]{caption}
\usepackage[rightcaption]{sidecap}
\theoremstyle{plain}
\newtheorem{theorem}{Theorem}[section]
\newtheorem{lemma}[theorem]{Lemma}
\newtheorem{corollary}[theorem]{Corollary}
\newtheorem{proposition}[theorem]{Proposition}
\theoremstyle{definition}
\newtheorem{remark}[theorem]{Remark}

\theoremstyle{definition}
\newtheorem{definition}[theorem]{Definition}

\numberwithin{equation}{section}

\newcommand{\ellfunc}{\lambda}

\newcommand{\lawstopped}{\nu}

\newcommand{\e}{\varepsilon}

\title{Particle systems and the supercooled Stefan problem with non-integrable initial data}
\author{Thomas Blore\thanks{Mathematical Institute, University of Oxford. \textit{Email:} thomas.blore@maths.ox.ac.uk },  D.G.M. Flynn\thanks{Mathematical Institute, University of Oxford. \textit{Email:} donald.flynn@maths.ox.ac.uk } and Ben Hambly\thanks{Mathematical Institute, University of Oxford. \textit{Email:} hambly@maths.ox.ac.uk }} 

\begin{document}
\maketitle

\begin{abstract}
    We consider an infinite system of particles on the positive real line, initiated from a Poisson point process, which move according to Brownian motion up until the hitting time of a barrier. The barrier increases when it is hit, allowing for the possibility of sequences of successive jumps to occur instantaneously. Under certain conditions, the scaling limit gives a representation for the supercooled Stefan problem and its free boundary. This allows us to give a precise asymptotic limit for the barrier and determine the rate of convergence. From this representation, we also investigate properties of the supercooled Stefan problem for initial data not in $L^1(\mathbb{R}^+)$. In a critical case, where the jump size matches the density of the Poisson process and the Stefan problem has an instantaneous explosion, we instead recover a scaling limit result dependent on the initial fluctuations of the point process.
\end{abstract}

\tableofcontents
\section{Introduction}

Consider a system of particles initially placed via a point process on $\mathbb{R}^+$ and a moving barrier, which starts at the origin. Each particle moves diffusively, independently of the other particles, until it hits the barrier, at which point the particle ``sticks" and the barrier height increases by a fixed amount. If the barrier moves over particles, they too ``stick" and cause further growth of the barrier. This provides a simple model for particle aggregation, where the barrier represents the boundary of a sticky aggregate, and a particle that makes contact with the boundary sticks to the aggregate, thereby causing the boundary of the aggregate to move.

In this work, we investigate the following specific version of this aggregation model. Let $g$ be a non-negative, measurable, locally integrable function, and $a>0$ a positive constant. At time zero a random configuration of particles is placed via a Poisson point process of intensity $g$ on the positive half line at points
$$0\leq W_0^{1}\leq ...\leq W_0^{n}\leq ... .$$

Particles diffusive in that, independently of the initial particle configuration, the time $t$ position of particle $i$ is given by
$$W^{i}_t=W^{i}_0+B^{i}_t,$$
where each $B^{i}$ is an independent standard Brownian motion. 

This is coupled to a moving barrier with position $\xi^g_t$ at time $t$, which is defined recursively in the following manner. At time $0$, we set $\xi^g_0=0$. $\xi^g_t$ remains at 0, until time $T_0$, the first hitting time of 0 by any particle. At this time, the barrier height increases to $ak$, where $k$ is the smallest integer such that the interval $[\xi^g_{T_0^-},\xi^g_{T_0^-}+ak]$ contains $k$ particles. Any particles that are now below $\Lambda_t$ are referred to as ``absorbed" and have no further effect on the movement of $\xi$. The barrier remains at $\xi^g_{T_0}$ until $T_1$, the first hitting time of $\xi^g_{T_0}$ by any non-absorbed particle. At time $T_1$, the barrier height again increases by $ak$, where $k$ is the smallest integer such that the interval $[\xi^g_{T_1^-},\xi^g_{T_1^-}+ak]$ contains $k$ particles, and these particles are ``absorbed". This mechanism for barrier growth is repeated at each subsequent hitting time $T_n$ of $\xi^g_{T_{n-1}}$ by any non-absorbed particle, which leads to the following equations for the system dynamics:
\begin{align}
\label{xidefnequation}
    &W_t^i=W_0^i+B_t^i,\nonumber\\
    &\tau_i=\inf\{t\geq 0:W_t^i\leq \xi_t^g\},\nonumber\\
    &\xi^g_t=\sum_{i=1}^\infty \mathbbm{1}_{\{\tau_i\leq t\}},\\
    &\Delta \xi^g_t=\inf\{x:\rho_{t-}([\xi^g_{t^-},\xi^g_{t^-}+a x])<x\},\nonumber\\
    &\rho_{t^-}= \left(\sum_{i=1}^\infty\delta_{W_{t}^i}\mathbbm{1}_{\{\tau_i\geq t\}}\right).\nonumber
\end{align}

A simulation of the system is shown in Figure \ref{fig1}.

\begin{figure}
    \centering
    \includegraphics[width=0.9\linewidth]{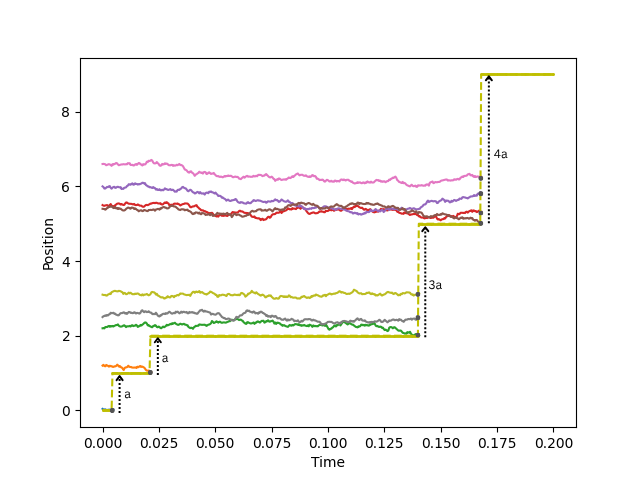}
    \caption{Dynamics for the particle system, and moving barrier, with $a=1$ and the lowest $9$ initial particles.}
    \label{fig1}
\end{figure}

In \cite{7}, a variant of this model is studied, where $g\equiv 1$, and $a<1.$ We define $\xi_t=\xi_t^g$ for the system \eqref{xidefnequation} with $g\equiv 1$. The authors of \cite{7} conjecture that under a rescaling, the barrier height
$\xi_t/\sqrt{t}$ converges to some distribution. We investigate this conjecture, and the further behaviour of the particle system, by connecting the system to the supercooled Stefan problem, a simplified model for the freezing of a supercooled liquid, which has attracted attention recently in the probability literature \cite{2,3,minimal,PDEresult,NSsurfacetension,nadtochiy2022tension}. By adapting a monotonicity property from \cite{Dembo_2019}, and extending previous work on probabilistic formulations of the supercooled Stefan problem \cite{delarue2015particle, 5} to the case of non-integrable initial data, we are able to determine the exact form of the limit for $\xi_t/\sqrt{t}$ when $a<1$, and the rate of convergence. In a critical case where $a=1$, we instead obtain a limit process for $\xi_t/t^{2/3}$, which is essentially determined by the fluctuations of the initial Poisson point process against its mean.

We also use the probabilistic formulation to discuss the behaviour of the free boundary in the supercooled Stefan problem when started from different initial conditions. We will show that it can jump to infinity, have infinitely many finite size jumps or can move continuously at different speeds. 

\subsection{The supercooled Stefan problem}

The one-dimensional one-phase supercooled Stefan problem can be formulated as the following free boundary problem for a function $V:\mathbb{R}_+\times\mathbb{R}_+\to \mathbb{R}$ and a free boundary $\Lambda:\mathbb{R}_+\to\mathbb{R}_+$:
\begin{align}
    \partial_t V(t,x) &=\frac{1}{2}\partial_{xx} V(t,x), \quad x>\Lambda_t,\nonumber\\
     \dot{\Lambda}_t &=\frac{1}{2}\partial_x V(t,\Lambda_t),\label{pde}\\
    V(t,\Lambda_t) &=0, \qquad\qquad t\geq 0 &\nonumber\\
    V(0,x) &=g(x),\qquad\;\; x>0.\nonumber
\end{align}

The function $-V(t,x)$ is a rescaling of temperature at a time $t$ and position $x$, where the freezing temperature is taken as zero. The free boundary $\Lambda_t$ is the position of the interface between the solid and liquid phases. As the initial data $g$ is non-negative, the temperature $-V(0,x)\leq 0$ is below the freezing point, leading to growth of the solid phase.  
The limiting behaviour of the barrier $\xi_t$ in our particle system will be determined by the behaviour of the free boundary in the Stefan problem \eqref{pde} for non-integrable initial data $g$. Two types of explicit special solutions for this system are known: similarity solutions and travelling wave solutions. These, respectively, have initial data given by $g(x)=a$ for any $a<1$, and $g(x)=1-e^{-vx}$ for any $v>0$. The free boundary is then given, respectively, by $\Lambda_t=K_{a}\sqrt{t}$, and $\Lambda_t=\frac{vt}{2}$, where the constant $K_a$ is the unique solution to
\begin{align}
\label{kalpha}
a=K_a\sqrt{2\pi} (1-\Phi(K_a))e^{K_a^2/2}.
\end{align} 
We use $\Phi$ to denote the standard Gaussian cumulative distribution function.

Explicit solutions to \eqref{pde} cannot be determined in general, however, 
for non-integrable initial data the long-time asymptotic behaviour of classical solutions to \eqref{pde} can be obtained
for various classes of initial data \cite{Ricci_Weiqing_1991,4,chadam}.

\subsection{Probabilistic reformulation of the Stefan problem}

For a generic function $g$, classical solutions to \eqref{pde} need not exist. In particular, the PDE may exhibit blow-up in finite time, where the derivative of $\Lambda$ ceases to exist, and jump discontinuities in the movement of the free boundary may occur \cite{2}. This causes weak-formulations of the system to be ill-posed. This issue was resolved for a large class of initial densities by Delarue et al. \cite{3} via a reformulation of the system as the solution to the following McKean-Vlasov type equation:
\begin{align}
    X_t &=X_{0^-}+B_t-\Lambda_t,\nonumber\\
    \tau &=\inf\{t\geq0:X_t\leq 0\},\label{mv}\\
    \Lambda_t &=\alpha\mathbb{P}(\tau\leq t),\nonumber\\
    \Delta\Lambda_t&=\inf\{x>0:\mathbb{P}(X_t\in [0,x],\tau>t)<x\}.\nonumber
\end{align}
In the above equation $\alpha=||g||_{L^1(\mathbb{R^+})}<\infty$, the initial condition $X_{0^-}$ has the probability density $g/\alpha$, and $X_{0^-}$ is independent of $B$, a standard Brownian motion on $\mathbb{R}$. The last equation in \eqref{mv} is called the physical jump condition and captures the jump discontinuities in the free boundary. This provides a weak formulation of the supercooled Stefan problem \eqref{pde}, where the free boundary is given by $\Lambda_t$, and $V(t,x)$ corresponds to the density of the McKean-Vlasov process at time $t$:
$$V(t,x)=\alpha\mathbb{P}(X_t+\Lambda_t\in dx,\tau>t).$$

The above probabilistic formulation \eqref{mv} requires $g\in L^1(\mathbb{R}^+)$ and so the free boundary $\Lambda_t$ for this system is uniformly bounded over time. We are interested in the situation where $g\notin L^1(\mathbb{R}^+)$, and the free boundary can grow to infinity. Therefore, we consider the following generalisation of the McKean-Vlasov system:
\begin{align}
    X^x_t &=x+B_t-\Lambda_t,\nonumber\\
    \tau^x &=\inf\{t\geq 0:X^x_t\leq 0\},\label{alternativetomvr}\\
    \Lambda_t &=\int_0^\infty \mathbb{P}(\tau^x\leq t)g(x)dx.\nonumber
\end{align}
As above, $B$ is a standard Brownian motion, and $g$ is a positive, measurable function. We also enforce that the free boundary has ``physical" jumps:
\begin{align}
\label{alternativephysical}
 \Delta \Lambda_t=\inf\left\{x>0:\int_0^\infty\mathbb{P}(X_t^y\in[0,x],\tau^y>t)g(y)dy<x\right\}.
\end{align}

\subsection{Particle system reformulation}
\label{reformulationsec}
In order to relate the behaviour of the McKean-Vlasov system \eqref{alternativetomvr}, to the particle system studied in \eqref{xidefnequation}, we introduce the following family of particle systems which are parametrised by integers $N$:
\begin{align}
 &W_t^{i,N} = W_{0}^{i,N}+B_t^{i,N},\nonumber \\
&{} \tau_{i}^N=\inf \left\{t \geq 0: W_t^{i,N} \leq \Lambda^N_t\right\},
\label{aligned:Particlealphanocentre} \nonumber\\
&\Lambda^N_t  =\frac{1 }{N}\sum_{i=1}^\infty \mathbbm{1}_{\left\{ \tau_{i}^N\leq t\right\}}{}_{,}\\
&\Delta \Lambda^N_t=\frac{1}{N}\inf\left\{k\in\mathbb{N}:\rho^N_{t^-}\left(\left[\Lambda^N_{t^-},\Lambda^N_{t^-}+\frac{k}{N}\right]\right)\leq \frac{k}{N}\right\}_{,}\nonumber\\
&\rho^N_{t^-}=\frac{1}{N}\sum_{i=1}^\infty \delta_{W_t^{i,N}}\mathbbm{1}_{\{t\geq \tau_{i}^N\}}\nonumber.
\end{align}
The points $(W_0^{i,N})_{i\in\mathbb{N}}$ are the points of a Poisson process of intensity $Ng$, which are independent from the collection of i.i.d. standard Brownian motions $(B_t^{i,N})_{i\in\mathbb{N}}$. The measure $\rho^N_{t^-}$ is the empirical measure of the non-absorbed particles just before time $t$. For a fixed $N$, this is exactly the system \eqref{xidefnequation}, with $a=1/N$, and with the intensity function $g$ replaced by $Ng$. Under certain conditions, the McKean-Vlasov equation \eqref{alternativetomvr} can be obtained from taking a limit over these particle systems.

In the case that $g\equiv a$, the system \eqref{aligned:Particlealphanocentre} corresponds to a rescaling of \eqref{xidefnequation}. We set $W_t^i=N\delta W_{t/N^2a^2}^{i,N}$ and $\xi_t=N\delta\Lambda^N_{t/N^2a^2}$. From this we obtain a system with identical dynamics to \eqref{xidefnequation} with initial intensity $g\equiv 1$. From this scaling relation, $\xi_{N^2a^2t}=Na\Lambda^N_{t}$, we may transfer information on the long time behaviour of system \eqref{xidefnequation}, with intensity $g\equiv 1$, to information on the system \eqref{aligned:Particlealphanocentre}, and vice versa.

\subsection{Main Results}

Our aims are to understand properties of the particle system \eqref{xidefnequation}, as well as the free boundary $\Lambda_t$ in the related McKean-Vlasov equation \eqref{alternativetomvr}.

Our first key result, in Theorem \ref{maintheorem}, is to establish the global existence of physical solutions to the McKean-Vlasov type equation \eqref{alternativetomvr} under conditions on the initial data to avoid explosion of $\Lambda$ to infinity in finite time. We obtain this existence result by demonstrating that the particle system \eqref{aligned:Particlealphanocentre} converges to a solution of the McKean-Vlasov type equation \eqref{alternativetomvr}. 

Under the further restriction that $g$ satisfies the weak-feedback constraint considered in \cite{5} a unique solution $\Lambda_t$ to \eqref{alternativetomvr} exists. In Theorem \ref{bdtheorem}, we find that under this constraint, $\Lambda^N_t-\Lambda_t$ is of size $1/\sqrt{N}$. As a consequence, we determine the limiting behaviour of $\xi_t/\sqrt{t}$ in the case of $a<1$, and that the fluctuations around this limit are of size $t^{-1/4}$.

\begin{theorem}
\label{conjtheorem}
For $a<1$, the barrier $\xi_t$ satisfies
 \[ \xi_t/\sqrt{t} \to K_a, \mbox{ in probability as $t\to\infty$}.  \]
 The constant $K_a$ is given by \eqref{kalpha}.  Also, there exists a constant $C_a>0$ such that for $x<t^{1/4}$, $$\mathbb{P}(|\xi_t/\sqrt{t}-K_a|>xt^{-1/4})\leq 2e^{-C_a x^2}.$$
\end{theorem}

The methodology we use to prove Theorem \ref{conjtheorem} is subsequently extended to the critical case of $a=1$, where we need a different scaling, and therefore the resultant limiting process is not related to the supercooled Stefan problem.
\begin{theorem}
\label{theoremoncrit}
For $a=1$, the barrier $\xi_t$ satisfies that 
$$\xi_{Nt}/N^{2/3}\implies R_t\textrm{ as }N\rightarrow\infty$$
where $R_t$ is defined via
$$R_t:=\inf\left\{x:\int_0^x2(B(s))_+ds>t\right\}$$
for a standard Brownian motion $B(s)$.
\end{theorem}

We make the meaning of the convergence in distribution, $\xi_{Nt}/N^{2/3}\implies R_t$, precise in Theorem \ref{crittheorem}.

For $g\notin L^1$ the free boundary in the McKean-Vlasov equation \eqref{alternativetomvr} can display a range of behaviours. We show, in Proposition~\ref{prop:infinitejumps}, that there are initial conditions which allow for the existence of solutions with an infinite number of jump discontinuities in $\Lambda$. We show also in Theorem \ref{speed} that in the presence of blow-ups, the asymptotic behaviour of solutions to this system is determined by the asymptotic behaviour of the corresponding classical solutions to the supercooled Stefan problem \eqref{pde}. Aside from the square root and linear growth shown for the free boundary we also note, in Remark~\ref{rem:speeds}, that there is a wide range of polynomial growth rates for the free boundary that are possible.

\subsection{Related Literature}

The particle system \eqref{xidefnequation} may be considered a continuous space (off-lattice) analogue of the frictionless multiparticle diffusion limited aggregation model (MDLA) studied in \cite{Dembo_2019}. In that model, $N_i$  particles are initially placed at each integer $i$, where each $N_i$ are independent Poisson random variables of mean 1. Particles move as independent, continuous time, simple symmetric random walks on $\mathbb{Z}$; and cause a barrier of height $\xi_t$ to grow when moving into contact with this barrier. The authors consider the critical case where $a=1$, and obtain an identical scaling limit for $\xi_t/t^{2/3}$ as we give in Theorem \ref{theoremoncrit}.
Such MDLA processes have previously been connected to the supercooled Stefan problem. In the work of \cite{dlachayesswindle}, a version of MDLA in one dimension is considered, where the particles move as a continuous time exclusion process on the lattice. It is then shown that a Stefan problem is obtained in the hydrodynamic limit. Such connections between MDLA and Stefan problems are also shown to hold in dimension two in \cite{dlaNS}, in which the authors prove that MDLA converges to a probabilistic analogue to the Stefan problem. 

The McKean-Vlasov formulation of the supercooled Stefan problem \eqref{mv} arose from problems in neuroscience \cite{DIRT1, delarue2015particle} and in systemic risk \cite{2}. The initial work on the neuroscience variant of this system, in \cite{DIRT1, delarue2015particle}, showed that there could be discontinuities and established existence of solutions under the physical jump condition. Subsequently, further properties of this system have been investigated, such as regularity \cite{2,3}, uniqueness of solutions \cite{5, mustapha2023wellposednesssupercooledstefanproblem, 3}, and different solution concepts \cite{minimal,zeroundercool}. 

The supercooled Stefan problem with a cumulative distribution function for its initial condition also arises in a model for cell polarization \cite{Calvez2012, BHJ2025}.

\subsection{Outline of the paper}

The structure of the paper is as follows. In Section \ref{setupsection} we discuss further the particle system representation \eqref{aligned:Particlealphanocentre} and give conditions under which we can expect global in time solutions to \eqref{alternativetomvr} to exist. In Section \ref{convergencestatementsection} we describe our general result on convergence of the particle system \eqref{alternativetomvr}, and give a sketch of the proof. In Section \ref{uniquenesssection} we give results for a ``weak feedback" regime, in particular we prove the quantitative rate of convergence given in Theorem \ref{conjtheorem}. We then extend these results to the case where $g$ is a non-decreasing function, and connect this to the model of \cite{BHJ2025} to improve the convergence results of \cite[Theorem 1.14]{BHJ2025}. In Section \ref{critsec}, we further extend the arguments of Section \ref{uniquenesssection} to prove Theorem \ref{theoremoncrit}.
In Section \ref{alternateg} we then discuss different conditions on $g$ which allow for blow-ups after any time point and also for polynomial asymptotic speeds for the barrier $\Lambda$. In Section \ref{asymptoticsection} we prove that the asymptotic speed of $\Lambda$ is unaffected by earlier discontinuities. In Section \ref{uniquenessappendix} we comment on conditions for uniqueness of solutions to \eqref{alternativetomvr}.
Finally, in Section \ref{proofsection}, we give the proof of our main convergence result in detail.

\section{Further details on the particle systems}
\label{setupsection}
\subsection{Reformulations of the particle systems}

As stated in Section \ref{reformulationsec}, the particle system for the barrier $\xi_t$, can be seen as a rescaling of the family of systems \eqref{aligned:Particlealphanocentre}, and properties of $\Lambda^N_t$ and $\xi_t$ can be related to each other. We therefore consider the family of particle systems \eqref{aligned:Particlealphanocentre}. To directly compare these particle systems to those appearing in earlier work on probabilistic formulations of the supercooled Stefan problem, it is convenient to re-centre the particle system so that the origin is shifted to the position $\Lambda^N_t$. We therefore define $X_t^{i,N}=W_t^{i,N}-\Lambda^N_t$ and re-write the system \eqref{aligned:Particlealphanocentre} in the following way:
\begin{align}
 &X_t^{i,N} = W_{0}^{i,N}+B_t^{i,N}-\Lambda^N_t,\nonumber \\
&{} \tau_{i}^N =\inf \left\{t \geq 0: X_t^{i,N} \leq 0\right\},
\label{aligned:Particlealpha}\nonumber \\
&\Lambda^N_t  =\frac{1 }{N}\sum_{i=1}^\infty \mathbbm{1}_{\left\{ \tau_{i}^N \leq t\right\}}{}_{,}\\
&\Delta \Lambda^N_t=\frac{1}{N}\inf\left\{k\in\mathbb{N}:\nu_{t^-}^N([0,\frac{k}{N}])\leq \frac{k}{N}\right\},\nonumber\\
&\nu^N_{t^-}=\frac{1}{N}\sum_{i=1}^\infty \delta_{X_{t^-}^{i}}\mathbbm{1}_{\{t\geq \tau_{i}^N\}}.\nonumber
\end{align}
If we were to replace the initial Poisson placed points $w^i$, by $N$ i.i.d. random variables $(W_0^{i,N})_{i=1,...,N}$ then this re-centred system would then become identical to a simplified version of the particle system considered in \cite{feinstein2022contagious}.

The particle system \eqref{aligned:Particlealpha} can also be viewed as $N$ interacting systems by relabelling the particles appropriately, so that we thin the Poisson process into $N$ independent Poisson processes of intensity $g$. We couple the above particle system \eqref{aligned:Particlealphanocentre} with a collection $(p_i)_{i\in\mathbb{N}}$ of i.i.d. uniform random variables on $\{1,...,N \}$. We relabel the particles using this collection in the following manner. For each $i\in\mathbb{N}$ we assign to $W^{i,N}_t$ the random index $(p_i,I_i)=(p_i,\sum_{j=1}^i\mathbbm{1}_{\{p_j=p_i\}})\in\{1,...,N\}\times\mathbb{N}$. 
Abusing notation, we then write $W^{j,k}_t$ to denote $W^{i,N}_t$, where $i$ is the unique integer such that $(p_i,\sum_{j=1}^i\mathbbm{1}_{\{p_j=p_i\}})=(j,k)$. We omit the dependence on $N$ in this notation as it is clear when we use this relabelling that we are considering the system parametrised by $N$.

    By standard results on the thinning of Poisson random measures, for each $i=1,..,N$, the initial configuration of $W^{i,k}$ particles, $(W_0^{i,k})_{k\in\mathbb{N}}$, are the points of $N$ independent Poisson point processes of intensity $g$. We may rewrite the re-centred system \eqref{aligned:Particlealpha} using these relabelled particles in the following manner:
\begin{align}
 &X_t^{i,k} = W_{0}^{i,k}+B_t^{i,k}-\Lambda^N_t,\nonumber \\
&{} \tau_{i}^k =\inf \left\{t \geq 0: X_t^{i,k} \leq 0\right\},\label{aligned:Particlewarticle} \\
&\Lambda^N_t  =\frac{1 }{N}\sum_{i=1}^N\sum_{k=1}^\infty \mathbbm{1}_{\left\{ \tau_{i}^k \leq t\right\}}{}_{.}\nonumber
\end{align}

The barrier height, $\Lambda^N$, can now be seen explicitly as the empirical average of $N$ interacting terms, suggestive of the following McKean-Vlasov type limit
\begin{align}
   &X_t^i=X_{0^-}^i+B_t^i-\Lambda_t,\nonumber\\
   &\tau_i=\inf\{t\geq0:X_t^i\leq 0\},\label{aligned:MVR}\\
   &\Lambda_t=\mathbb{E}\left(\sum_{i=1}^\infty\mathbbm{1}_{\{\tau_i\leq t\}}\right).\nonumber
\end{align}

The collection $(X_{0^-}^i)_{i\in\mathbb{N}}$ are the points of a Poisson point process of intensity $g$, and which are independent of a collection $(B_t^i)_{i\in\mathbb{N}}$ of independent Brownian motions.

Solutions to this system are equivalent to solutions of \eqref{alternativetomvr}, and the ``physical" jump condition \eqref{alternativephysical} is equivalent to the following condition on jump sizes:
\begin{align}
\Delta \Lambda_t=\inf\{x>0:\nu_{t^-}([0,x])<x\}, \label{phys3}
\end{align}
where
$$\nu_{t-}(A)=\sum_{i=1}^\infty \mathbb{P}(X_{t-}^i\in A,\tau_i\geq t).$$

\subsection{Assumptions on the density g}
From now on we shall refer to the intensity function $g$ as a density, in keeping with the language of the PDE and McKean-Vlasov formulations. For certain densities, such as $g\equiv 1+\epsilon$ for some $\epsilon>0$, the particle system \eqref{aligned:Particlealphanocentre} is only defined up to a finite random time $T^{*,N}$ at which $\Lambda^N$ explodes to infinity. Further, it is possible that solutions to \eqref{alternativetomvr} exist only on a finite time interval $[0,T^*)$. To avoid such behaviour we place the following assumptions on the density $g$.
\begin{itemize}
 \item \hypertarget{A1}{(A1)}:\label{A1} The density $g$ is uniformly bounded by some constant $C$.
 \item \hypertarget{A2}{(A2)}:\label{A2} There exists some $x^*$ such that $\int_0^{x^*}(g(x)-1)dx<0$, and $\int_{x^*}^y(g(x)-1)dx\leq 0$ for all $y\geq x^*$.

 \end{itemize}

We will assume that these conditions hold on $g$ throughout this work, except for Section \ref{alternateg} where we discuss alternative conditions on the density $g$.
\begin{remark}
The condition \hyperlink{A1}{(A1)} is a technical condition used to simplify our proofs, while the condition \hyperlink{A2}{(A2)} is used to provide global in time bounds on $\Lambda^N$. This condition is stronger than necessary for solutions of \eqref{alternativetomvr} to exist. However, many cases of interest satisfy \hyperlink{A2}{(A2)}. For example this condition is satisfied by $g$ when $g$ is a CDF, when $g$ is such that solutions of \eqref{alternativetomvr} exhibit asymptotic travelling wave behaviour, and by certain densities which have fluctuations out to infinity, such as $g(x)=1+\cos(x)$ for $x\geq x^*$. Further, while the above conditions are not sharp, weaker conditions are more complicated to state, and it is less clear how to transfer bounds on the mean field system \eqref{alternativetomvr} to the particle system.
\end{remark}
\begin{remark}
It is unclear exactly what a sharp condition for existence of global in time physical solutions to \eqref{alternativetomvr} should be, however we expect the following to hold:
\begin{enumerate} 
\item Suppose that there exists a constant $\vartheta>0$, and a sequence $y_n\uparrow\infty$ such that $\int_0^{y_n}(g(x)-1)dx\leq -\vartheta$. Then a global in time solution to \eqref{alternativetomvr} exists.
\item  Suppose that there exists $x^*$ such that $\int_0^y(g(x)-1)dx>0$ for all $y\geq x^*$. Then a global in time solution to \eqref{alternativetomvr} does not exist.
\end{enumerate}
\end{remark}

Under the conditions  \hyperlink{A1}{(A1)},\hyperlink{A2}{(A2)} on $g$, a linear bound holds on $\Lambda^N$ with high probability.
\begin{lemma}
\label{lemma on bounding}
Let $T>0.$ There exists a constant $z$ such that 
$$\limsup_{N\rightarrow \infty }\mathbb{P}(\Lambda_T^N>z)=0.$$
In particular, $z$ may be taken to be $z=cT+x^*$ for a constant $c$, and $x^*$ as in condition \hyperlink{A2}{(A2)}, independently of $T$.
\end{lemma}

To determine the bound given in Lemma \ref{lemma on bounding}, we consider the behaviour of the system with a deterministic barrier $f_t$. If the number of particles absorbed by a barrier of height $f_t$ would be too few to increase the barrier height to the value $f_t$, then $\Lambda^N_t$ must be bounded above by $f_t$.

We define a random operator $\Gamma_N$, on the space of cadlag functions. For a function in this space, we define $\Gamma_N(f)$ to be the function given by
\begin{align*}
    &X_t^{i,N}(f)=W_0^{i,N}+B_t^{i,N}-f_t,\\
    &\tau_{i}^N(f)=\inf\{t\geq 0: X_t^{i,N}(f)\leq 0\},\\
    &\Gamma_N(f)_t=\frac{1}{N}\sum_{i=1}^\infty\mathbbm{1}_{\{\tau_{i}^N(f)\leq t\}}.
\end{align*}

In subsequent sections, we apply this operator to linear functions which we denote by $l_t^{c,d}=ct+d$ and $l_t^c=l_t^{c,0}.$

\begin{lemma}
\label{Alemmatoboundathing}
    Consider a fixed $T>0,$ and an arbitrary $f$ in the space of continuous, non-decreasing functions on $[0,T].$ If $\Gamma_N(f)_t\leq f_t$ for all $t\in[0,T]$, and $f(0)\geq 0$, then $\Lambda^{N}_t\leq f_t$ for all $t\in[0,T]$. Similarly, if $\Gamma_N(f)_{t}\geq f_t$ for all $t\in[0,T]$, and $f(0)\leq 0,$ then $\Lambda^N_t\geq f_t$ for all $t\in[0,T]$.
\end{lemma}
This follows through minor modifications to the arguments of \cite[Proposition 2.1]{Dembo_2019}. 
\begin{remark}
The upper bound in the above lemma also holds for cadlag functions, as well as continuous functions. The lower bound holds for cadlag functions, as long as the jumps are only downwards.
\end{remark}
With this result, we sketch the proof of Lemma \ref{lemma on bounding}.
\begin{proof}
 We take $f_t=l^{c,x^*}_t$, and shall apply the result of Lemma \ref{Alemmatoboundathing} for this function $f$. The expectation $\mathbb{E}(\Gamma_1(l^{c,x^*})_t)$ can be computed, and is bounded above by $ct+x^*-\epsilon$ for all suitably large $c$, and suitably small $\epsilon$. By applying the Strong Law of Large Numbers, we may determine that the process $\Gamma_N(l^{c,x^*})_t$ converges uniformly in probability to $\mathbb{E}(\Gamma_1(l^{c,x^*})_t)\leq ct+x^*-\epsilon$. Therefore: $$\mathbb{P}(\Gamma_N(l^{c,x^*})_t)<l_t^{c,x^*} \quad \forall t\leq T)\rightarrow 1 \textrm{ as } N\rightarrow\infty.$$ Applying Lemma \ref{Alemmatoboundathing}, it then follows that $\mathbb{P}(\Lambda^{N}_T<cT+x^*)\rightarrow 1$ as $N\rightarrow\infty$. 
 \end{proof}

\begin{remark}
\label{remark on g^n}
The intensity $g$ of the Poisson point process can be taken to vary over the parameter $N\in\mathbb{N}$. In particular, for each integer $N$ we may replace the intensity $g$ by $g^N$, and may obtain particle systems and moving barriers defined in the same manner as the particle system \eqref{aligned:Particlealphanocentre} given above. Unless specified otherwise, our results extend to sequences $g^N$ which are uniformly bounded and converge pointwise to a function $g$ which satisfies conditions \hyperlink{A1}{(A1)} and \hyperlink{A2}{(A2)}.
\end{remark}

\section{General convergence results}
\label{convergencestatementsection}

We view the particle system on $[0,T]$ for arbitrary positive $T$, and shall determine the convergence to a limiting McKean-Vlasov type system. The results extend naturally to $t\in[0,\infty)$ by considering the processes on $[0,T_n]$ for a sequence $T_n\uparrow\infty$.

Since we do not assume any regularity on the density $g$ near zero, to obtain convergence we extend the particle system to a larger set of times by setting 
$$
\bar X^{i,N}_t = 
     \begin{cases}
       W^{i,N}_0 &t\in[-1,0),\\
       X^{i,N}_t & t\in[0,T],\\
       X^{i,N}_T+B_t^{i,N}-B_T^{i,N}&t\in(T,T+1].
     \end{cases}
$$
We also extend $W^{i,N}, B^{i,N}$, and $\Lambda^N$ in a similar manner;
\begin{align*}
&\bar W^{i,N}_t = 
     \begin{cases}
       W^{i,N}_0 &t\in[-1,0),\\
       W^{i,N}_t & t\in[0,T],\\
       W^{i,N}_T+B_t^{i,N}-B_T^{i,N}&t\in(T,T+1],
     \end{cases}\\
&\bar B^{i,N}_t = 
     \begin{cases}
       0 &t\in[-1,0),\\
       B^{i,N}_t & t\in[0,T+1],
     \end{cases}\\
&\bar \Lambda^{N}_t = 
     \begin{cases}
       0 &\qquad \qquad \quad \textrm{ }t\in[-1,0),\\
       \Lambda^{N}_t & \qquad \qquad \quad \textrm{ }t\in[0,T],\\
       \Lambda^N_T& \qquad \qquad \quad \textrm{ } t\in(T,T+1].
     \end{cases}
\end{align*}

We write $\mathcal{C}$ for the space of continuous functions on $[-1,T+1]$ to $\mathbb{R}$, and endow this space with the uniform topology. We view $\bar W^i$ as elements of this space. We also denote by $\tilde D$ the space of cadlag functions on $[-1,T+1]$ to $\mathbb{R}$, which are left continuous at $T+1$, and endow this space with the Skorokhod M1 topology as described in \cite{whitt}. This topology allows for convergence of multiple small jumps to a single jump in the limit, and in general it is necessary to weaken the topology in such a manner so that the processes $\bar \Lambda^N$ converge.
The restriction of the space of cadlag functions is a standard technicality, since such a restriction is used for \cite[Theorem 12.2.2]{whitt} and the consequences thereof. We view $\bar X$ and $\bar \Lambda$ as elements of this space.

The extension to $[-1,T+1]$ was introduced in \cite{minimal} and allows us to establish the tightness of the processes $\bar X^{1},\bar \Lambda^N$ without additional conditions on $g(x)$ for $x$ near 0. The left extension is unnecessary with stricter conditions on the initial density $g$. In particular, if $g<1$ in a neighbourhood of zero, then the argument of \cite[Proposition 2.2]{3} can be modified to give $1/2$-H\"older type bounds on $\Lambda^N_t$ for sufficiently large $N$ and small $t$. Such bounds are sufficient to establish tightness.

From now on we omit the bar $\bar{z}$ on a variable $z$ and identify the processes $X^{i,N},W^{i,N},\Lambda^N$ with their extension on $[-1,T+1]$. We later consider the law of an extended Brownian motion on $[-1,T+1]$, by which we will mean the law of the Brownian motion extended in the manner of $\bar B^1$.

For a Polish space $\mathcal{X}$, we define $M(\mathcal{X})$ to be the space of locally finite Borel measures on $\mathcal{X}$. We endow this space with the vague topology corresponding to the topology on $\mathcal{X}$. For completeness, we recall the definition of the vague topology.
\begin{definition}
\label{defdef}
    For a measure $V$ and a measurable function $f$, write $V(f)$ for $\int f dV$.
    Let $\mathcal{X}$ be a Polish space. A sequence $V^N\subset M(\mathcal{X})$ converges vaguely to $V$ if for all functions $f:\mathcal{X}\rightarrow R$ which are bounded, continuous, and with bounded support, $$V^N(f)\rightarrow V(f)\textrm{ as }N\rightarrow\infty.$$ 
\end{definition}
By \cite[Lemma 4.1]{kallenberg} this is equivalent to considering only functions $f$ which are positive, uniformly continuous, or Lipschitz continuous. We write $\mathcal{B}(\mathcal{X})$ to denote the space of Borel subsets of $\mathcal{X}$, and $\mathcal{\hat B}(\mathcal{X})$
the subset of bounded, Borel sets. Vague convergence of $V^N$ to $V$ is further equivalent to
$$V(B^o)\leq \liminf_{N\rightarrow\infty} V^N(B)\leq \limsup_{N\rightarrow\infty} V^N(B)\leq V(\overline{B})
    \textrm{ for all }B\in\mathcal{\hat B}(\mathcal{X}).$$
For a Polish space $\mathcal{X}$, the vague topology on $M(\mathcal{X})$ is metrisable, and $M(\mathcal{\mathcal{X}})$ is Polish. For further details see \cite[Section 4]{kallenberg}. 

The empirical measure of $W^i$ particles, an element of $M(\mathcal{C})$, and the empirical measure of $X^i$ particles, as an element of $M(\tilde D)$, are defined by
\begin{align}
    &\upsilon^N:=\frac{1}{N}\sum_{i=1}^\infty\delta_{W^{i,N}}=\frac{1}{N}\sum_{i=1}^N\sum_{k=1}^\infty\delta_{W^{i,k}} {}_, \\
    &\mu^N:=\frac{1}{N}\sum_{i=1}^\infty\delta_{X^{i,N}}=\frac{1}{N}\sum_{i=1}^N\sum_{k=1}^\infty\delta_{X^{i,k}} {}_.
\end{align}
We define a map $\tilde h:M(\mathcal{C})\times \tilde D\rightarrow M(\tilde D)$ by:
$$\tilde h(\zeta,l)(A)=\zeta(A+l) \textrm{ for }A\in\mathcal{B}(\tilde D).$$

From the definition of $\mu^N$, it follows that $\mu^N=\tilde h(\upsilon^N,\Lambda^N)$. The function $\tilde h$ is continuous and therefore convergence results for $\mu^N$ may be deduced from those for $\upsilon^N$ and $\Lambda^N$.
\begin{lemma}
\label{continuity of tilde h}
    The function $\tilde h$ is a continuous function from $M(\mathcal{C})\times \tilde D$ viewed with the product topology, to $M(\tilde D)$.
\end{lemma}
The proof is elementary, and details are provided in Section \ref{Continuity of h} of the appendix. 

The main result of this section is the convergence of the empirical measure $\mu^N$ to physical solutions of the McKean-Vlasov system \eqref{alternativetomvr}. 
\begin{theorem}
\label{maintheorem}
Consider a fixed $T>0$. Suppose that $g$ satisfies assumptions \hyperlink{A1}{(A1)} and \hyperlink{A2}{(A2)}.
The processes $(\mu^N,\Lambda^N)$ are tight on $M(\tilde D)\times \tilde D$ viewed with the product topology. Denote by $\Pi^N$ the law of $(\mu^N,\Lambda^N)$. For any subsequential limit $\Pi^\infty$, almost surely under $\Pi^\infty$, $(\mu,\Lambda)$ form a physical solution to the McKean-Vlasov system \eqref{alternativetomvr} up to time $T$.

In particular, for any realisation of the limiting process, we define processes indexed by $x\in\mathbb{R}^+$ via
\begin{align*}
    &X_t^x=x+B_t-\Lambda_t \textrm{ for }B \textrm{ an extended Brownian motion},\\
    &\tau^x=\inf\{t\geq 0:X_t^x\leq 0\}.
\end{align*}
For $\Pi^\infty$ almost every realisation of $(\mu,\Lambda)$, and arbitrary $A\in\mathcal{B}(\tilde D)$, it holds that
$$\mu(A)=\int_0^\infty \mathbb{P}(X_{.}^x\in A)g(x)dx,$$ $$\Lambda_t=\int_0^\infty \mathbb{P}(\tau^x\leq t)g(x)dx.$$ 
Further the discontinuities of $\Lambda$ satisfy the physicality condition \eqref{alternativephysical}. 
\end{theorem}

When the McKean-Vlasov type system \eqref{alternativetomvr} has a unique solution, Theorem \ref{maintheorem} implies convergence of the particle system to the solution of the McKean-Vlasov type system. Some conditions on uniqueness of the limiting process are discussed in Section \ref{uniquenessappendix}.
\begin{corollary}
\label{convergencecor}
Suppose the McKean-Vlasov equation \eqref{alternativetomvr} has a unique physical solution $\Lambda$ on $[0,T]$. Then $\Lambda^N$ converges in probability in the M1 topology to $\Lambda$. Further $\mu^N$ converges in probability under the vague topology to the measure corresponding to the McKean-Vlasov type system \eqref{alternativetomvr}.
\end{corollary}

As a further corollary to our main result, we obtain the propagation of chaos for the empirical measures of particles.
For $i=1,...,N$ we define $\upsilon^{i,N},\mu^{i,N}$ by
\begin{align*}
&\upsilon^{i,N}:=\sum_{k=1}^\infty\delta_{W^{i,k}},\\
&\mu^{i,N}:=\sum_{k=1}^\infty\delta_{X^{i,k}}.
\end{align*}
Above, we have used the relabelling of $(W^{i,k})_{i=1,...,N,\space k\in\mathbb{N}}$ given in \eqref{aligned:Particlewarticle}. Consider integers $l$, and $k_1<k_2<...<k_l$. We consider $(\mu^{k_1,N},...,\mu^{k_l,N})$ to be an element of $ M(\tilde D)^l$ endowed with the product topology.

\begin{corollary}
Suppose that the McKean-Vlasov system \eqref{alternativetomvr} has a unique solution $\Lambda$. Then $(\mu^{k_1,N},...,\mu^{k_l,N})$ converges in distribution to $l$ independent copies of $\tilde h(\upsilon^{1},\Lambda)$.
\end{corollary}

\begin{remark}
By definition $\tilde h(\upsilon^1,\Lambda)$ is the measure given by placing a Dirac measure at the position of each particle in the McKean-Vlasov system \eqref{aligned:MVR}, where we set $W_0^{1,i}=X_{0^-}^{i}$. This result thus shows that the empirical measures of the $N$ particle systems converge to the empirical measure of the mean-field particle system.
\end{remark}

\begin{proof}
    Consider an arbitrary continuous bounded function $f:M(\tilde D)^k\rightarrow \mathbb{R}$. It is easily seen that $\mu^{k_j,N}=\tilde h(\upsilon^{k_j},\Lambda^N)$ for each $j=1,...,l$. From Corollary \ref{convergencecor} $\Lambda^N\rightarrow \Lambda$ in distribution as $N\rightarrow\infty$. Since $f$ is bounded, continuous, and $\tilde h$ is continuous, then by the Dominated Convergence Theorem:
    $$E(f(\tilde h(\upsilon^{k_1},\Lambda^N)),...,\tilde h(\upsilon^{k_l},\Lambda^N))\rightarrow E(f(\tilde h(\upsilon^{k_1},\Lambda),...,\tilde h(\upsilon^{k_l},\Lambda)))\textrm{ as }N\rightarrow\infty.$$
    As $(\upsilon^{k_j})_{j=1}^l$ are independent copies of $\upsilon^1$, and as $f$ was arbitrary we conclude the convergence in distribution.
\end{proof}

\subsection{Sketch of the proof of Theorem~\ref{maintheorem}}
\label{sketchsection}
We sketch the proof of Theorem \ref{maintheorem} here. The details will be provided in Section \ref{proofsection}.

We first check that $\Lambda^N$ is tight. Given the bound from Lemma \ref{lemma on bounding}, this follows directly from the definition of a pre-compact set in the Skorokhod M1 topology.  Therefore, if the result of Lemma \ref{lemma on bounding} holds for some density $g$ (which need not satisfy condition \hyperlink{A2}{(A2)}), $\Lambda^N$ will then be tight and we may proceed as in the rest of this proof.

 We then define a continuous map $\theta:\mathbb{R}\times \mathcal{C}\rightarrow\mathcal{C}$ by:
$$\theta(x,y)_t=x+y_t.$$ We denote by $\mathcal{L}(B)$ the law of the extended Brownian motion on $\mathcal{C}$. We further denote the measure of density $g$ as $G$. The random measure
$\upsilon^N$ converges to $\theta_{*}(G\times\mathcal{L}(B))$ vaguely, in distribution.

Since $\mu^N=\tilde h(\upsilon^N,\Lambda^N)$ and $\tilde h$ is continuous, it then immediately follows that $\mu^N$ is tight.

Having established the tightness of the processes $\nu^N,\mu^N$ and $\Lambda^N$ we work on a subsequence which converges in distribution. We then apply the Skorokhod Representation Theorem to obtain sequences $\hat \upsilon^N,\hat \mu^N,\hat \Lambda^N$ which converge almost surely to $\hat \upsilon,\hat \mu,\hat \Lambda$. 

To obtain properties of the limiting measures, $\hat \mu,\hat \nu$, we work with finite measures by restricting these to suitable sets of finite measure. To do this we fix arbitrary $K>0$ and consider the behaviour of particles which are initially below $K$. We define for $K>0$ the set of such paths, $S_K$, by:

$$S_K={\{x\in \tilde{D}: x_{-1}\leq K\}}_,$$ and for a measure $\varrho$ on $\tilde D$ (or $\mathcal{C}$) define 
$$\varrho_{S_K}(A)=\varrho(S_K\cap A) \quad \forall A \in \mathcal{B}(\tilde D) \textrm { (or } \mathcal{B}(\mathcal{C})).$$
The measures $\hat\mu^N_{S_k},\hat \nu^{N}_{S_K}$ then converge weakly to $\hat \mu_{S_K},\hat \nu_{S_k}$.

Since $\hat \mu_{S_K}$ and $\hat \nu_{S_K}$ are finite measures which converge weakly, we are then able to apply the results of \cite{delarue2015particle} to obtain properties of the limiting measures. In particular, we determine that $\hat\mu(\inf_{s\leq t}x_s\leq 0)=\hat \Lambda_t$, where we use $x_t$ to denote the canonical process on $\tilde D$. To do this, we consider arbitrary $K>0$, and argue as in \cite[Lemma 5.9]{delarue2015particle} that $$\hat \mu^N_{S_K}(\inf_{s\leq t}x_s\leq 0)\rightarrow \hat \mu_{S_K}(\inf_{s\leq t}x_s\leq 0)\textrm{ as }N\rightarrow\infty.$$
By considering the number of particles started above $K$, which reach $0$ by time $T$ we also obtain that 
$$\hat \mu_{S_K}(\inf_{s\leq t}x_s\leq 0)\geq \hat \mu(\inf_{s\leq t}x_s\leq 0)-\tilde Ce^{-\frac{\tilde CK^2}{T}}$$
for some constant $\tilde C$. Taking a limit over $K\in\mathbb{N}$ appropriately, we are then able to establish that 
$$\hat \Lambda_t=\lim_{N\rightarrow\infty} \hat \mu^N(\inf_{s\leq t}x_s\leq 0)=\hat \mu(\inf_{s\leq t}x_s\leq 0).$$

From this result, and the definition of the maps $\theta$ and $\tilde h$ we then deduce that $(\hat \mu,\hat \Lambda)$ form a solution to the McKean-Vlasov type equation \eqref{alternativetomvr}.

Finally, we determine that $\hat \Lambda$ satisfies the physical jump condition. We do this via a simple modification of the proof of \cite[Theorem 6.4]{minimal}, by again considering $\hat \mu_{S_K}$ and applying an exponential bound on the number of particles which start above $K$, and reach $0$ by time $T$.

\section{Simplifications under weak feedback}
\label{uniquenesssection}

In this section we consider a ``weak-feedback" regime, which is defined analogously to the regime considered in \cite{5}. We define the ``weak feedback" regime to be one in which any initial density satisfies the following condition.
\begin{itemize}
    \item \hypertarget{W}{(W)}:  $g(x)\leq 1$ for all $x$, and there exists $\epsilon_g>0$ such that $\sup_{x\leq \epsilon_g}g(x)<1$. 
\end{itemize}

Under the condition \hyperlink{W}{(W)} on $g$, from an analogous argument to that of \cite[Lemma 2.1]{2} we may determine that for any time $t$, the measure of non-absorbed particles $\nu_t$, possesses a density with respect to the Lebesgue measure, and which satisfies the condition \hyperlink{W}{(W)}.
\begin{proposition}
\label{6.1}
    If $g$ is bounded above by a constant $\tilde C$, then $\nu_t$ possesses a density, $\tilde V(t,x)$, which is bounded by $\tilde C\Phi(\frac{\Lambda_t+x}{\sqrt{t}})$. In particular for $g$ satisfying \hyperlink{W}{(W)}, $\tilde C$ may be taken to be 1, and $\tilde V(t,x)$ also satisfies condition \hyperlink{W}{(W)}.
\end{proposition}
From a simple adaptation of \cite[Proposition 2.2]{3} to this setting, it may be seen that the function $\Lambda$ is locally $1/2-$H\"older continuous in time.
\begin{corollary}
\label{holdercont}
   If $g$ satisfies condition \hyperlink{W}{(W)}, it follows that for any $s\geq 0$, there exists a constant $C_s$, such that for $ s\leq t\leq T$:
$$\Lambda_t-\Lambda_s\leq C_s \sqrt{t-s}.$$ 
\end{corollary}

Further, from a simple modification of the proof of \cite[Lemma 2.1, Theorem 2.2]{5}, it can be seen that solutions to \eqref{alternativetomvr} are unique.
\begin{proposition}
\label{minimal}
   If $g$ satisfies condition \hyperlink{W}{(W)}, then solutions to the McKean-Vlasov equation \eqref{alternativetomvr} are unique.
\end{proposition}

Since the solution to \eqref{alternativetomvr} with an initial density $g$ which satisfies \hyperlink{W}{(W)} is unique and continuous, applying Corollary \ref{convergencecor} and \cite[Theorem 12.4.2]{whitt}, uniform convergence of $\Lambda^N$ then follows.
\begin{corollary}
\label{thhtth}
    Suppose $g$ satisfies the condition \hyperlink{W}{(W)}. Then $\Lambda^N$ converges under the uniform topology, in probability, to $\Lambda$, the unique solution of \eqref{alternativetomvr}.
\end{corollary}

The above corollary may also be proven through an elementary argument, which we give in Section \ref{weakproofsection}. This enables quantitative bounds on $|\Lambda^N-\Lambda|$ to be established.

\begin{theorem}
\label{bdtheorem}
    Suppose $g$ satisfies condition \hyperlink{W}{(W)}. Then for each $T\geq 0$ there exists a constant $K$, such that for $0<y\leq \sqrt{N}\epsilon_g/2$
    $$\mathbb{P}(\sup_{t\leq T}\sqrt{N}(\Lambda^N_t-\Lambda_t)\geq y)\leq 2e^{-Ky^2/(1+y/\sqrt{N})}.$$
\end{theorem}

As a corollary, we obtain the result of Theorem \ref{conjtheorem}, which we restate here for convenience.
 
\begin{corollary}
    Consider the particle system for $\xi$ discussed in the introduction, but with particles placed via an initial density $\tilde g(x)$. Suppose that $\tilde g(Nx)\rightarrow a<1$ as $N\rightarrow\infty$.
    It follows that $\xi_t/\sqrt{t}\rightarrow K_a$ in probability as $t\rightarrow\infty$. Here $K_a$ is the parameter for the similarity solution of the supercooled Stefan problem, given by equation \eqref{kalpha}.
    Further, if $\tilde g\equiv a<1$, then for $y\leq t^{1/4}$:
    $$\mathbb{P}\left(\left|\xi_t/\sqrt{t}-K_a\right|\geq yt^{-1/4}\right)\leq 2e^{-\frac{(1-a)^2y^2}{(3K_{a}+3)}}.$$
\end{corollary}

\begin{proof}
Since $\xi_t=Na\Lambda^N_{t/N^2a^2}$, it is sufficient to determine the convergence of $\Lambda^N_1$ for the particle system \eqref{aligned:Particlealpha}, with particles placed with intensity $Ng^N(x)=Na \tilde g(Na x)$. By assumption, $g^N\rightarrow a \mathbbm{1}_{[0,\infty)}$ pointwise. By Remark \ref{remark on g^n}, we then need only consider the convergence of the particle system with initial density $g(x)=a \mathbbm{1}_{[0,\infty)}(x)$.
Since $g$ satisfies condition \hyperlink{W}{(W)}, by Lemma \ref{minimal} the McKean-Vlasov equation \eqref{alternativetomvr} has a unique solution. The similarity solution of the PDE \eqref{pde} gives a solution to \eqref{alternativetomvr}. It therefore must agree with the unique solution to \eqref{alternativetomvr}.
The first result now follows immediately from Corollary \ref{thhtth}. The second result follows in an analogous manner from Theorem~\ref{bdtheorem}, by again using the rescaling $\xi_t={Na}\Lambda^N_{t/a^2N^2}$, and identifying the constants $K$ and $\epsilon_g$ from the proof of Theorem~\ref{bdtheorem}.
\end{proof} 

\subsection{Elementary proof of quantitative bounds}
\label{weakproofsection}

We may prove Theorem \ref{bdtheorem} via an elementary argument. We first define an operator $\Gamma$ on cadlag functions $f$ by:
\begin{align}
    &X_t^{i}(f)=W_0^i+B_t^i-f_t,\nonumber\\
    &\tau_i(f)=\inf\{t\geq 0:X_t^{i}(f)\leq 0\},\label{gammadefn}\\
    &\Gamma(f)=\sum_{i=1}^\infty \mathbb{P}(\tau_i(f)\leq t).\nonumber
\end{align}
This is identical to setting $\Gamma(f)_t=\mathbb{E}(\Gamma_1(f)_t)$, where $\Gamma_N$ is the operator defined above Lemma \ref{Alemmatoboundathing}. 
An analogous operator was first defined in \cite{DIRT1} and used in \cite{2, NS19, minimal}.

For an arbitrary time $t\leq T$, and arbitrary $\epsilon< \epsilon_g/2$ we bound \begin{align*}
&\Gamma(\Lambda+\epsilon)_t-\Gamma(\Lambda)_t\\&=\int_0^\infty\left( \mathbb{P}\left(\inf_{s\leq t}(x+B_s-\Lambda_s)\leq \epsilon\right)-\mathbb{P}\left(\inf_{s\leq t}(x+B_s-\Lambda_s)\leq 0\right)\right)g(x)dx,\\
&=\mathbb{E}\left(\int_{-\inf_{s\leq t}(B_s-\Lambda_s)}^{-\inf_{s\leq t}(B_s-\Lambda_s)+\epsilon}g(x)dx\right),\\
&\leq \epsilon \mathbb{E}\left(\sup_{}\{g(x) \space\textrm{ for }  x\leq -\inf_{s\leq T}(B_s-\Lambda_s)+\epsilon_g/2 \}\right),\\
&\leq(1-\kappa)\epsilon.
\end{align*}

Above, $\kappa$ is a constant which is dependent on $g$ and $T$.
Since $t\leq T$ was arbitrary, this bound holds for the supremum over $t\leq T$. Further since $\Gamma(\Lambda)=\Lambda$, we determine that
\begin{align}
\label{boundviakappa}
    \sup_{t\leq T}(\Gamma(\Lambda+\epsilon)_t-\Lambda_t-\epsilon)\leq \kappa \epsilon.
\end{align}

An identical argument establishes also that
\begin{align}\label{boundviakappalower}\inf_{t\leq T}(\Lambda_t-\epsilon-\Gamma(\Lambda-\epsilon)_t)\leq -\kappa\epsilon.\end{align}

Recall the notation $\Gamma_N$ from above Lemma \ref{Alemmatoboundathing}. Consider the event $A^{N,+}_{\epsilon}$ given by
$$A^{N,+}_\epsilon=\left\{\sup_{t\leq T}\left(\Gamma_N(\Lambda+\epsilon)_t-\Gamma(\Lambda+\epsilon)_t\right)\leq \kappa \epsilon\right\}.$$
Applying the bound of \eqref{boundviakappa}, on this event it follows that $$\Gamma_N(\Lambda+\epsilon)_t\leq \Gamma(\Lambda_t+\epsilon)+\kappa\epsilon\leq \Lambda_t+\epsilon \textrm{ for all }t\leq T.$$ By Lemma 
\ref{Alemmatoboundathing} it then follows that $\Lambda^N_t\leq \Lambda_t+\epsilon$ for $t\leq T$. 

This argument can be repeated for $\Lambda-\epsilon$ by considering the event
$$A^{N,-}_\epsilon=\left\{\sup_{t\leq T}\left(\Gamma(\Lambda-\epsilon)_t-\Gamma_N(\Lambda-\epsilon)_t) \right)\leq \kappa\epsilon\right\}.$$
Applying the bound from \eqref{boundviakappalower}, on this event it holds that $\Lambda^N_t\geq \Lambda_t-\epsilon$ for $t\leq T$.

We can therefore obtain bounds on the probability that $|\Lambda_t^N-\Lambda_t|>\epsilon$, by bounding the probability of the sets $A_{\epsilon}^{N,+},$ $A_{\epsilon}^{N,-}$.

\begin{lemma} \label{lem:aplus}
For $0<y\leq \sqrt{N}\epsilon_g/2$,
    $$\mathbb{P}(A^{N,+}_{y/\sqrt{N}})\geq 1-e^{\frac{-\kappa ^2y^2}{\Gamma(\Lambda+1)_T+\kappa y/\sqrt{N}}}. $$
\end{lemma}

\begin{proof}
Consider an arbitrary deterministic continuous function $\tilde \Lambda$. 
We define a filtration $(\mathcal{F}_{t})_{t\leq T}$ by $\mathcal{F}_t=\sigma(\Gamma_N(\tilde \Lambda)_s,\space s\leq t)$. We may check that $\Gamma_N(\tilde \Lambda)_t-\Gamma(\tilde \Lambda)_t$ is a martingale; 
\begin{align*}
    &\mathbb{E}\left(\Gamma_N(\tilde \Lambda)_t-\Gamma(\tilde \Lambda)_t\mid \mathcal{F}_s\right)_,\\
    &=\mathbb{E}\left(\frac{1}{N}\sum_{i=1}^\infty\mathbbm{1}_{\{s<\tau_i^N(\tilde \Lambda)\leq t\}}\mid \mathcal{F}_s\right)+\Gamma_N(\tilde \Lambda)_s-{\Gamma(\tilde\Lambda)_t}_,\\
    &=\Gamma(\tilde\Lambda)_t-\Gamma(\tilde\Lambda)_s+\Gamma_N(\tilde\Lambda)_s-{\Gamma(\tilde\Lambda)_t}_,\\
    &=\Gamma_N(\tilde\Lambda)_s-\Gamma(\tilde\Lambda)_s.
\end{align*}

By standard results on the thinning of Poisson processes, the sum $\sum_{i=1}^\infty\mathbbm{1}_{\{s<\tau_i(\tilde \Lambda)\leq t\}}$ is a Poisson random variable of mean $N(\Gamma(\tilde\Lambda)_t-\Gamma(\tilde\Lambda)_s)$, which is independent of $\mathcal{F}_s$. From this result, the third line above follows. Therefore, $\Gamma_N(\tilde \Lambda)_t-\Gamma(\tilde \Lambda)_t$ is a martingale.

For an arbitrary $\lambda>0$, we apply Doob's inequality to the submartingale \\$\exp\left(\lambda\left(\Gamma_N(\tilde \Lambda)_t-\Gamma(\tilde \Lambda)_t\right)\right)$. We obtain that, for any $\lambda>0$,
$$\mathbb{P}(\sup_{t\leq T}(\Gamma_N(\tilde \Lambda)_t-\Gamma(\tilde \Lambda)_t)\geq y)\leq \mathbb{E}(e^{\lambda(\Gamma_N(\tilde \Lambda)_T-\Gamma(\tilde \Lambda)_T-y)}).$$

Recall that $N\Gamma_N(\tilde\Lambda)_T$ is a Poisson random variable of mean $N\Gamma(\tilde \Lambda)_T$. Using the same arguments as for deriving the Chernoff bound for Poisson random variables, we determine that
$$\mathbb{P}(\sup_{t\leq T}\sqrt{N}(\Gamma_N(\tilde \Lambda)_t-\Gamma(\tilde \Lambda)_t)\geq x)\leq e^{-\frac{x^2}{(2\Gamma(\tilde \Lambda)_T+x/\sqrt{N})}}.$$

Taking $\tilde \Lambda=\Lambda+y/\sqrt{N}$, we therefore determine that
$$\mathbb{P}(A^{N,+}_{y/\sqrt{N}})\geq 1-e^{-\frac{\kappa^2y^2}{(2\Gamma(\Lambda+1)_T+ \kappa y/\sqrt{N})}}.$$
\end{proof}

Using identical arguments, we can also control the probability of the corresponding lower bound event.

\begin{lemma}\label{lem:aminus}
For $0<y\leq \sqrt{N}\epsilon_g/2$,
     $$\mathbb{P}(A^{N,-}_{y/\sqrt{N}})\geq 1-e^{\frac{-\kappa^2 y^2}{\Gamma(\Lambda+1)_T+\kappa y/\sqrt{N}}}. $$
\end{lemma}

Combining the results of Lemmas~\ref{lem:aplus} and~\ref{lem:aminus}, we may then prove Theorem~\ref{bdtheorem}.

\begin{proof}[Proof of Theorem~\ref{bdtheorem}.]
    From the remarks above Lemma~\ref{lem:aplus}, it holds that for arbitrary positive $y$,
    $$\left\{\sup_{t\leq T}(\Lambda^N_t-\Lambda_t)\leq\frac{y}{\sqrt{N}}\right\}\subset A^{N,+}_{\frac{y}{\sqrt{N}}}$$
    and 
    $$\left\{-\frac{y}{\sqrt{N}}\leq \inf_{t\leq T}(\Lambda^N_t-\Lambda_t)\right\}\subset A^{N,-}_{\frac{y}{\sqrt{N}}}.$$

    Consequently:
    $$\mathbb{P}\left(\sup_{t\leq T}\sqrt{N}|\Lambda^N_t-\Lambda_t|>y\right)\leq \mathbb{P}\left(\left(A^{N,+}_{\frac{y}{\sqrt{N}}}\right)^c\right)+\mathbb{P}\left(\left(A^{N,-}_{\frac{y}{\sqrt{N}}}\right)^c\right).$$

    Applying the results of Lemmas~\ref{lem:aplus} and~\ref{lem:aminus}, we then determine that
    $$\mathbb{P}\left(\sup_{t\leq T}\sqrt{N}|\Lambda^N_t-\Lambda_t|>y\right)\leq 2e^{\frac{-\kappa^2 y^2}{\Gamma(\Lambda+1)_T+\kappa y/\sqrt{N}}}.$$
\end{proof}

\begin{remark}
The martingale argument used above requires particles to be initially placed via a Poisson point process. However, for $g\in L^1(\mathbb{R}^+)$, in previous works (e.g.  \cite{delarue2015particle}), exactly $N$ particles are initially placed. The initial position of these particles are independent and identically distributed random variables which have a distribution corresponding to the density $g$. In this case, the results of this section follow (with possibly different constants) from an application of the Dvoretzky–Kiefer–Wolfowitz inequality.
\end{remark}
\subsection{Monotone initial density}

For any CDF $F$, and any constant $\alpha<1$ the function $\alpha F$ satisfies the condition \hyperlink{W}{(W)}. The supercooled Stefan problem with initial data $\alpha F$ corresponds to the model studied in \cite{BHJ2025}. For this model, the authors are able to study the case $\alpha>1$. Due to this correspondence, it is natural to extend some results of the weak feedback regime to this case.

\begin{proposition}
For any $\alpha\geq 0$, and any $F$ a CDF such that $\alpha F(0)<1$, there exists a maximal interval $[0,T^*)$ on which solutions to \eqref{alternativetomvr} with initial density $\alpha F$ exist. The solution in this case is unique. If $\alpha \leq 1$, then $T^*=\infty$, otherwise if $\alpha>1$ then $T^*<\infty$.
\end{proposition}

\begin{proof}
    We need only consider $\alpha>1$, otherwise $g
    \alpha F$ satisfies condition \hyperlink{W}{(W)}. For sufficiently small $x^*>0$, $T>0$, and $\epsilon>0$, it is simple to determine that $\Gamma(x^*)_t\leq x^*-\epsilon$ for $t\leq T$. The argument of Lemma \ref{lemma on bounding} can be applied to determine that $\Lambda^N_T\leq x^*$ for all large enough $N$. With this bound, the arguments of Theorem \ref{maintheorem} may be repeated to determine existence of a solution to \eqref{alternativetomvr} on $[0,T)$. 

    Any solution to \eqref{alternativetomvr} with initial density $\alpha F$, possesses a density at time $t$, $\tilde V(t,x)$, which is monotonically non-decreasing. The arguments of \cite[Proposition 5.2]{3} may then be applied to determine that the solution to \eqref{alternativetomvr} is unique in this case. Further, the monotonicity of the density, combined with the physical jump condition \eqref{alternativephysical} implies that if there were to be a discontinuity in $\Lambda$, it must hold that $\Delta \Lambda_t=\infty$. Therefore, $\Lambda$ must remain continuous at any time at which the solution to \eqref{alternativetomvr} exists. This further implies that $\limsup_{x\downarrow 0}\tilde V(t,x)<1$. Repeating the existence argument, we determine that there exists a unique solution to \eqref{alternativetomvr} up to the time $$T^*:=\inf\{t\geq 0:\lim_{x\downarrow 0}\tilde V(s^-,x)\geq 1\}\wedge \inf\{t\geq 0:\lim_{s\uparrow t}\Lambda_s=\infty\}.$$

    If $T^*=\infty$, then it must hold that $\Lambda$ is continuous on $[0,\infty)$. It is simple then to argue that 
    $$\int_0^\infty(1-\alpha F(x))e^{-\lambda x}dx=\frac{\lambda }{2}\int_{0}^\infty e^{-\lambda\Lambda_s-\frac{\lambda^2s}{2}}ds.$$
    Details are provided in Section \ref{itoformula} in the appendix. For a sufficiently small $\lambda$, the left hand side is negative, while the right hand side is positive. This cannot occur, and so $T^*$ must be finite.
\end{proof}

\begin{corollary}
    Suppose that there exists $x^*$ such that $\int_0^{x^*}g(x)dx<x^*$. Then, there exists a positive time $T$ for which a solution to \eqref{alternativetomvr} with initial density $g$ exists. If solutions to \eqref{alternativetomvr} with initial density $g$ are unique, the maximal time of existence for such a solution is given by 
    $$T^*:=\inf\{t\geq 0:\int_0^x\tilde V(t^-,y)dy\geq x\textrm{ for all }x>0\}\wedge \inf\{t\geq 0:\lim_{s\uparrow t}\Lambda_s=\infty\}.$$
\end{corollary}

For a generic initial density $g$, the particle system may also exist only up to a (possibly) finite time $T^{*,N}$,  
$$T^{*,N}:=\inf\{t\geq 0:\inf\{x:\nu^N_{t^-}([0,x])<x\}=\infty\}.$$ We define 
$T^{N}=T^{*,N}\wedge (T^*+1)$, and $T^{N,M}=\inf\{t\geq 0:\Lambda^N_t\geq M\}\wedge T^N$.

\begin{proposition}
\label{explosiontime}
Suppose that the McKean-Vlasov equation \eqref{alternativetomvr} has a unique solution, with maximal time of existence $T^*$. For any $\epsilon>0$
\begin{align*}
\lim_{M\uparrow\infty}&\liminf_{N\rightarrow\infty}\mathbb{P}(T^{N,M}\leq T^*+\epsilon)=1,\\&\liminf_{N\rightarrow\infty}\mathbb{P}(T^{N,*}\geq T^*-\epsilon)=1.
    \end{align*}
\end{proposition}

This result follows from elementary calculations and is given in Section \ref{finiteexplosionsec} in the appendix.
\begin{remark}
We cannot expect that in general $T^{*,N}\leq T^*$ with high probability. For example, if $g\equiv \mathbbm{1}_{[0,\infty)}$, then $T^{*,N}=\infty$ almost surely, while $T^*=0$. 
\end{remark}

The arguments of Section \ref{weakproofsection} extend to the case of $g=\alpha F$ for $\alpha>1$ and allow us to determine a $\sqrt{N}$ rate of convergence for the particle system up to times bounded away from the explosion time.

\begin{proposition}\label{CDFalpha}
    Let $F$ be a CDF, and $\alpha$ a positive constant such that $\alpha F(0)<1$. For any $T<T^*$, the process $\sup_{s\leq T}\sqrt{N}|\Lambda^N_s-\Lambda_s|$ is stochastically bounded. 
\end{proposition}

This follows from a simple adaptation of the proof of Theorem \ref{bdtheorem}, noting that $\sup_{t\leq T}\sup_{x\leq x^*}\tilde V(t,x)<1$ for some small $x^*$, and that very few particles which are outside of $[0,x^*]$ at a time $t$ are absorbed by the time $t+\Delta$, for small $\Delta$. We give the proof in Section \ref{finiteexplosionsec} of the appendix.

A $\sqrt{N}$ rate of convergence is given in \cite{BHJ2025} for the case $\alpha<1$. Given the correspondence between our setting and the setting of \cite{BHJ2025}, it is natural to expect a $\sqrt{N}$ rate of convergence for the model studied in \cite{BHJ2025}, up to times bounded away from an explosion time, and this is indeed the case.

As in \cite{BHJ2025}, for any $\alpha\geq 0$ we define the following processes
\begin{align}
&X_t^{i,N}=X_0^{i,N}+B_t^{i,N}-\alpha \bar L_t^N+L_t^{i,N},\nonumber\\
&L_t^{i,N}=\sup_{s\leq t}(X_0^i+B_s^{i,N}-\alpha \bar L_s^i)_-,\label{BHJparticles}\\
&\bar L_t^N=\frac{1}{N}\sum_{i=1}^NL_t^{i,N}.\nonumber
\end{align}

The random variables $X_0^i$ are i.i.d. with CDF $F$, and the random variables $B^i$ are i.i.d. standard Brownian motions, independent of the collection $(X_0^i)_{i=1,...,N}$.
This system is well defined up to some random explosion time $T^{*,N}$, for details see \cite{BHJ2025}. We also define a process $X_t$,
\begin{align}
&X_t=X_0+B_t-\alpha l_t+L_t,\nonumber\\
&L_t=\sup_{s\leq t}(X_0+B_s-\alpha l_s)_-,\label{BHJMV}\\
&l_t=\mathbb{E}(L_t).\nonumber
\end{align}
The processes $(X_0,B)$ are distributed identically to $(X_0^1,B^1)$ above. This system has a unique solution up to some explosion time $T^*$, which may be finite or infinite.

\begin{proposition}\label{BHJextension}
    Let $T^*$ to be the maximal time of existence of the solution to \eqref{BHJMV}. For any $T<T^*,$ the process $\sup_{t\leq T}\sqrt{N}|\bar L^N_t-l_t|$ is stochastically bounded.
\end{proposition}

These results follow from a simple adaptation of the arguments in \cite[Theorem 1.14]{BHJ2025}, by noting that, similarly to Proposition \ref{CDFalpha}, for suitable $x^*$, and small $\Delta$, the contribution to $\bar L_{t+\Delta}^N$ by particles which are not in $[0,x^*]$ at time $t$ is small. The proof is given in Section \ref{finiteexplosionsec} of the appendix.

\section{A critical case}
\label{critsec}
The methodology of Section \ref{weakproofsection}, and the result of Theorem \ref{conjtheorem} may be extended to the case of $a=1$, where points are initially placed via a Poisson process of unit intensity. In this case, we determine a scaling limit for $\xi_t/t^{2/3}$, rather than $\xi_t/\sqrt{t}$. We restate Theorem \ref{theoremoncrit} with the previously defined notation, so it is clear what we mean by convergence in distribution.

Fix an arbitrary $T>0$. We view $(\xi_{N^2t})_{t\leq T}$ as a cadlag process. We then extend this in an identical manner as $\Lambda^N$, to an element of $\tilde D$.  
We also define a process $R_t$,
    $$R_t:=\inf\left\{x:\int_0^x2(B(s))_+ds>t\right\}.$$
    Here, $B(s)$ is a standard Brownian motion. We extend $R_t$, for $t\leq T$, to an element of $\tilde D$ in the same way as $\xi^N$.
\begin{theorem}
\label{crittheorem}
    Consider the model \eqref{xidefnequation} with $a=1$. The processes $(\frac{\xi_{N^2t}}{N^{4/3}})_{t\leq T}$ converge in distribution as elements of $\tilde D$ endowed with the Skorokhod M1 topology, to the process $(R_t)_{t\leq T}$.
\end{theorem}

The scaling and the limit in this result are identical to that of \cite[Theorem 1.3]{Dembo_2019}. Given that our model may be seen as a continuous space analogue to the model studied in \cite{Dembo_2019}, the required scaling, and the form of the limit are not surprising.
Further, many of the steps of our proof are analogous to arguments given in \cite{Dembo_2019}, though we differ in how we approach the limiting process. In \cite{Dembo_2019}, the authors work directly with the particle system throughout, using a "discrete PDE" to provide bounds on $\xi_t$ via processes of the form $\inf\{x:B^N(t+x)>y\}$ where $B^N$ is a normalised, rescaled version of the initial particle CDF, and which converges to a Brownian motion. We meanwhile first use martingale arguments to determine the tightness and non-degeneracy of $R_t^N=\xi_{N^2t}/N^{4/3}$. We then determine the behaviour of the limiting process $R_t$ in small time intervals.

The proof of Theorem \ref{crittheorem} is simpler than that in \cite{7} since direct computations using the hitting time distribution of Brownian particles to a line may be carried out, as well as martingale arguments applied due to our more restrictive setting that the initial particle configuration is via a Poisson point process.

Below, we highlight some areas of the proof which motivate the rescaling, and the form of the limit. 

We first motivate the scaling. Suppose that $\xi_{N^2t}/N^{1+\alpha}$ has a non-degenerate limit for some $\alpha$. We apply the rescaling given in Section \ref{reformulationsec} to replace $\xi_{N^2t}/N^{1+\alpha}$ by $\Lambda^N_t/N^{\alpha}$.

If $\Lambda^N_t/N^{\alpha}$ has a non-degenerate limit $R_t$, then for some large constants $c$ and $d$ it should hold that $\frac{t}{c}-d\leq R_t\leq ct+d$ with high probability. It is sensible then to compare $\Lambda^N$ to linear functions with speeds of order $N^\alpha$.

Recall the definition of $l^{c,d}$ from below Lemma \ref{Alemmatoboundathing}. To determine that $\Lambda^N_t\geq N^{\alpha}(ct-d)$, by Lemma \ref{Alemmatoboundathing} it is sufficient to determine that $\Gamma_N(N^{\alpha}l^{c,-d})_t\geq N^{\alpha}({ct-d})$. 
We may write 
\begin{align}
&\Gamma_N(N^{\alpha}l^{c,-d})_t-(ct-d)\nonumber\\
&=(\Gamma_N(N^{\alpha}l^{c,-d})_t-\Gamma(N^{\alpha}l^{c,-d})_t)\label{e1}\\
&+(\Gamma(N^{\alpha}l^{c,-d})_t- N^{\alpha}({ct-d})).\nonumber\end{align} 
If $\Lambda^N_t/N^{\alpha}$ is the correct scaling, this should be positive with large probability for appropriate constants $c,-d$. To bound the second term in the sum, we require bounds on the differences between $\Gamma(l^{c,-d})_t$ and $ct-d$.

\begin{lemma}
\label{critlem}
    For any $cd\geq 1$, it holds that \begin{align}&\inf_{t\leq 0}(\Gamma(l^{c,-d})_t-ct+d)=\frac{1}{2c},\label{critlemcd}\\
    & \sup_{t\geq 0}(\Gamma(l^{c,-d})_t-ct+d)= d,\\
    &\sup_{t\geq 0}(\Gamma(l^c)_t-ct)\leq \frac{1}{2c},\label{critlemc}\\
    &\inf_{t\geq 0}(\Gamma(l^c)_t-ct)=0.\end{align}

    Further, there exists a constant $\tilde C$, such that, for any $c,d\geq 0$,
    $$\left|\Gamma(l^{c,-d})_t-(ct+d)-\frac{1}{2c}\right|\leq \tilde Ce^{-\frac{c}{2}}\textrm{ for }t\geq \frac{1}{c}.$$
\end{lemma}

Since $\Gamma(l^{c,-d})_t$ can be calculated explicitly, these results follow from elementary calculations. We give the proof in Section \ref{hitting a line} of the appendix.

From the bound of Lemma \ref{critlem} equation \eqref{critlemcd}, the second term on the right hand side of \eqref{e1} is bounded below by $1/(2N^{\alpha}{c})$ (as long as $d\geq 1/cN^{\alpha}$). 
The process $\Gamma_N(N^{\alpha}l^{c,-d})_t$ is 1/N multiplying a Poisson random variable of mean $N\Gamma(N^{\alpha }l^{c,-d})_t\approx N^{1+\alpha}(ct-d)$. Therefore, from a Poisson tail bound as in the proof of Theorem \ref{bdtheorem}, it can be shown that
$$\mathbb{P}\left(\inf_{t\leq T}N\left(\Gamma_N(N^{\alpha}l^{c,-d})_t-\Gamma(N^{\alpha}l^{c,-d})_t\right)\leq -N/cN^{\alpha}\right)\leq 2e^{-\frac{N^{2-2\alpha}/c^2}{3N^{1+\alpha}cT}}\approx2e^{-\frac{N^{1-3\alpha}}{3c^3T}}$$
for sufficiently large $N$.
For the limit of this bound to be non-degenerate, we require $\alpha=1/3$.

To upper bound $\Lambda^N$, we first define
$$F(x):=\sum_{i=1}^\infty \mathbbm{1}_{\{W_0^{i,N}\leq x/N\}}.$$
 This is a Poisson process of unit intensity. We define also
$$B^N_0(x):=\frac{N^{4/3}x-F({N^{4/3}x})}{N^{2/3}}, \quad x\in \mathbb{R}^+.$$  This converges to a standard Brownian motion $(B_0(x))_{x\in\mathbb{R}^+}$, as $N\rightarrow\infty$. We set $x^*=N^{1/3}\inf\{x:B^N_0(x)\geq k\}$. The mass of particles absorbed by a barrier started at this height, is initially $kN^{-1/3}$ less than the height. Consider the linear function $N^{\alpha}ct+x^*$. By the result of Lemma \ref{critlem} equation \eqref{critlemc}, we can bound $\Gamma(N^\alpha ct)-N^{\alpha}ct\leq \frac{1}{2cN^{\alpha}}$. Thus, for $\Gamma_N(ct+x^*)\geq ct+x^*$ for some $t\leq T$ it must hold that 
$$\Gamma_N(N^{\alpha}l^c+x^*)_t-\Gamma_N(x^*)_t-\Gamma(N^{\alpha}l^c)_t\geq \frac{k}{N^{1/3}}-\frac{1}{2cN^{\alpha}}.$$
Identically as for the lower bound, the probability of this occurring  has a non-degenerate limit exactly when $\alpha=1/3$ and $k\geq 1/2c$. 

Formalising these arguments we obtain the following theorem.
\begin{theorem}
\label{crit}
If points are initially placed via a Poisson point process of unit intensity, then 
$$ \lim_{y\rightarrow\infty}\liminf_{N\rightarrow\infty}\mathbb{P}(1/y \leq \xi_{N^2}/N^{4/3}\leq y)=1. $$
\end{theorem}

We give a full proof of this result in Section \ref{critfull} of the appendix.
\begin{remark}
    This result also holds for $\xi_{N^2t}/N^{4/3}$ for an arbitrary $t>0$, through identical arguments.
\end{remark}

We now describe heuristically how the limiting process appears by comparing the behaviour in a small time increment of $t$, to the particle density near $\Lambda_t^N$. We define random processes $R^N,$ $F^N$, $\tilde F^N$, and $B^N$ by 
\begin{align*}
&R_t^N:=\Lambda^N_t/N^{1/3},\\
    &F_t^N(x):=\frac{1}{N}\sum_{i=1}^\infty \mathbbm{1}_{\{W_0^{i,N}+B_t^{i,N}\leq x\}},\\
    &\tilde F_t^N(x):=\frac{1}{N}\sum_{i=1}^\infty \mathbbm{1}_{\{W_0^{i,N}+B_t^{i,N}\leq x, \tau^i>t\}},\\
    &B^N_t(x):=N^{1/3}\left(\int_{-\infty}^{N^{1/3}x}\Phi\left(\frac{z}{\sqrt{t}}\right)dz-F_t^N(N^{1/3}x)\right).
\end{align*}

We remark that the definition $B^N_0(x)$ defined here, agrees with the previously given definition. Further, we use $B^N_t(x)$ to denote a cadlag process in $x$, for each fixed $t$, which should not be confused with $B_t^{i,N}$ used to denote a process in time $t$.
For each $t$, $(B^N_t(x))_{x\in\mathbb{R}^+}$ converges to a Brownian motion $(B_t(x))_{x\in\mathbb{R}^+}$. We can expect that the Brownian movement of particles should average out, so for any time $t$, $(B^N_t(x))_{x\in\mathbb{R^+}}$ converges to the same Brownian motion $(B_0(x))_{x\in\mathbb{R^+}}$.

We consider the behaviour of $R_{s}^N$ on a small time interval $[s,s+\Delta]$. We define an operator $\tilde \Gamma_N^s$ in an identical manner to $\Gamma_N$, but with an initial particle distribution given by $\tilde F^N_s$. 

Since the speed of the $\Lambda^N_t$ movement is much faster than the particle movements, we expect that $\tilde F^N_s((N^{1/3}x+R_s^N))-\tilde F_s(R^N_s)$ should be approximately distributed as 1/N multiplying a Poisson process of rate $N^{4/3}$, and which is independent of $R_s^N$. This suggests that $\tilde \Gamma_N^s(N^{1/3}l^{c,d+R_s^N})$ should be approximately distributed as $F^N_s(N^{1/3}R_s^N)$ plus an independent copy of $\Gamma_N(N^{1/3}l^{c,d})$, which we denote by $\Gamma_{N}'(N^{1/3}l^{c,d})$.

Applying the above heuristics, and applying the result of \eqref{critlemcd} we expect that for small $t$, and small $d$,

\begin{align*}
    &N^{1/3}(\tilde\Gamma^s_N(N^{1/3}l^{c,d+R^N_s})_t-N^{1/3}(ct+d+R_s^N))\\
    &\overset{d}{\approx}N^{1/3}(\Gamma_N'(N^{1/3}l^{c,d})_t+F^N(N^{1/3}R_s^N)-N^{1/3}(ct+d+R^N_S)),\\
    &=N^{1/3}\left(\Gamma_N'(N^{1/3}l^{c,d})_t-\Gamma(N^{1/3}l^{c,d})_t\right)+N^{1/3}(F^N(N^{1/3}R_s^N)-N^{1/3}R_s^N)\\
    &\hspace{6.5cm}+N^{1/3}\left(\Gamma(N^{1/3}l^{c,d})_t-N^{1/3}(ct+d)\right)\\
    &\approx N^{1/3}(\Gamma_N'(N^{1/3}l^{c,d})_t-\Gamma(N^{1/3}l^{c,d})_t)-B_0(R_s)+\frac{1}{2c},\\
    &\approx o(1)_{t\rightarrow 0}+\left(\frac{1}{2c}-B_0(R_s)\right).
\end{align*}

When $\tilde\Gamma^s_N(N^{1/3}l^{c,d+R^N_s})_t-N^{1/3}(ct+d+R_s^N)$ is positive for $t\leq \Delta$, the result of Lemma \ref{Alemmatoboundathing} implies that $R^N_{t+s}\geq ct +d+R^N_s$ on $[s,s+\Delta]$. When it is negative, $R^N_{t+s}\leq ct+d+R^N_s$ on $[s,s+\Delta]$.
The first above is small for small $t$, hence the above is negative when $1/c< 2(B_0(R_s))_+$, and positive when $1/c>2(B_0(R_s))_+$. We then expect that when $B_0(R_t)$ is positive, and $\Delta$ is small, $$R_{t+\Delta}\approx R_t+\frac{\Delta}{2(B_0(R_t))_+}.$$
 Therefore:
$$\dot R_t=\frac{1}{2(B_0(R_t))_+}.$$

When $B_0(R_{t^-})$ is negative, any value of $c$ should give that $R_{t+\Delta}\geq {c\Delta}$, and therefore $R_{t^-}$ should jump up by the smallest $x$ such that $B_0(R_{t^-}+x)>0$ again.

This yields exactly the behaviour of $R$ as described above Theorem \ref{crittheorem}. To prove the result, we check each of the above heuristics.

\subsection{Joint convergence with the initial particle configuration}

We first check that the sequence $R_t^N$ is tight, and that it converges jointly with the $B^N_t(x)$ processes to some $R_t$ and $B_0(x)$, where $B_0(x)$ is a standard Brownian motion.

\begin{lemma}
\label{Rtight}
     $(R_t^N)_{t\in [-1,T+1]}$ is a tight process under the Skorokhod M1 topology.
\end{lemma}

\begin{proof}
We define a pre-compact set $A_K$ by the following:
$$A_K=\{l: ||l||_{\infty}\leq K,\quad l\textrm{ non-decreasing, } l\textrm{ constant on }[-1,0)\textrm{ and on }[T,T+1]\}$$
This set is pre-compact by \cite[Theorem 12.12.2]{whitt}.
From the extension procedure, $R^N$ is non-decreasing and constant on $[-1,0)$ and on $[T,T+1]$. Hence $\mathbb{P}(R^N\in  A_K)=\mathbb{P}(R^N_{T}\leq K)$ for every $N$. By Theorem \ref{crit}, it holds that $$\lim_{K\rightarrow\infty}\limsup_{N\rightarrow\infty }\mathbb{P}(R^N_T \geq K)=0.$$ Therefore, for any $\epsilon>0$, there exists a $K_\epsilon$ such that $\liminf_{N\rightarrow\infty} \mathbb{P}( R^N\in A_{K_{\epsilon}})\geq 1-\epsilon$. 
 \end{proof}
  
\begin{lemma}
\label{lemmabm}
    For an arbitrary $t\geq 0,$ the process $(B_t^N(x))_{x\in\mathbb{R}^+}$ converges weakly to a standard Brownian motion $(B_t(x))_{x\in\mathbb{R}^+}$. 
\end{lemma}
\begin{proof}
The process $B_t^N(x)$ satisfies the conditions of \cite[Theorem 7.1.4]{11}, from which the result follows.
\end{proof}

\begin{lemma}
\label{lemmabm2}
    For any integer $k$ and for $0\leq t_1\leq t_2\leq ...\leq t_k$, the processes \\$(B_0^N(x),B_{t_1}^N(x),...,B^N_{t_k}(x))_{x\in\mathbb{R}^+}$ converge in distribution under the uniform topology to $k$ identical copies of a standard Brownian motion $B_0$:
    $$(B_0^N(x),B_{t_1}^N(x),...,B^N_{t_k}(x))_{x\in\mathbb{R}^+}\implies (B_0(x),B_0(x),...,B_0(x))_{x\in\mathbb{R}^+}.$$

    Further, for any countable set $S$, $((B^N_t(x))_{x\in\mathbb{R}^+})_{t\in S}$ as an element of $D(\mathbb{R}^+)^\infty$ endowed with the product Skorokhod topology converges weakly to $((B_0(x))_{x\in\mathbb{R}^+})_{t\in S}$.
\end{lemma}

This proof follows from Lemma \ref{lemmabm}, and elementary computations. It is given in full in Section \ref{critfull}.

From these above results, we obtain the joint (subsequential) convergence of $R^N$, and $B_t^N$.
\begin{corollary}
    For any subsequential limit $R$ of $R^N$, there is a countable dense set $\hat J$ such that $\mathbb{P}(R_t=R_{t-} \quad \forall t\in \hat J)=1$. The convergent subsequence $R^{N_k}$ may be coupled with $(B^{N_k}_t(x))_{t\in \hat J}$ such that the processes $(R^N,(B^{N_k}_t(x))_{t\in\hat J})$ converge in distribution to $(R,(B_0(x))_{t\in\hat J})$, where $B_0(x)$ is a standard Brownian motion.
\end{corollary}

\begin{remark}
    We expect that in fact $(B_t^N(x))_{t\leq T}$ as an element of $D([0,T],D(\mathbb{R}^+))$ should converge under the uniform topology to $(B_0(x))_{t\leq T}$. However, we are presently unable to show this. Therefore, we work only with a countable collection of times, so that the random variables under consideration are elements of separable spaces.
\end{remark}
\subsection{Bounds on the particle system movement}

We first show approximately that $R^N_{t+s}\leq R^N_t+s/2c$ when $B^N_0(R^N_t)\geq c$. This provides the upper bound on the derivative for the limit process $R$.
\begin{proposition}
\label{propcritupper}
        Fix arbitrary $a,\beta,l,\epsilon>0,$ $0<C_2<C_1$, and $t\geq 0$. Define events $A_1^N,A_2^N,A_3^{N}$ by:
    \begin{align*}
    &A_1^{N}:=\left\{\sup_{s\leq l}\left(\Lambda_{t+s}^N-\left( \Lambda_t^N+\frac{sN^{1/3}}{2a}+\beta N^{1/3}\right)\right)>0\right\},\\
    &A_2^{N}:=\left\{\inf_{0\leq x\leq l/2a+\beta}B^N_t(R^N_t+x)\geq (a+\epsilon)\right\},\\
    &A_3^{N}:=\left\{\Lambda^N_T\leq C_1N^{1/3},C_2N^{1/3}\leq \Lambda_t^N\right\},\\
    &A^{N}:=A_1^{N}\cap A_2^{N}\cap A_3^{N}.
    \end{align*}
    It holds that
    $$\limsup_{N\rightarrow\infty}\mathbb{P}(A^{N})=0.$$
    
\end{proposition}

\begin{proof}[Proof of Proposition \ref{propcritupper}]
  Consider the function $f_s^N=\Lambda^N_{t\wedge s}+{N^{1/3}s}/(2a)+\beta N^{1/3}$. If $\Gamma_N(f)_s\leq f_s$ for all $s\leq t+l$, then, from Lemma \ref{Alemmatoboundathing}, it must hold that $\Lambda^N_s\leq f_s$ for $s\leq t+l$. Therefore, on $A^N$ it must hold that $\Gamma_N(f)_{t+x}\geq f_{t+x}$ for some $x\leq l$. We write $\hat B_s^i=B_{t+s}^i-B_t^i$.

We first observe that
 \begin{align*}
    \Gamma_N(f)_{t+x}&=\frac{1}{N}\sum_{i=1}^\infty \mathbbm{1}_{\{W_0^{i,N}+B_t^{i,N}\leq \Lambda^N_t+\sup_{s\leq x}(\frac{sN^{1/3}}{2a}+\hat B_s^{i,N})+\beta N^{1/3}\}}\\
    &+\frac{1}{N}\sum_{i=1}^\infty \mathbbm{1}_{\{\tau_i^N\leq t,W_0^{i,N}+B_t^{i,N}>\Lambda^N_t+\sup_{s\leq x}(\frac{sN^{1/3}}{2a}+\hat B_s^{i,N}+\beta N^{1/3})\}},\\
    &\leq \frac{1}{N}\sum_{i=1}^\infty \mathbbm{1}_{\{W_0^{i,N}+B_t^{i,N}\leq \Lambda^N_t+\sup_{s\leq x}(\frac{sN^{1/3}}{2a}+\hat B_s^{i,N})+\beta N^{1/3}\}}\\
    &+\frac{1}{N}\sum_{i=1}^\infty \mathbbm{1}_{\{W_0^{i,N}+\inf_{s\leq t}B_s^{i,N}\leq \Lambda^N_t,W_0^{i,N}+B_t^{i,N}>\Lambda^N_t+\beta N^{1/3}\}}.  
\end{align*}

To obtain the inequality above, we have increased the set of indices $\{i:\tau_i^N\leq t\}$ to the set $\{i:W_0^{i,N}+\inf_{s\leq t}B_s^{i,N}\leq \Lambda^N_t\}$, and have increased the set $\{i:W_t^{i,N}>\Lambda^N_t+\sup_{s\leq x}(\frac{sN^{1/3}}{2a}+\hat B_s^{i,N}+\beta N^{1/3})\}$ to $\{i:W_t^{i,N}>\Lambda^N_t+\beta N^{1/3}\}$.
Subtracting and adding $N^{-1/3}B_t^N(R_t^N+\beta+x/2a)$, we observe that if $\Gamma_N(f)_{t+x}> f_{t+x}$ for some $x\leq l$, it must hold that
\begin{align*}
&\frac{1}{N}\sum_{i=1}^\infty \mathbbm{1}_{\{W_t^{i,N}\leq \Lambda^N_t+\sup_{s\leq x}(\frac{sN^{1/3}}{2a}+\hat B_s^{i,N})+\beta N^{1/3}\}}\\
&-\frac{1}{N}\sum_{i=1}^\infty \mathbbm{1}_{\left\{W_t^{i,N}\leq \Lambda_t^N+\beta N^{1/3}+\frac{N^{1/3}x}{2a}\right\}}\\
    &+\frac{1}{N}\sum_{i=1}^\infty \mathbbm{1}_{\{W_0^{i,N}+B_t^{i,N}>\Lambda^N_t+\beta N^{1/3},W_0^{i,N}+\inf_{s\leq t}B_s^{i,N}\leq \Lambda^N_t\}}\\
    &-N^{-1/3}B^N_t\left(R^N_t+\beta +\frac{x}{2a}\right)>0.
    \end{align*}

Therefore, on the event $A^N$ it holds that
\begin{align*}
    &\frac{1}{N}\sum_{i=1}^\infty \mathbbm{1}_{\{W_t^{i,N}\leq \Lambda^N_t+\sup_{s\leq x}(\frac{sN^{1/3}}{2a}+\hat B_s^{i,N})+\beta N^{1/3}\}}\\
    &-\frac{1}{N}\sum_{i=1}^\infty \mathbbm{1}_{\left\{W_t^{i,N}\leq \Lambda_t^N+\beta N^{1/3}+\frac{N^{1/3}x}{2a}\right\}}\\
    &+\frac{1}{N}\sum_{i=1}^\infty \mathbbm{1}_{\{W_0^{i,N}+B_t^{i,N}>\Lambda^N_t+\beta N^{1/3},W_0^{i,N}+\inf_{s\leq t}B_s^{i,N}\leq \Lambda^N_t\}}\\
    &\quad\quad>(a+\epsilon)N^{-1/3}.
\end{align*}

 We define random processes $G^{N,y}_x,$ $H^{N,y}_x$, $M^{N,y}_x$ and $Z^{N,y}$ by

\begin{align*}&G_{x}^{N,y}:={\frac{1}{N}\sum_{i=1}^\infty\mathbbm{1}_{\{W_0^{i,N}+B_t^{i,N}\leq y+\sup_{s\leq x}(\frac{sN^{1/3}}{2a}+\hat B_s^{i,N})+\beta N^{1/3}\}}}_,\\
&H_x^{N,y}:={\frac{1}{N}\sum_{i=1}^\infty \mathbbm{1}_{\{W_0^{i,N}+B_t^{i,N}\leq y+\beta N^{1/3}+\frac{N^{1/3}x}{2a}\}}}_,\\
&M_x^{N,y}:=G_x^{N,y}-H_x^{N,y},\\
&Z^{N,y}:=\frac{1}{N}\sum_{i=1}^\infty \mathbbm{1}_{\{W_0^{i,N}+B_t^{i,N}>y+\beta N^{1/3},W_0^{i,N}+\inf_{s\leq t}B_s^{i,N}\leq y\}}.
\end{align*}

The equation given above the definition of these processes can then be rewritten as 
$$M_x^{N,\Lambda^N_t}+Z^{N,\Lambda^N_t}>(a+\epsilon)N^{-1/3}.$$
The result of the proposition then follows from the following three elementary lemmas, the proofs of which are given in Section \ref{critfull} of the appendix.

\begin{lemma}
    \label{lemma1incritlabel}
    $$\mathbb{P}\left(\sup_{\{y\leq C_1N^{1/3},y\in\frac{1}{N}\mathbb{Z}\}}\sup_{x\leq l}(M^{N,y}_x-\mathbb{E}(M^{N,y}_x))>\frac{\epsilon}{3N^{1/3}}\right)\rightarrow 0\textrm{ as }N\rightarrow\infty.$$
\end{lemma}

This follows via Chernoff bounds on $M^{N,y}_x$ for a fixed $x$, followed by bounding the differences of $M^{N,y}_{x_1}-M^{N,y}_{x_2}$ via a martingale argument. 

\begin{lemma}
\label{lemmaonMmean}
    For all sufficiently large $N$, 
    $$\mathbb{E}(M^{N,y}_x)\leq \frac{a+\epsilon/6}{N^{1/3}}.$$
\end{lemma}

This follows from simple computations and Lemma \ref{critlem}.

\begin{lemma}
\label{lemmaz}
    $$\mathbb{P}\left(\sup_{\{y\leq C_1N^{1/3},y\in\frac{1}{N}\mathbb{Z}\}}Z^{N,y}>\frac{\epsilon}{2N^{1/3}}\right)\rightarrow 0\textrm{ as }N\rightarrow\infty.$$
\end{lemma}

Recall that on $A^N$ there is some $x\leq l$ such that
$$M_x^{N,\Lambda^N_t}+Z^{N,\Lambda^N_t}>(a+\epsilon)N^{-1/3}.$$
Therefore, either
\begin{align*}
    &\sup_{\{y\leq C_1N^{1/3},\textrm{ }y\in \frac{1}{N}\mathbb{N}\}}{\sup_{\{x\leq l/2a\}}(M_x^{N,y})-aN^{-1/3}>\frac{\epsilon}{2N^{1/3}}}_,
\end{align*}

or

\begin{align*}
    \sup_{\{y\leq C_1N^{1/3},\textrm{ }y\in \frac{1}{N}\mathbb{N}\}}Z^{N,y}>\frac{\epsilon}{2N^{1/3}}.
\end{align*}

Applying Lemma \ref{lemmaonMmean}, followed by Lemma \ref{lemma1incritlabel}, we obtain that the probability of the first event occurring converges to zero, as $N\rightarrow\infty$. Applying Lemma \ref{lemmaz}, we obtain that the probability of the second event occurring converges to zero as $N\rightarrow\infty$. Therefore, the probability of the event $A^N$ converges to zero, as $N\rightarrow\infty$.

\end{proof}
We also check that $R^N_{t+s}\geq R^N_t+s/2c$ when $B^N_t(\Lambda^N_t)\leq c$. This provides a lower bound on the derivative of the limiting process.
\begin{proposition}
\label{propcritlower}
         Fix arbitrary $a,l>0$, $0<\epsilon<a$, $t\geq 0$, and $0<C_2<C_1$. We define sets $A_1^{N},A_2^{N},A_3^{N}$ in the following manner: 
    \begin{align*}
    &A_3^{N,'}:=\left\{\inf_{s\leq l}\left(\Lambda_{t+s}^N-\left( \Lambda_t^N+\frac{sN^{1/3}}{2a}-{2a}{N^{-1/3}}\right)\right)<0\right\},\\
    &A_2^{N,'}:=\left\{\inf_{-2aN^{-2/3}\leq x\leq l/2a}B^N_t(R^N_t+x)\leq (a-\epsilon)\right\},\\
    &A_3^{N,'}:=\left\{\Lambda^N_T\leq C_1N^{1/3},C_2N^{1/3}\leq \Lambda_t^N\right\},\\
    &A^{N,'}:=A_1^{N,'}\cap A_2^{N,'}\cap A_3^{N,'}
    \end{align*}

    It holds that
    $$\limsup_{N\rightarrow\infty}\mathbb{P}(A^{N,'})=0.$$
    
\end{proposition}
\begin{proof}[Proof of Proposition \ref{propcritlower}]
     We define a function $f_s$: $$f_{s}:=\left(\Lambda^N_{s\wedge t}+\frac{N^{1/3}(s-t)_+}{2a}-2aN^{-1/3}\mathbbm{1}_{\{s> t\}}\right).$$ As in the proof of Proposition \ref{propcritupper}, on the event $A^N$ it must hold that $f_{t+x}<\Gamma_N(f)_{t+x}$ for some $x\leq l$. 

We bound $\Gamma_N(f)_{t+x}$ in the following manner:
    \begin{align*}\Gamma_N(f)_{t+x}&\geq {\frac{1}{N}\sum_{i=1}^\infty \mathbbm{1}_{\left\{W_0^{i,N}+B_t^{i,N}\leq \Lambda^N_t\vee \left(\Lambda^N_{ t}+\sup_{s\leq x}\left(\frac{N^{1/3}s}{2a}+\hat B_s^i-2aN^{-1/3}\right)\right)\right\}}}_,\\
    &\geq \frac{1}{N}\sum_{i=1}^\infty\mathbbm{1}_{\left\{W_0^{i,N}+B_t^{i,N}\leq \Lambda^N_t\vee \left(\Lambda^N_{t}+\frac{N^{1/3}x}{2a}-2aN^{-1/3}\right)\right\}}\\
    &\quad -\frac{1}{N} \sum_{i=1}^\infty\mathbbm{1}_{\left\{W_0^{i,N}+B_t^{i,N}\leq \Lambda^N_t+\frac{N^{1/3}x}{2a}-2aN^{-1/3}\right\}}\\
    &-N^{-1/3}B^N_t(N^{-1/3}f_{t+x})+f_{t+x},\\
    &\geq \frac{1}{N}\sum_{i=1}^\infty\mathbbm{1}_{\left\{W_0^{i,N}+B_t^{i\,N}\leq \Lambda^N_t\vee \left(\Lambda^N_{t}+\sup_{s\leq x}\left(\frac{N^{1/3}s}{2a}+\hat B_s^i-2aN^{-1/3}\right)\right)\right\}}\\
    &\quad -\frac{1}{N} \sum_{i=1}^\infty\mathbbm{1}_{\left\{W_0^{i,N}+B_t^{i,N}\leq \Lambda^N_t+\frac{N^{1/3}x}{2a}-2aN^{-1/3}\right\}}\\
    &-(a-\epsilon)N^{-1/3}+f_{t+x}.\end{align*}

Analogously to the proof of Proposition \ref{propcritupper}, we define random processes by

\begin{align*}
    &G_x^{N,y}:={\frac{1}{N}\sum_{i=1}^\infty \mathbbm{1}_{\left\{W_0^{i,N}+B_t^{i,N}\leq y\vee \left(y+\sup_{s\leq x}\left(\frac{N^{1/3}s}{2a}+\hat B_s^i-2aN^{-1/3}\right)\right)\right\}}}_,\nonumber\\
    &H_x^{N,y}:={\frac{1}{N} \sum_{i=1}^\infty\mathbbm{1}_{\left\{W_0^{i,N}+B_t^{i,N}\leq y+\frac{N^{1/3}x}{2a}-2aN^{-1/3}\right\}}}_,\\
    &M^{N,y}_x:=G_x^{N,y}-H^N_{x,y}\nonumber
\end{align*}

On $A^{N,'}$, it must hold then that

\begin{align*}
    &\inf_{\left\{C_2N^{1/3}\leq y\leq C_1N^{1/3}, \textrm{ }y\in \frac{1}{N}\mathbb{N} \right\}}\inf_{x\leq l}\left(M^{N,y}_x-\frac{a}{N^{1/3}}\right)<-\frac{\epsilon}{N^{1/3}}.
\end{align*}

The result of the Proposition follows from the below lemmas.

\begin{lemma}
\label{lowerlemmacrit1}
    $$\mathbb{P}\left(\inf_{\left\{C_2N^{1/3}\leq y\leq C_1N^{1/3}, \textrm{ }y\in \frac{1}{N}\mathbb{N} \right\}}\inf_{x\leq l}\left(M^{N,y}_x-\mathbb{E}(M_x^{N,y})\right)<-\frac{\epsilon}{2 N^{1/3}}\right)\rightarrow 0 \textrm{ as }N\rightarrow\infty.$$
\end{lemma}
This follows from an analogous argument to Lemma \ref{lemma1incritlabel}.

\begin{lemma}
\label{lemmaonMmeanlower}
    For all sufficiently large $N$,
     \begin{align*}
    &\mathbb{E}(M^{N,y}_x)\geq \frac{a-\epsilon/2}{N^{1/3}}.
    \end{align*}
\end{lemma}
This follows from simple computation, analogously to  \ref{lemmaonMmean}. 

Applying the result of Lemma \ref{lemmaonMmeanlower}, and then Lemma \ref{lowerlemmacrit1}, we observe that

\begin{align*}
    &\mathbb{P}\left(\inf_{\left\{C_2N^{1/3}\leq y\leq C_1N^{1/3}, \textrm{ }y\in \frac{1}{N}\mathbb{N} \right\}}\inf_{x\leq l}\left(M^{N,y}_x-\frac{a}{N^{1/3}}\right)<-\frac{\epsilon}{N^{1/3}}\right)\rightarrow 0 \textrm{ as }N\rightarrow\infty.
\end{align*}
\end{proof}

Fix a weakly convergent subsequence $R^{N_k}$. We write $(R,B_0)$ for the corresponding limit of $(R^{N_k},(B_t(x))_{t\in\hat J})$, where we have used $B_0(x)$ to identify $(B_0(x))_{t\in \hat J}$, since this function is constant in $t$. From the above propositions, we obtain bounds on these limit processes. 
\begin{corollary}
\label{corupper}
    Almost surely, for any $a,l>0$ and $t\leq T$:
    $$B_0(R_t+x)\geq a \textrm{ for }x\leq l/2a$$
    $$\implies R_{t+s}\leq R_t+\frac{s}{2a}\textrm{ for }s\leq l.$$
\end{corollary}

This follows through elementary arguments by applying the Portmanteau Theorem to the result of Proposition \ref{propcritupper}. We give the proof in Section \ref{critfull}.

\begin{corollary}
\label{corlower}
    Almost surely, for any $a,l>0$ and $t\leq T$:
    \begin{align*}&B_0(R_t+x)\leq a\textrm{ for } x\leq l/2a\\
    &\implies R_{t+s}\geq R_t+s/2a\textrm{ for }s\leq l.\end{align*}
\end{corollary}

The result follows through an identical argument to the previous corollary, replacing Proposition \ref{propcritupper}, with \ref{propcritlower}.

\subsection{Identification of the limit}

We consider first the behaviour of $R$ when $B_0({R_t})>0$.

\begin{lemma}
\label{deriv}
    At any time $t$ where $(B_0({R_t}))_+$ is non-zero, $R$ is differentiable, with
    $$\dot R_t=\frac{1}{2(B_{R_t})_+}.$$

    Therefore, for any $0\leq s\leq t$ such that $(B_0({R_l}))_+>0$ for $l\in(s,t)$
    $$R_l-R_s=\inf\left\{x:\int_0^x(2B_0(R_s+y))_+dy>l-s\right\} \textrm{ for }l\in[s,t).$$
\end{lemma}

The first statement of the result follows by taking limits over $l$ and $a$ in Corollary \ref{corlower} and Corollary \ref{corupper}, while the second follows from identifying the solution to the corresponding ODE. More details are given in Section \ref{critfull} of the appendix.

We consider also the behaviour of $R$ when $(B_0({R_{t-}+x}))_+=0$ for $x$ in some interval.

\begin{lemma}
\label{jump}
    Suppose that at time $t$, $(B_0({R_{t^-}+x)})_+=0$ on $[0,\varpi]$. It follows that $$\Delta R_t\geq \varpi.$$
    Further, at any time $t$,
    $$\Delta R_t\leq \inf\{x:B_0({R_{t^-}+x})>0\}.$$ 
    Therefore:
    $$\Delta R_t=\inf\{x:B_0(R_{t^-}+x)>0\}=\inf\left\{x:\int_0^x(2B_0(R_{t^-}+s))_+ds>0\right\}.$$
\end{lemma}

The first result follows by taking limits in Corollary \ref{corlower}, while the second follows from a simple adaptation to the arguments of Corollary \ref{corupper}. Details are given in Section \ref{critfull} of the appendix.

\begin{lemma}\label{denseness}
The set $\{t\leq T:R_t \textrm{ is differentiable at }t\}$ is open, and may be written as $\cup_{i\in\mathbb{N}}(a_i,b_i)$ for disjoint intervals $(a_i,b_i)$. Any point $t<T$ either lies in some $(a_i,b_i)$ interval, or is equal to $a_i$ for some $i$.
\end{lemma}

This is a technical result rather than a probabilistic one and follows from simple arguments after applying the results of the previous two lemmas. Details are given in Section \ref{critfull} of the appendix.

From the results of this section, we may characterise exactly the behaviour of $(R,B_0)$.
\begin{theorem}
    For any limiting $(R_t,B_0)$, almost surely
    $$R_t=\inf\left\{x:\int_0^x(2B_0(s))_+ds>t\right\} \textrm{ for all }t\in[0,T].$$
\end{theorem}

\begin{proof}

    Defining $\tilde R_t:=\inf\left\{x:\int_0^x(2B_0(s))_+ds>t\right\}$, the desired limit process, this satisfies the properties of Lemmas \ref{deriv} and \ref{jump}. We shall check that any other process $R$ which satisfies these properties must agree with $\tilde R$.

    Suppose that $\tilde R_t\neq R_t$ for some $t\leq T$. Let $\tau=\inf\{t\geq 0:\tilde R_t\neq R_t\}$ which is less than or equal to T by assumption. We first note that $\tau$ cannot be equal to $T$. It is immediate that $R_{\tau^-}=\tilde R_{\tau^-}$.
    If there is a discontinuity at time $\tau$, by Lemma \ref{jump} it is exactly given by:
    $$\Delta R_{\tau}=\inf\left\{x:\int_{R_{\tau^-}}^x2(B_0(s))_+ds>0\right\}=\inf\left\{x:\int_{\tilde R_{\tau^-}}^x2(B_0(s))_+ds>0\right\}=\Delta \tilde R_{\tau}.$$
    By construction both $R$ and $\tilde R$ are constant on $[T,T+1]$. It then holds that $R_t=\tilde R_t$ for $t\leq T+1$ contradicting the definition of $\tau$, and therefore that $\tau<T$. 
    
    By Lemma \ref{denseness}, three options for $\tau$ are left. The first is that $\tau$ is a left endpoint of some $(a_i,b_i)$ interval, at which $R$ is discontinuous. The second is that is that $\tau$ is a left endpoint of some $(a_i,b_i)$ interval, at which $R$ is continuous. The third is that $\tau$ lies inside some $(a_i,b_i)$ interval. 

    In the first case, identically as for the case of $\tau=T$, we obtain that $R_\tau=\tilde R_{\tau}$. Applying Lemma \ref{deriv}, it holds that for $s\leq b_i-a_i$:
    $$R_{\tau+s}=R_{\tau}+\inf\left\{x:\int_{ R_{\tau}}^x2(B_0(s))_+ds>s\right\}=\tilde R_\tau+\inf\left\{x:\int_{ \tilde R_{\tau}}^x2(B_0(s))_+ds>s\right\}=\tilde R_{\tau+s}.$$
    This contradicts the definition of ${\tau}$, hence the first case cannot occur. 

    In the second case it is immediate that $R_{\tau}=\tilde R_{\tau}$. From an identical application of Lemma \ref{deriv}, we determine that $R_t=\tilde R_t$ on $[a_i,b_i)$, which again contradicts the definition of $\tau$. The third case follows via an identical argument to the second.

    It must therefore hold that $\tau>T$ almost surely, and so $R=\tilde R$ as claimed.
\end{proof}

\begin{remark}
    In \cite[Theorem 1.3]{Dembo_2019}, $x_i$ points are placed at each integer, where $x_i\sim x_1$ are i.i.d. with mean 1, and finite second moment. Defining the random CDF function $F(x)=\sum_{i\leq x} x_i$, $F$ has stationary and independent increments. Therefore, in the case of continuous space $F(x)$ should be a Levy point-process of mean 1, and finite second moment. If we take particles only of unit mass, this yields that $F(x)$ is a standard Poisson process. Without this restriction, it is also reasonable to consider $F(x)$ to be a compound Poisson process, corresponding to us placing a random amount of mass at each point. Most of the bounds from above generalise to this case by replacing  $\sum_{i=1}^\infty \mathbbm{1}_{\{X_0^{i,N}+B_t^{i,N}\leq x\}}$ by
     $\sum_{i=1}^\infty  p_i\mathbbm{1}_{\{W_0^{i,N}+B_t^{i,N}\leq x\}}$ for i.i.d. $p_i$, independent of the collection $(W^i_0,B^i)_{i\in\mathbb{N}}$. However, the Chernoff bound used in Lemmas \ref{lemma1incritlabel} and \ref{lowerlemmacrit1} must be re-derived, for the given $p_i$ distribution, and this requires $p_i$ to have a moment generating function. 
     The results should also generalise to $p_i$ with at least several moments by using moment bounds rather than Chernoff bounds.

     In \cite[Theorem 1.8]{Dembo_2019}, the authors extend the equivalent result to more generic initial particle configurations. This is outside the scope of our work and cannot be handled by our martingale arguments, which require the initial random CDF to correspond to a compound Poisson process. 
    
\end{remark}

\section{Alternative conditions on the density, and related barrier dynamics}
\label{alternateg}
In this section we discuss alternative conditions on the density $g$ to condition \hyperlink{A2}{(A2)}, under which global solutions to \eqref{alternativetomvr} exist. We first present an alternative condition on the density $g$, under which the particle system of Section \ref{setupsection} converges to physical solutions of \eqref{alternativetomvr}. We demonstrate that, in general, there need not be a time after which the solution ceases to have blow-ups. We then present another alternative condition under which global solutions of the McKean-Vlasov system \eqref{alternativetomvr} exist, and demonstrate that, in general, there need not be a linear bound on $\Lambda_t$, unlike for densities satisfying condition \hyperlink{A2}{(A2)}.

\subsection{A condition allowing for arbitrarily many jumps of the barrier}

We consider the following condition on the density.
\begin{itemize}
    \item \hypertarget{A3}{(A3)}: There exists sequences $x_n\uparrow\infty,$ $y_n\uparrow\infty$ such that $\int_0^{x_n}(g(x)-1)dx\leq -y_n$ for all $n\in \mathbb{N}$.
\end{itemize}

\begin{lemma}
\label{lemmaa3}
    Suppose that $g$ satisfies conditions \hyperlink{A1}{(A1)}, \hyperlink{A3}{(A3)}. For each $T>0$ there exists a constant $C_T$ such that $\mathbb{P}(\Lambda^N_T>C_T)\rightarrow 0$ as $N\rightarrow\infty$. 
\end{lemma}

The proof of this result follows from a simple modification of the proof of Lemma \ref{lemma on bounding}, using $f_t\equiv x_n$ rather than $f_t=ct+x^*$.

As discussed at the start of the proof sketch in Section \ref{sketchsection}, for any density $g$ where such a bound on $\Lambda^N$ holds, the conclusion of Theorem \ref{maintheorem} holds. Therefore, the result of Theorem \ref{maintheorem} holds under replacing the condition \hyperlink{A2}{(A2)} with condition \hyperlink{A3}{(A3)}. 

\begin{corollary}
    Suppose that $g$ satisfies conditions \hyperlink{A1}{(A1)} and \hyperlink{A3}{(A3)}. Then the statement of Theorem \ref{maintheorem} holds for the system \eqref{aligned:Particlealpha} with initial density $g$.
\end{corollary}

For densities satisfying \hyperlink{A3}{(A3)}, it is possible to have a sequence of blow-up times growing to infinity, unlike for solutions for $g\in L^1(\mathbb{R}^+)$. To give an example we require a criteria for the blow-up of solutions. We therefore modify the proof of \cite[Theorem 1.1]{2} to give a criteria applicable in our setting.  

\begin{theorem}
\label{blowuptheorem}
    Suppose that there exists some Borel set $A$ of positive measure such that
    $$\int_A xg(x)dx<\frac{1}{2}\int_Ag(x)\int_0^xg(y)dydx.$$
    Then any solution $\Lambda$ to \eqref{alternativetomvr} cannot be continuous for all time.

    In particular, if there exists $x$ such that
    $$\int_0^xg(y)dy>2x,$$ then any solution $\Lambda$ to \eqref{alternativetomvr} cannot be continuous for all time.
\end{theorem}

\begin{proof}
    Suppose that $\Lambda$ is a solution to \eqref{alternativetomvr} which is continuous for all time. We recall the definition of $\tau^x$
    from \eqref{alternativetomvr}, and define $\Lambda^x_t=\mathbb{P}(\tau^x\leq t)$ so that $\Lambda_t=\int_0^\infty\Lambda^x_tg(x)dx$. Since we assume that $\Lambda_t$ is continuous, it must hold that 
    $$x+B_{\tau^x\wedge t}-\Lambda_{\tau^x\wedge t}\geq 0 \textrm{ for }x\in\mathbb{R}^+.$$
    Let $A$ be a Borel measurable subset of $\mathbb{R^+}$. Taking expectations, and integrating against $g(x)\mathbbm{1}_{A}$, it follows that for any $t>0$
    \begin{align*}
    \int_Axg(x)dx&\geq \int_A\mathbb{E}(\Lambda_{\tau^x\wedge t})g(x)dx\\
    &=\int_A\int_0^t\Lambda_sd\Lambda^x_sg(x)dx.
    \end{align*}
Sending $t\uparrow\infty$, and applying the Monotone Convergence Theorem,
we see that
 \begin{align*}
    \int_Axg(x)dx&\geq\int_A\int_0^\infty\Lambda_sd\Lambda^x_sg(x)dx,\\
&=\int_A\left(\int_0^\infty\left(\int_0^\infty\Lambda^y_sg(y)dy\right)d\Lambda^x_s\right)g(x)dx,\\
    &\geq \int_A\left(\int_0^\infty\left(\int_0^xg(y)dy\right)\Lambda^x_sd\Lambda^x_s\right)g(x)dx \;\;\textrm{ since }\Lambda^x_s \textrm{ is decreasing in }x,\\
    &= \int_A\frac{(\Lambda^x_\infty)^2}{2}g(x)\left(\int_0^xg(y)dy\right)dx\;\;\textrm{ since }\Lambda^x_t \textrm{ is continuous in }t \textrm{ with }\Lambda^x_0=0,\\
    &=\frac{1}{2}\int_A g(x)\left(\int_0^xg(y)dy\right)dx.
    \end{align*}

Since $A$ was arbitrary, we have shown that to have a continuous solution to \eqref{alternativetomvr}, it must hold that for any Borel measurable set $A$,
$$\int_Axg(x)dx\geq \frac{1}{2}\int_Ag(x)\int_0^xg(y)dydx,$$
which concludes the proof of the first statement. 

The second statement is an immediate consequence of the first. If $\int_0^xg(y)dy=2x+\epsilon$, for some positive $\epsilon>0$, then $\int_0^zg(y)dy>2z$ for all $z$ in some neighbourhood of $x$ since $g$ is bounded. We can then take the set $A$ to be this neighbourhood of $x$.
\end{proof}

We now give an example of a density which satisfies condition \hyperlink{A3}{(A3)}, and for which the solution to the McKean-Vlasov type equation \eqref{alternativetomvr}, $\Lambda$, possesses infinitely many discontinuities.
Take constants $L\geq 10$, and $\alpha=1/4$.
We consider a density $h$ given by
\begin{align*}
&h(x)=\begin{cases}
    0\quad x\in \cup_{i\in\mathbb{N}_0}[\hat x_{2i-1},\hat x_{2i}]  , \\
    (1-\alpha)L \quad x\in\cup_{i\in\mathbb{N}_0} [\hat x_{2i},\hat x_{2i+1}],
\end{cases}\\
&\hat x_i=\sum_{j=0}^iL^j=\frac{L^{i+1}-1}{L-1},\quad i\in\mathbb{N}_0,\\
&\hat x_{-1}=0.
\end{align*}

\begin{proposition}\label{prop:infinitejumps}
    There exists a physical solution $\Lambda$ to \eqref{alternativetomvr} with initial density $h$, and a sequence $t_i\uparrow\infty$ such that $\Delta\Lambda_{t_i}>0$ for all $i\in\mathbb{N}$.
\end{proposition}

We sketch the proof below. A full proof is given in Section \ref{infiniteblowup} in the appendix.

\begin{proof} 
This density satisfies condition \hyperlink{A3}{(A3)}, hence a solution to \eqref{alternativetomvr} with such initial density does exist. To determine existence of a sequence of blow-up times, growing to infinity, we suppose that there is some time $\zeta$ after which no blow-up occurs. We consider the time $\tau^*$ at which $\Lambda_t$ reaches $\hat x_i$ for some large even $i$ such that $\tau^*$ is greater than $\zeta$. We then consider \eqref{alternativetomvr} with initial density given by the density of non-absorbed, re-centred particles at the time $\tau^*$, and verify that the condition of Theorem \ref{blowuptheorem} holds for this density. Thus there is a blow-up of the solution after time $\tau^*$. 
\end{proof}

\subsection{Different barrier speeds}

We now give a weaker condition than \hyperlink{A2}{(A2)}, which allows for existence of ``physical" solutions to the McKean-Vlasov type system \eqref{alternativetomvr} which have polynomial growth. To avoid working with the particle system, we recall the definition of minimal solutions from \cite{minimal}.
We call $\underline{\Lambda}$ a minimal solution of \eqref{alternativetomvr} if, for any solution $\Lambda$ of \eqref{alternativetomvr}, $\underline{\Lambda}_t\leq\Lambda_t$ for all $t\geq 0$. If any solutions to \eqref{alternativetomvr} exist, then a minimal solution must exist, due to the monotonicity of the operator $\Gamma$, and the bound $\Gamma(0)\leq \Gamma(\Lambda)=\Lambda$. Existence of a minimal solution then holds by applying Tarski's Theorem as in  \cite[Theorem 4.1]{ledger2024supercooledstefanproblemtransport}. 

The minimal solution to the system \eqref{alternativetomvr} has physical jumps.

 \begin{theorem}\label{minimalisphysical}
     Let $\underline{\Lambda}$ be the minimal solution to \eqref{alternativetomvr}. Then $\underline{\Lambda}$ satisfies the physicality condition \eqref{alternativephysical}
 \end{theorem}
 \begin{proof}
This follows from an identical argument to that of \cite[Theorem 4.7]{ledger2024supercooledstefanproblemtransport}. In particular, suppose that there is a time $t^*$ where there is a non physical jump. We can obtain a minimal solution from initial density $\nu_{ t*-}(dx)$ which must have a physical jump, since solutions with a physical jump exist. We then paste this on to the previous solution up to time $ t^* -$ to give a strictly smaller solution, contradicting the definition of minimality.
\end{proof}

Other properties of minimal solutions are discussed in \cite{minimal} and extend to our setting.

We now replace the condition \hyperlink{A2}{(A2)} with the following:

\begin{itemize}
\item \hypertarget{A4}{(A4)}: 
    $$\limsup_{\lambda\rightarrow 0}\frac{1}{\lambda}\int_0^\infty (1-g(x))e^{-\lambda x}dx=\infty .$$
\end{itemize}

\begin{remark}
 For bounded density $g$, the condition \hyperlink{A4}{(A4)} is strictly weaker than condition \hyperlink{A2}{(A2)}. Suppose that $g$ satisfies conditions \hyperlink{A1}{(A2)} and \hyperlink{A2}{(A2)}. It follows that $$\int_0^\infty g(x)e^{-\lambda x}dx\leq \frac{C(1-e^{- \frac{\lambda(x^*-a)}{C}})+e^{-\lambda x^*}}{\lambda}=-a+O(\lambda)_{\lambda \rightarrow0}.$$ 
 Therefore:
 $$\limsup_{\lambda \rightarrow 0}\frac{1}{\lambda}\int_0^\infty (1-g(x))e^{-\lambda x}dx=\infty.$$
 
It is presently unclear if the condition \hyperlink{A4}{(A4)} can be applied to determine bounds on the particle system behaviour. In particular, the use of our condition \hyperlink{A2}{(A2)} in Lemma \ref{lemma on bounding} required that we could obtain uniform convergence of $\Gamma_N(f)_t$, for some $f$, and that we could bound $\Gamma(f)_t\leq f_t-\epsilon$ for some $\epsilon>0$. We are unable to determine such an $\epsilon$ in general for a density satisfying condition \hyperlink{A4}{(A4)}.
\end{remark}

\begin{theorem}
\label{conda4}
    Suppose that $g$ satisfies conditions \hyperlink{A1}{(A1)} and \hyperlink{A4}{(A4)}. Then a global, minimal solution to \eqref{alternativetomvr} with initial density $g$ exists.
\end{theorem}

This relies upon the following result for the Laplace transform of the initial data.

\begin{lemma}
\label{Ito}
For any solution to \eqref{alternativetomvr} on the interval $[0,t]$, and for any $\lambda>0$, the solution satisfies the following equation:
\begin{align}&\int_{\Lambda_t}^\infty V(t,x)e^{-\lambda x-\frac{\lambda^2t}{2}}dx+\int_0^\infty (1-g(x))e^{-\lambda x}dx-\frac{1}{\lambda}e^{-\lambda \Lambda_t-\frac{\lambda^2t}{2}}\\&=\frac{\lambda}{2}\int_0^te^{-\lambda\Lambda_{s^-}-\frac{\lambda^2s}{2}}ds+\sum_{s\leq t}e^{-\frac{\lambda^2s}{2}}\left(-\int_{\Lambda_{s^-}}^{\Lambda_s} V(s^-,x)e^{-\lambda x}dx-\frac{1}{\lambda}\Delta e^{-\lambda\Lambda_s}\right).\nonumber
\end{align}

Above, we denote $V(t,x)=\tilde V(t,x+\Lambda_t)$, and $V(t^-,x)=\tilde V(t^-,x+\Lambda_{t^-})$.
\end{lemma}

This follows from elementary calculations, which are given in full in Section \ref{itoformula} of the appendix.

To prove the result of Theorem \ref{conda4} we consider the value of solutions to \eqref{alternativetomvr} for initial densities $g(x)\mathbbm{1}_{[0,R]}(x)$. Applying condition \hyperlink{A4}{(A4)} to Lemma \ref{Ito}, we obtain a uniform bound on these solutions. Taking a limit over $R$, we are able to deduce existence of a minimal solution to \eqref{alternativetomvr}. We include the proof in full as it motivates the condition given in \hyperlink{A4}{(A4)}.

\begin{proof}

For $R>0$ we define a density $g_R$ by $g_R(x):=g(x)\mathbbm{1}_{[0,R]}(x)$ for $x\geq 0$. Let $\Lambda^R$ be the minimal solution to \eqref{alternativetomvr} with initial density $g_R$. Consider an arbitrary $T>0$ and suppose that there exists a $\tilde C$ such that $$\limsup_{R\rightarrow\infty}\Lambda_T^{R}\leq \tilde C.$$ 

Extending solutions to $[-1,T+1]$ as discussed in Section \ref{convergencestatementsection}, the set $\{\Lambda_t^{R} \textrm{ for }{R\geq 0}\}$ is compact. Therefore, for any sequences $R_m$ such that $R_m\uparrow\infty$ as $m\rightarrow\infty$, there is a subsequence $R_{m_k}$ such that $\Lambda^{R_{m_k}}$ converges under the Skorokhod M1 topology as $k\rightarrow\infty$. From a slight modification to the arguments of \cite[Proposition 2.1]{minimal}, we obtain that $\Lambda^{R_{m_k}}_t$
converges under the M1 Skorokhod topology to some process $\Lambda_t$ which satisfies
\begin{align}\label{a4sol}\Lambda_t=\int_0^\infty \mathbb{P}(\inf_{s\leq t}(x+B_s-\Lambda_s)\leq 0) g(x)dx.\end{align}

 Such a process $\Lambda_t$ is a solution to \eqref{alternativetomvr}, and therefore a minimal solution exists.

To determine the existence of a bound on $\Lambda_T^{\tilde g_R}$ over $R$, we suppose for a contradiction that there is some sequence $R_k$ such that $\Lambda_T^{ R_k}\rightarrow\infty$.

Since $g_{R_k}\in L^1(\mathbb{R}^+)$, there is a global, physical solution to \eqref{alternativetomvr} with initial density $g_{R_k}$. From Lemma \ref{Ito}, taking the limit $t\rightarrow\infty$, it follows that
    \begin{align*}
        &\int_0^\infty (1-g_{R_k}(x))e^{-\lambda x}dx\\
        &=\frac{\lambda}{2}\int_0^\infty e^{-\lambda \Lambda^{R_k}_{t^-}-\frac{\lambda^2t}{2}}dt+\sum_{\Delta \Lambda^{R_k}_t>0}e^{-\frac{\lambda^2t}{2}}\left(\int_{\Lambda^{R_k}_{t^-}}^{\Lambda^{R_k}_t}\left(-e^{-\lambda x}\right) V(t^-,x)dx-\frac{1}{\lambda}\Delta e^{-\lambda \Lambda^{R_k}_t}\right).\end{align*}

        By the physical jump condition, $\int_{\Lambda^{R_k}_{t^-}}^{\Lambda^{R_k}_{t^-}+x} V(t^-,y)dy\geq x$ for $x\leq \Delta \Lambda^{R_k}_t$. Therefore, since $e^{-\lambda x}$ is decreasing, each term in the sum is bounded above by $$-\int_{\Lambda^{R_k}_{t^-}}^{\Lambda^{R_k}_t} e^{-\lambda x}dx-\frac{1}{\lambda}\Delta e^{-\lambda \Lambda^{R_k}_t}=0.$$
        It follows that
         \begin{align*}
        &\int_0^\infty (1-g_{R_k}(x))e^{-\lambda x}dx
        \leq \frac{\lambda}{2}\int_0^\infty e^{-\lambda \Lambda^{R_k}_{t^-}-\frac{\lambda^2t}{2}}dt.\end{align*}
        
For arbitrary $\lambda>0$, consider the process $h(\lambda,R_k)_t=e^{-\lambda \Lambda^{R_k}_t}$. This is bounded between 0 and 1, and is monotone decreasing. For any integer $K$, we extend the process $(h(\lambda,R_k)_t)_{t\in [0,K]}$ to a process on $[-1,K+1]$ in the same manner as $\Lambda$ in Section \ref{convergencestatementsection}. Under this extension $\{h(\lambda,R_k)_t \textrm{ for }k\in\mathbb{N}\}$ is  compact under the M1 Skorokhod topology on $\tilde D([-1,K+1])$, and the sequence $h(\lambda,R_k)_t$ converges on a subsequence to some function $h^K(\lambda)_t$. This can repeated for each $K$, and applying a diagonal argument we may obtain a further subsequence of $h(\lambda,R_k)_t$ which converges Lebesgue almost everywhere to a cadlag function $h(\lambda)_t$. This function is bounded above by $1$ and by assumption $h(\Lambda)_t=0$ for $t\geq T$.

From the Dominated Convergence Theorem, it follows that

$$ \int_0^\infty ( 1-g(x))e^{-\lambda x}dx\leq \frac{\lambda}{2} \int_0^T h(\lambda)_te^{-\frac{\lambda^2t}{2}}dt.$$

Bounding $h(\lambda)_t$ above by 1, it follows that 
$$\int_0^\infty (1-g(x))e^{-\lambda x}dx\leq \frac{\left(1-e^{-\frac{\lambda^2T}{2}}\right)}{\lambda}.$$

Dividing both sides by $\lambda>0$, and then taking the limsup over $\lambda \downarrow 0$, we determine that $\frac{T}{2}\geq \infty$ which cannot hold. Therefore, the assumption that $\Lambda_T^{R_k}\rightarrow\infty$ is impossible. 
\end{proof}

\begin{remark}\label{rem:speeds}
Unlike for a density $g$ which satisfies condition \hyperlink{A2}{(A2)}, condition \hyperlink{A4}{(A4)} allows for solutions to \eqref{alternativetomvr} which grow with no linear bound.
For example, we may consider $h(x)=1$ on $[0,1]$ and $\frac{-\beta}{x^{\beta+1}}$ on $[1,\infty]$, for some $0<\beta<1$. We then set $g(x)=1-h(x).$ It holds that 
$$\int_0^\infty g(x)e^{-\lambda x}dx=\frac{1}{\lambda}-\Gamma(1-\beta)\lambda^{\beta}+O(\lambda)\textrm{ as }\lambda\rightarrow 0.$$

Here $\Gamma$ denotes the Gamma function, rather than the operator on $\tilde D$ discussed earlier.

The density $g$ therefore satisfies the condition \hyperlink{A4}{(A4)}, and therefore a global, minimal solution to the McKean-Vlasov type system \eqref{alternativetomvr} with initial density $g$, $\underline\Lambda$, exists. Arguing as in the proof of Theorem \ref{conda4}, we may determine that 
$$\int_0^\infty e^{-\lambda \underline{\Lambda}_t-\lambda^2 t/2}dt=\Gamma(1-\beta)\lambda^{\beta-1}+O(1)_{\lambda\rightarrow 0}.$$

By comparing the above integral to that of $e^{-\lambda t^{\zeta}}$, we determine that 
$$\limsup_{t\rightarrow\infty}\frac{\underline{\Lambda}_t}{t^{\frac{1}{1-\beta}}}>0$$  In particular, $\underline\Lambda$ may not be bounded by a linear function. It is also easily determined that $$\liminf_{t\rightarrow\infty}\frac{\underline{\Lambda}_t}{t^{\frac{1}{1-\beta}}}<\infty$$ and hence this barrier cannot grow faster than polynomially.
\end{remark}

\begin{remark}
For a global in time solution to \eqref{alternativetomvr} with finitely many discontinuities to exist, the condition \hyperlink{A4}{(A4)} must be satisfied by the density $g$. Considering the expression given in Lemma \ref{Ito}, and bounding the sum term above, we obtain that

       \begin{align*}
        &\frac{1}{\lambda }\int_0^\infty (1-g(x))e^{-\lambda x}dx\geq \frac{1}{2}\int_0^\infty e^{-\lambda \Lambda_{t}-\frac{\lambda^2t}{2}}dt-\sum_{\Delta\Lambda_t>0}(\Delta\Lambda_t)^2-{}O_{\lambda \downarrow 0}(\lambda).\end{align*}
   
Applying Fatou's lemma, it is simple to see the right hand side grows to infinity as $\lambda\downarrow 0$. 

\end{remark}

\section{Asymptotic barrier speed with initial discontinuities}
\label{asymptoticsection}
The condition \hyperlink{A2}{(A2)} on $g$ allows for finite time blow-ups. However, under certain conditions, the asymptotic barrier movement is unaffected by these blow-ups, and depends only on the behaviour of corresponding classical solutions to the supercooled Stefan problem \eqref{pde}.

\begin{theorem}
\label{speed}
    Suppose that the density $g$ satisfies conditions \hyperlink{A1}{(A1)}, \hyperlink{A2}{(A2)}, and is such that $g(x)\leq 1$ for $x\geq x^*$.
    
    Suppose further that $1-g(x)\in L^1(\mathbb{R}^+)$, and set $\frac{2}{K}=\int_0^\infty 1-g(x)dx$.
    Then for any physical solution to \eqref{alternativetomvr}, $\Lambda$, 
    $$ \frac{\Lambda_t}{t}\rightarrow K \textrm{ as }t\rightarrow\infty.$$ 
    
    Suppose instead that $g(x)\rightarrow a<1$ as $x\rightarrow\infty$. Then for any physical solution to \eqref{alternativetomvr}, $\Lambda$,
    $$\frac{\Lambda_t}{\sqrt{t}}\rightarrow K_a \textrm{ as } t\rightarrow\infty.$$
\end{theorem}

\begin{proof}
Let $\Lambda$ be a solution of \eqref{alternativetomvr}. Recall from Proposition \ref{6.1} that the measure given by the non-absorbed, re-centred particles has a density we denote $\tilde V(t,x)$. From the conditions on the density $g$, we may determine that there is some time $\tau^*$ after which $\tilde V(\tau^*,x)$ is strictly less than 1 and is continuous in $x$. This can be done by bounding the density above by ignoring the killing of particles, and by arguing as in \cite[Proposition 3.4]{3} to determine continuity of the density in $x$.
We consider restarting the system \eqref{alternativetomvr} with initial density $\tilde V({\tau^*},.)$. From Proposition \ref{minimal}, this has a unique solution, which therefore agrees with the unique solution to the supercooled Stefan problem \eqref{pde} with initial data $\tilde V({\tau^*},.)$. We denote this $\Lambda_t^{{\tau^*}}$. It follows that $\Lambda_{t+\tau^*}=\Lambda_t^{{\tau^*}}$, and therefore $\Lambda$ has identical asymptotic behaviour to $\Lambda^{{\tau^*}}$.

In the first case we may determine that $$\int_0^\infty 1-\tilde V({\tau^*},x)dx=\int_0^\infty 1-g(x)dx.$$

We rearrange the formula given in Lemma \ref{Ito}

to obtain that:
    \begin{align*}
        &\left|\int_0^\infty (1-g(x))e^{-\lambda x}dx-\int_{\Lambda_{\tau^*}}^\infty(1- V(\tau^*,x))e^{-\lambda x-\frac{\lambda^2\tau^*}{2}}dx\right|\\
        &=\left|\frac{\lambda}{2}\int_0^{\tau^*} e^{-\lambda \Lambda_{t^-}-\frac{\lambda^2t}{2}}dt+\sum_{\Delta \Lambda_t>0,t\leq \tau^*}e^{-\frac{\lambda^2t}{2}}\left(\int_{\Lambda_{t^-}}^{\Lambda_t}\left(-e^{-\lambda x}\right) V(t^-,x)dx-\frac{1}{\lambda}\Delta e^{-\lambda \Lambda_t}\right)\right|,\\
        &\leq \frac{\lambda \tau^*}{2}+\sum_{\Delta \Lambda_t>0,t\leq \tau^*}\int_{\Lambda_{t^-}}^{\Lambda_t}V(t^-,x)\left((e^{-\lambda\Lambda_{t^-}}-e^{-\lambda x})\vee (e^{-\lambda x}-e^{-\lambda \Lambda_t})\right)dx,\\
        &\leq \frac{\lambda \tau^*}{2}+C\lambda(\Lambda_{\tau^*})^2. \end{align*}
In the first inequality we have bounded $$(\Lambda_{t}-\Lambda_{t^-})e^{-\lambda \Lambda_{t}}\leq \frac{e^{-\lambda \Lambda_{t^-}}-e^{-\lambda \Lambda_t}}{\lambda}\leq (\Lambda_{t}-\Lambda_{t^-})e^{-\lambda \Lambda_{t^-}},$$
and recalled that $$\int_{\Lambda_{t^-}}^{\Lambda_t}V(t^-,x)dx=\Lambda_t-\Lambda_{t^-}.$$
In the final inequality we have used that $V(t^-,x)$ is bounded above by $C$, which follows from Proposition \ref{6.1}.

Sending $\lambda\downarrow 0$, the right hand side in the above inequality converges to 0. Since we assume $1-g(x)\in L^1(\mathbb{R}^+)$, and since $V(\tau^*,x)\leq 1$, we deduce that $1-V(\tau^*,x)\in L^1(\mathbb{R}^+)$, and that $\int_0^\infty 1-\tilde V(\tau^*,x)dx=2/K$.

From \cite[Theorem 5.3]{Ricci_Weiqing_1991} it then follows that 
$$\frac{\Lambda_t}{t}=\frac{\Lambda^{{\tau^*}}_{t-\tau^*}}{t}\rightarrow K \textrm{ as }t\rightarrow \infty.$$
In the second case, it is simple to verify that $V({\tau^*},x)\rightarrow a$ as $x\rightarrow\infty$. Applying \cite[Theorem 6.1]{Ricci_Weiqing_1991} we determine:
$$\frac{\Lambda_{t}}{\sqrt{t}}=\frac{\Lambda^{{\tau^*}}_{t-\tau^*}}{\sqrt{t}}\rightarrow K_a \textrm{ as }t\rightarrow\infty.$$
\end{proof}

\begin{remark}
The results of \cite[Theorem 6.1]{Ricci_Weiqing_1991} may be extended to other types of initial data allowing for different types of barrier speeds. In particular, for density $g$ which satisfies condition \hyperlink{A4}{(A4)} it should be possible to have $\Lambda_t/t^{\alpha}\rightarrow 1$ for any $\alpha\geq1/2$. However, to apply this for different barrier speeds we must identify initial data that corresponds to a given asymptotic barrier speed. 
From \cite[Section 4]{4}, we may identify that global in time solutions of the PDE
\eqref{pde} must satisfy:
$$\frac{\lambda}{2} \int_0^\infty e^{-\lambda \Lambda_t-\frac{\lambda^2 t}{2}}dt=\int_0^\infty (1-g(x))e^{-\lambda x}dx\quad\forall \lambda>0.$$
If this uniquely determined a continuous function $\Lambda$, then we could identify initial data giving a certain asymptotic speed from the Laplace transform of the data. We conjecture that such a result holds, however are presently unable to determine this.

\end{remark}

\section{Uniqueness}
\label{uniquenessappendix}
In this section we discuss some of the conditions for uniqueness of physical solutions to \eqref{alternativetomvr}.
\begin{proposition}
\label{uniquecont}
    Suppose that $g$ is bounded, measurable, and that a minimal solution, $\underline{\Lambda}$, to \eqref{alternativetomvr} exists on some time interval $[0,T)$. Suppose further that at any time $T'$ where $\underline{\Lambda}$ is discontinuous, there exists a positive $\epsilon_{T'}>0$ such that the minimal solution is continuous on $(T',T'+\epsilon_{T'})$. Then, the physical solution to \eqref{alternativetomvr} is unique.  
\end{proposition}

This is a simple consequence of the Laplace transform representation result of Lemma \ref{Ito}, which follows by comparing the formula given in Lemma \ref{Ito} for a solution $\Lambda_t$, to the formula for the minimal solution, and noting that $e^{-\lambda \underline{\Lambda}_t}-e^{-\lambda\Lambda_t}$ is non-negative.
The details are given in Section \ref{uniquenessappendix2}.

\begin{corollary}
\label{contunique}
Suppose that $g\leq 1$. Then the physical solution to \eqref{alternativetomvr} is continuous on $(0,\infty)$, and is unique.
\end{corollary}

\begin{proof}
While a physical solution can jump at time $0$, it is simple to check that for any $t>0$, the condition \hyperlink{W}{(W)} holds for the restarted density $\tilde V(t,x)$. Therefore, from the result of Corollary \ref{holdercont}, it holds that $\Lambda_t$ is continuous on $(0,\infty)$. The condition of Proposition \ref{uniquecont} then holds, which implies uniqueness of the physical solution.
\end{proof}

\begin{corollary}
\label{contunique3}
    Suppose that there exists $\epsilon>0$ such that $g(x)\leq 1$ for $x<\epsilon$ and that $\inf\{x:\int_0^xg(y)dy<x\}=0$. Then, there exists $T'>0$ such that there is a unique physical solution $\Lambda_t$ on $[
    0,T']$. This solution is continuous. 
\end{corollary}

From simple computations it can be checked that for small $t>0$, for any solution $\Lambda_t$, the density $\tilde V(t,x)$ must be strictly less than $1$ for sufficiently small $x$. Therefore $\Lambda_t$ must be continuous for all sufficiently small $t>0$, from which uniqueness follows via the result of Proposition \ref{uniquecont}. Details of the computations are given in Section \ref{uniquenessappendix2} in the appendix.

\begin{proposition}
\label{contunique2}
    Suppose that $\int_{0}^xg(y)dy\leq  x$ for all $x>0$, and that a physical solution to \eqref{alternativetomvr} exists. Then the physical solution to \eqref{alternativetomvr} is unique.
\end{proposition}

The condition given in the statement of this Proposition is nearly that of \cite[Theorem 1.1]{PDEresult}. By considering approximations of the density, we are able to check that for the minimal solution to \eqref{alternativetomvr}, the sum term in the formula of Lemma \ref{Ito} is zero. With this established we can compare this formula to that of any physical solution, from which uniqueness follows by noting that $e^{-\lambda \underline{\Lambda}_t}-e^{-\lambda\Lambda_t}$ is non-negative, as in Proposition \ref{uniquecont}. The details are given in Section \ref{uniquenessappendix2} of the appendix.

\begin{remark}
    The arguments of Proposition \ref{contunique2} do not remove the possibility of a jump in $\Lambda_t$ at time zero. However, it must hold that $\tilde V(t^-,x)=1$ for almost every $x\leq \Delta \Lambda_t$ at any jump time. It can be checked that for fixed $t>0$, $\tilde V(t^-,x)$ is analytic on $(0,\infty)$. Therefore, the given condition for a jump in $\Lambda_t$ can occur only at time zero.
    
    The arguments used in the proof of Proposition \ref{contunique2} may be applied to any density $g$ for which there exists a sequence $z_n\downarrow 0$, such that $\int_{z_n}^xg(y)dy<x$ for all $x>0$. For example, the oscillatory density $g(x)=(1+(\sin(\frac{1}{x}))_-+\alpha (\sin(\frac{1}{x}))_+)\mathbbm{1}_{[0,\frac{1-\alpha}{4\pi}]},$ for any $\alpha<1$, with corresponding sequence $z_n=\frac{1}{(2n+1)\pi}$, and therefore the physical solution to \eqref{alternativetomvr} with such initial density is unique.
\end{remark}

\begin{remark}
The result of Corollary \ref{contunique} implies that of \cite[Theorem 1.1]{mustapha2023wellposednesssupercooledstefanproblem}. Further the condition of Proposition \ref{contunique2} holds for the data studied in \cite[Sections 4]{mustapha2023wellposednesssupercooledstefanproblem}. The condition of Proposition \ref{uniquecont} holds on a small time interval for solutions to \eqref{alternativetomvr} with an initial density which changes monotonicity finitely many times on compact sets. 

In \cite[Theorem 1.4]{3}, it is shown that for $g\in L^1(\mathbb{R}^+)$, such that $g$ changes monotonicity finitely often on compact sets, then the same holds for $\tilde V(t,x)$ at any time. Therefore global uniqueness of physical solutions to \eqref{alternativetomvr} is obtained. We expect the same to hold for $g\notin L^1(\mathbb{R}^+)$, and are able to extend the 
regularity results of \cite[Section 2]{3}, and of \cite{2} to the case of $g\notin L^1(\mathbb{R}^+)$. However, it is presently unclear to us whether the results of \cite[Proposition 3.2, Lemma 3.7]{3} which rely upon results from \cite{DIRT1} and \cite{delarue:hal-00870991} readily extend to $g\notin L^1(\mathbb{R}^+)$.
\end{remark}

\begin{remark}
 Non-uniqueness of physical solutions to \eqref{alternativetomvr} on a small time interval can only occur if there is a sequence of times $t_n\downarrow 0$ such that $\Delta \underline{\Lambda}_{t_n}>0.$ This must occur for any initial density $g$ which satisfies that for any $t>0$, there exists a sequence $\lambda_n$ increasing to infinity such that $ \int_0^\infty (1-g(x))e^{-\lambda_n x}dx>0$ for odd $n$, and $ \int_0^\infty (1-g(x))e^{-\lambda_n x}dx<-\frac{Ce^{-\frac{\lambda_n^2t}{2}}}{\lambda_n}$ for even $n$. If there was not a sequence of jump times decreasing to zero, we could consider the expression given in Lemma \ref{Ito} for any time $t$ less than the first jump time. Multiplying by $e^{\lambda\Lambda_t+\frac{\lambda^2t}{2}}$, we would observe the left hand side is less than zero infinitely often as $\lambda \uparrow \infty$, while the right hand side grows to infinity, which is not possible. 
\end{remark}

\begin{remark}
    Uniqueness of the particle system limit can be determined for a particle system perturbed upwards, identically to that in \cite{minimal}, without requiring uniqueness of solutions to \eqref{alternativetomvr}. The result of \cite[Proposition 6.1]{minimal} extends to our setting, by replacing the bound from the Dvoretzky–Kiefer–Wolfowitz inequality by Poisson tail bounds.
\end{remark}

\section{Proof of general particle system convergence}
\label{proofsection}
We first establish the tightness of the processes $\mu^N,\Lambda^N$, and the convergence of the process $\upsilon^N$.

 Recall that we define a continuous map from $\mathbb{R}\times \mathcal{C}$ with the product topology to $\mathcal{C}$ by
$\theta(x,y)_t=x+y_t,$ we denote by $\mathcal{L}(B)$ the law of the extended Brownian motion on $\mathcal{C}$, and we write $G$ for the measure of density $g$ on $\mathbb{R^+}$.
\begin{lemma}
\label{lemma 2.2}
 The measure $\upsilon^1$ is a Poisson random measure with mean $\theta_{*}(G\times\mathcal{L}(B))$ .
\end{lemma}
This is a simple consequence of the fact that the particles are initially placed via a Poisson point process of density $g$, and move via Brownian motions, independent of this Poisson point process. 

Applying the Strong Law of Large Numbers, and the result of the above lemma, we obtain convergence of $\upsilon^N$ to a deterministic measure.
\begin{proposition}
\label{convergence of Hn}
    $\upsilon^N\rightarrow \theta_{*}(G\times\mathcal{L}(B))$ vaguely, in distribution.
\end{proposition}
\begin{proof}
    From \cite[Theorem 4.11]{kallenberg} it is enough to show that $\upsilon^N(f)\rightarrow \theta_{*}(G\times\mathcal{L}(B))(f)$ in distribution for $f$ continuous, bounded, and with bounded support. Using the representation of our system as \eqref{aligned:Particlewarticle}, we can see that
    $$\upsilon^N(f)=\frac{1}{N}\sum_{i=1}^N(\sum_{k=1}^\infty f(W^{i,k})).$$
    This is a sum of $N$ i.i.d. terms with distribution equal to $\upsilon^1(f)$. Further each term has finite expectation. Indeed since $f$ has bounded support it is supported on $\{x:|x_{-1}|\leq K\}$ for some finite $K$. Therefore:
    $$\mathbb{E}(|\sum_{k=1}^\infty f(W^{i,k})|)\leq \mathbb{E}(\upsilon^1(|f|))\leq ||f||_{\infty}\mathbb{E}(\upsilon^1(\{x:|x_{-1}|\leq K\}))=||f||_\infty\int_0^Kg(x)dx.$$
    
    By the Strong Law of Large Numbers, it follows that $$\upsilon^N(f)\rightarrow \mathbb{E}(\upsilon^1(f))$$ almost surely. Since $\upsilon^1$ is a Poisson random measure of mean $\theta_{*}(G\times \mathcal{L}(B))$, $\mathbb{E}(\upsilon^1(f))=\theta_{*}(G\times \mathcal{L}(B))(f)$ as desired.
\end{proof}

Tightness of $\Lambda^N$ follows from the definition of a pre-compact set in the M1 topology, and the result Lemma \ref{lemma on bounding}. 
\begin{lemma}
\label{tightnessof xi}
$\Lambda^N$ is tight, when viewed as a process on $\tilde D$.
\end{lemma}
\begin{proof}
This follows from identical arguments to Lemma \ref{Rtight}, replacing the bound on $R^N$ by the bound from Lemma \ref{lemma on bounding}.
\end{proof}

Another consequence of Lemma \ref{lemma on bounding}, is a bound on any subsequential limit of $\Lambda^N$.
 \begin{corollary}
 \label{boundedness of xi limit}
     There exists a constant $z$ such that for any subsequential limit $ \Lambda^{N_k}\rightarrow\Lambda$ in distribution, $\Lambda_{T+1}\leq z-1$ almost surely. 
 \end{corollary}

Since $\tilde h$ is continuous, and $\mu^N=\tilde h(\upsilon^N,\Lambda^N)$, applying the Skorokhod Representation Theorem the tightness of $\mu^N$ follows.
\begin{corollary}
    $ \mu^N$ is tight. Further, for any subsequential limit $(\upsilon,\Lambda,\mu)$ of $(\upsilon^N,\Lambda^N,\mu^N)$, it holds that $\mu=\tilde h(\upsilon,\Lambda)$ almost surely.
    \label{subseq}
\end{corollary}

\subsection{Weak convergence on initially bounded sets}
\label{7.1}
We have shown that $(\upsilon^N,\mu^N,\Lambda^N)$ is tight. Therefore, from now on we may consider only subsequences such that $\upsilon^N,\Lambda^N,\mu^N$ converge in distribution. We omit the subsequence notation. By the Skorokhod Representation Theorem we may take a representation such that $\hat \upsilon^N\rightarrow \hat \upsilon$ vaguely, $\hat \mu^N\rightarrow \hat \mu$ vaguely, and $\hat \Lambda^N\rightarrow \hat \Lambda$ in the M1 Skorokhod topology, almost surely. From now on we work with these representations.

Recall the notation of Section 3.1 for measures restricted to a set. We first determine the weak convergence of the measure, restricted to particles started on $[0,K]$.
\begin{lemma}
\label{weak convergence of inital}
      $\hat \upsilon^N_{S_K}\rightarrow \hat \upsilon_{S_K}$ weakly.
\end{lemma}

To prove this result we first note that it is simple to determine weak convergence of $\upsilon^N_B$ for a bounded set. 
\begin{lemma}
\label{acor}
    $\hat \upsilon^N_{B_R(0)}\rightarrow \hat \upsilon_{B_R(0)}$ for any positive $R$.
\end{lemma}
This result follows from elementary arguments and is omitted.

\begin{proof}[Proof of Lemma \ref{weak convergence of inital}]
From the above lemma, it holds that $\hat \upsilon^N_{B_R(0)}\rightarrow \hat \upsilon_{B_R(0)}$ weakly for any $R$. We omit the subscript $K$ on the set $S_K$. It is easily seen $N\hat \upsilon^N(S\cap B_R(0)^c)$ has the same distribution as the number of particles for the system \eqref{aligned:Particlealphanocentre} which are initially below $K$, and which leave $[-R,R]$ by time $T+1$.
For the $ith$ particle started in $[0,K]$ to leave $[-R,R]$ by time $T+1$, it must hold that $\sup_{s\leq t} |B_s^i|\geq R-K$. The probability of such an event may be bounded by using the reflection principle and a standard Gaussian tail bound by $$8\exp\left(\frac{-(R-K)^2}{2(T+1)}\right).$$ From condition \hyperlink{A1}{(A1)}, the number of particles initially in $[0,K]$ is stochastically dominated by a Poisson random variable of mean $CNK$. Since the Brownian motions are independent of the initial conditions, from the thinning property of Poisson random variables $N\hat \upsilon^N(S\cap B_R(0)^c)$ can then be stochastically dominated by a Poisson random variable of mean $8NCK\exp\left(\frac{-(R-K)^2}{2(T+1)}\right)$. In particular, writing $n(\alpha)$ for a Poisson random variable of mean $\alpha$, we may couple $\hat \upsilon^N(S\cap B_R(0)^c)$, with such a Poisson random variable such that almost surely:
    $$\hat \upsilon^N(S\cap B_R(0)^c)\leq \frac{n(8NCK\exp\left(\frac{-(R-K)^2}{2(T+1)}\right)}{N}. $$
    By the Strong Law of Large Numbers it follows that
\begin{equation}
\limsup_{N\rightarrow\infty}\hat \upsilon^N(S\cap \hat B_R(0)^c)\leq 8CK\exp\left(\frac{-(R-K)^2}{2(T+1)})\right) \label{eq:limsupbd}
\end{equation}
The above result holds for arbitrary $K$, and $R\geq K$ hence also holds if replacing $R,$ $K$ by $R-\epsilon,$ $ K+\epsilon$ for arbitrary small $\epsilon$.
    
    Take $f$ to be a continuous and bounded function. 
    \begin{align*}
    \limsup_{N\rightarrow\infty}|\int f d\hat \upsilon^N_S-\int f d\hat \upsilon_S|\leq &\left|\int f d(\hat \upsilon^N_S-\hat \upsilon^N_{S\cap B_R(0)})\right|\\&+\left|\int fd(\hat \upsilon^N_{S\cap B_R(0)}-\hat \upsilon_{S\cap B_R(0)})\right|+\left|\int fd(\hat \upsilon_{S\cap B_R(0)}-\hat \upsilon_S)\right|,\\
    \end{align*}
    \begin{align}
    &\leq \limsup_{N\rightarrow\infty}||f||_\infty(\hat \upsilon(S\cap B_R(0)^c)+\hat \upsilon^N(S \cap B_R(0)^c))\label{eqtndirectlyabove}\\
    &\quad \quad+ \left|\int fd(\hat \upsilon^N_{S\cap B_R(0)}-\hat \upsilon_{S\cap B_R(0)})\right|.\nonumber
    \end{align}
    We write $S^\epsilon$ for the set $\{y:\exists x\in S, d(x,y)<\epsilon\}$. For any $\epsilon>0$:
    \begin{align*}&\hat \upsilon(S\cap B_R(0)^c)\leq \lim_{k\rightarrow\infty}\hat \upsilon((S^\epsilon\cap B_{R-\epsilon}(0)^c)^o\cap  B_k(0)),\\
    &\leq \lim_{k\rightarrow\infty}\limsup_{N\rightarrow\infty} \hat \upsilon^N(S^\epsilon \cap B_{R-\epsilon}(0)^c\cap B_k(0))\leq \limsup_{N\rightarrow\infty} \hat \upsilon^N(S^\epsilon \cap B_{R-\epsilon}(0)^c).
    \end{align*}
    This follows from the vague convergence of $\hat \upsilon^N\rightarrow \upsilon$ since
    $(S^\epsilon\cap B_{R-\epsilon}(0)^c)^o\cap  B_k(0)$ is open and bounded for each $k$. 
    From the inequality \eqref{eq:limsupbd}, it follows that
    $$\limsup_{N\rightarrow\infty}||f||_\infty(\hat \upsilon(S\cap B_R(0)^c)+\hat \upsilon^N(S \cap B_R(0)^c))\leq 16||f||_\infty C(K+\epsilon)\exp\left(\frac{-(R-K-2\epsilon)^2}{2(T+1)}\right) .$$
    
   From Lemma \ref{acor}, $\hat \upsilon^N_{B_R(0)}$ converges weakly to $\hat \upsilon_{B_R(0)}$. Since $S$ is closed, it follows from the Portmanteau Theorem that $\hat \upsilon^N_{B_R{(0)}\cap S}$ converges weakly to $\hat \upsilon_{B_R(0)\cap S}$.
    Sending $N\rightarrow\infty$, then $R\rightarrow \infty$ in equation \eqref{eqtndirectlyabove}, we have thus established:
     $$\limsup_{N\rightarrow\infty}\left|\int f d\hat \upsilon^N_S-\int f d\hat \upsilon_S\right|=0$$
     Since $f$ was an arbitrary continuous, bounded function, this gives the desired weak convergence.
\end{proof}
\begin{corollary} 
    Consider $G_K$ the measure on $\mathbb{R}$ with density $g\mathbbm{1}_{[0,K]}$.
    $$\hat \upsilon^N_{S_K}\rightarrow\theta_{*}(G_K\times\mathcal{L}(B)) \textrm{ weakly, almost surely.}$$
\end{corollary}
\begin{proof}
From Proposition \ref{convergence of Hn}, $\hat \upsilon^N\rightarrow \hat \upsilon=\theta_{*}(G\times\mathcal{L}(B))$ almost surely. Applying Corollary \ref{weak convergence of inital}, we determine that $\hat \upsilon^N_{S_K}\rightarrow \hat \upsilon_{S_K}$ weakly. We then note that the restriction of $\theta_{*}(G\times\mathcal{L}(B))$ to the set $S_K$, is exactly given by $\theta_*(G_K\times \mathcal{L}(B))$.
\end{proof}
As a consequence of this result, and noting that $\hat \mu^N_{S_K}=\tilde h(\hat \upsilon^N_{S_K},\hat \Lambda^N)$, we determine weak convergence of $\hat \mu^N_{S_K}$. 
\begin{corollary}
   Almost surely, $\hat \mu^N_{S_K}\rightarrow \hat \mu_{S_K}$ weakly.
\end{corollary}

In the following sections, we apply the Skorokhod Representation Theorem to the measures $\hat \upsilon_{S_K},\hat \mu_{S_K}$. To make this clear, we normalise the measures. For a measure $\varrho$ in $M(\tilde D)$ or $M(\mathcal{C})$, we define $\varrho_K$ as $\varrho_{S_K}$ normalised to be a probability measure:
    
$$\varrho_K:=\frac{\varrho_{S_K}}{\varrho(S_K)}.$$
\begin{corollary}
    Almost surely, for any integer $K$, $\hat \upsilon^N_K\rightarrow \hat \upsilon_K$ weakly, $\hat \mu^N_K\rightarrow \hat \mu_K$ weakly and 
        $\hat \upsilon(x_{-1}\leq K)=\hat \mu(x_{-1}\leq K)=\int_0^Kg(x)dx$.
\end{corollary}

\subsection{Identification of the limit}

We define the following hitting time map $\tau:\tilde D\rightarrow \mathbb{R}^+$:
$$\tau(x):=\inf\{t\geq0:x_t\leq 0\},\;\;x\in\tilde D$$
We also define a map $\lambda:\tilde D\rightarrow \tilde D$ by
$$\lambda_t(x):= \mathbbm{1}_{\{\tau(x)\leq t\}}, \;\;x\in\tilde D,$$
and for a measure $\rho$, write $\rho(\lambda)$ for the cadlag process $\rho(\lambda)_t=\rho(\lambda_t)$.
These are analogous to maps defined in \cite{minimal}.
In order to show that the limiting $(\hat \mu,\hat \Lambda)$ solves the McKean-Vlasov type equation \eqref{alternativetomvr}, it is necessary to determine that $$\hat \Lambda_t=\lim_{N\rightarrow\infty }\hat \mu^N(\tau\leq t)=\hat \mu(\tau\leq t).$$
Equivalently, that the mapping $\rho\rightarrow \rho(\lambda)$ is continuous at $\hat \mu$. To do this we consider the set of paths which possess the following crossing property:
$$A_\textrm{Crossing}:=\{x\in\tilde D: \forall h>0,\textrm{ }\tau(x)<T+1\implies \inf_{s<h\wedge T+1-\tau(x)}(x_{\tau(x)+s}-x_{\tau(x)})<0\}.$$

Such a property holds for $\hat \mu$ almost every path.
\begin{lemma}
   $\hat \mu(A_\textrm{Crossing}^c)=0$ almost surely.  
    \label{crossinglemma}
\end{lemma}

\begin{proof}
    Almost surely, it holds that $$\hat \mu(A_\textrm{Crossing}^c)=\tilde h(\theta_*(G\times \mathcal{L}(B)),\hat \Lambda)(A^c_{\textrm{Crossing}}).$$

    By the definition of the mappings $\tilde h,$ $\theta$, under the measure $\tilde h(\theta_*(G\times \mathcal{L}(B))$, $x_t-x_0+\hat \Lambda_t$ is a Brownian motion. Arguing from the properties of the Brownian motion as in \cite[Lemma 5.9]{delarue2015particle}, it follows that 
    $$\tilde h(\theta_*(G\times \mathcal{L}(B),\hat \Lambda)_{K}(A_\textrm{Crossing}^c)=0 \quad \forall K>0, \textrm{ almost surely.}$$
    Taking a monotone limit over $K\uparrow \infty$, we establish that $\hat \mu_K(\textrm{Crossing}^c)=0$ almost surely.
\end{proof}

To compare $\hat \mu_{S_K}(\tau\leq t)$ to $\hat \mu(\tau\leq t)$, we give a bound on the number of particles started above $K$, and which reach $K$ by a time $T$.
\begin{lemma}
\label{lemma 1.1}
   Let $L^N(T,K)$ be the number of particles started above $K$ which reach $K$ by time $T$, for the particle system given by equation \eqref{aligned:Particlealphanocentre}. The random variable $L^N(T,K)$ is Poisson distributed with mean bounded by $\sqrt{\frac{2T}{\pi}}CN$.
\end{lemma}
\begin{proof}
    We may identify that $L^N(T,K)$ is Poisson distributed since the Brownian motions are independent of the initial Poisson point process. 
    To bound the mean, we note that $L^N(T,K)$ for the case of density $g$ may be stochastically dominated by $L^N(T,K)$ for the case of $g=C\mathbbm{1}_{\mathbb{R}^+}$, due to the bound on $g$ given by \hyperlink{A1}{(A1)}. Using the rescaling $\xi_t=Na\Lambda^N_{t/N^2a^2}$, it follows that $L^N(T,K)$ has the same distribution as the number of particles started above $NCK$, reaching $NCK$ by time $N^2C^2T$, for the system considered in \cite{7}. From the computations in \cite[Lemma 4]{7} the result follows.
\end{proof}

With the above results, we now prove the convergence of $\hat \mu^N(\lambda)$ to $\hat \mu(\lambda)$.

\begin{lemma}
\label{lemmac}
    For the representations $\hat \mu^N$, $\hat \mu$,
    $$\hat \mu^N(\lambda)\rightarrow \hat\mu(\lambda) \textrm { as elements of }\tilde D\textrm{ almost surely as }N\rightarrow\infty.$$
\end{lemma}

\begin{proof}
   
Fix an integer $K$. Consider the system of $W^i$ particles. For a particle started above $K$ to hit $\Lambda^N$ by time $T$, then either $\Lambda^N_T>z$, or the particle must first hit $K$, and then it must further hold that $\inf_{T_K^i \leq t\leq T+1} B_t^{i,N}-B^{i,N}_{T_{K}^{i,N}}<-K+z$. Here $T_K^i=\inf\{t\geq 0:W^i_t\leq K\}$ is the time at which the particle first hits $K$. We write $A_K^i$ for the second event described above:
$$A_K^{i,N}=\left\{\inf_{T_K^{i,N} \leq t\leq T+1} B_t^{i,N}-B^{i,N}_{T_{K}^{i,N}}<-K+z,\inf_{t\leq T+1} X^{i,N}_t\leq K, X_{-1}^{i,N}\geq K\right\}.$$

Since $\mu^N(\tau(x)\leq T+1,|x_{-1}|>K)\leq \mu^N(\tau(x)\leq T+1)\leq\Lambda^N_{T+1}=\Lambda^N_T$, it follows that
$$\mu^N\left(\tau(x)\leq T+1,|x_{-1}|>K\right)\leq  \Lambda^N_T \mathbbm{1}_{ \Lambda^N_T>z}+\frac{1}{N}\sum_{i=1}^\infty \mathbbm{1}_{A^{i,N}_K}.$$

We may obtain that $\sum_{i=1}^N \mathbbm{1}_{A_K^{i,N}}$ may be stochastically dominated by a Poisson random variable of mean $2C\sqrt{\frac{2T}{\pi}}N\exp(-\frac{(K-z)^2}{2(T+1)})$ in the following manner.

Since the particles are initially placed via a Poisson point process, independent of the Brownian motions, it may be seen that $\sum_{i=1}^\infty\mathbbm{1}_{A_K^i}$ is Poisson distributed. From Lemma \ref{lemma 1.1} we know that the number of particles started above $K$, which reach $K$ by time $T$ is Poisson distributed with mean bounded by $\sqrt{\frac{2T}{\pi}}NC$. Consider the event $\{\inf_{T_K^{i,N} \leq t\leq T_K^{i,N}+T+1} B_t^{i,N}-B^{i,N}_{T_{K}^{i,N}}<-K+z\}$. By the strong Markov property of the Brownian motion, this event is independent of $\mathcal{F}_{T_K^i}$, the natural filtration of $(B^i)_{i\in\mathbb{N}}$ up until the hitting time $T_K^i$. By standard Gaussian tail bounds this event has probability bounded above by $2\exp(-\frac{(K-z)^2}{2T})$. Therefore, using the thinning property of the Poisson we may bound the number of particles started above $K$ which reach $z$ by time $T$ by a Poisson random variable of mean $2CN\sqrt{2T/\pi}\exp(-\frac{(K-z)^2}{2T})$. This gives the bound on $\sum_{i=1}^\infty \mathbbm{1}_{A_K^{i,N}}$.

Recall that $\hat \upsilon_K\rightarrow \hat \upsilon_K$ weakly, $\hat \Lambda^N\rightarrow \hat \Lambda$, and $\hat\mu^N_K\rightarrow \hat\mu_K$ weakly, where $\hat\mu_K$ satisfies the crossing property almost surely. 
From the arguments of \cite[Proposition 5.8]{delarue2015particle}, it follows that there exists some co-countable $J_K$ such that almost surely $\hat \mu^N_K(\lambda)_t\rightarrow \hat \mu_K(\lambda)_t$ for all $t\in J_K$, in particular $\hat\mu_K(x_t=x_{t^-},\textrm{ }t\in J_K)=1$ almost surely. Setting $J$ to be the co-countable set $\cap_{K\in\mathbb{N}} J_K$ , then almost surely $\hat \mu^N_K(\lambda)_t\rightarrow \hat \mu_K(\lambda)_t$ for all $K$, for all $t$ in $J$. This convergence also holds for $\hat \mu^N_{S_K}(\lambda)$.

We may bound $\hat \mu^N(\lambda)_t$ for any integer $K$:
$$\hat \mu^N(\lambda)_t\leq \hat \mu_{S_K}^N(\lambda)_t+\hat \mu^N(\tau(x)\leq {T+1},|x_{-1}|>k)$$ 
Almost surely the first term converges to $\hat \mu_{S_K}(\lambda)_t\leq \hat \mu(\lambda)_t$ at any $t\in J$. Using that $\hat \mu,\hat \Lambda$ have equal distributions to $\mu,\Lambda$, from the bound on $\sum_{i=1}^\infty \mathbbm{1}_{A_K^{i,N}}$ from above, the second term may be bounded by $1/N$ times a Poisson random variable of mean $2C\sqrt{\frac{2T}{\pi}}N\exp\left(-\frac{(K-z)^2}{2(T+1)}\right)$, on the event that $\hat \Lambda^N_{T+1}\leq z$. By Corollary \ref{boundedness of xi limit}, for a large enough $z$ it holds that $\hat \Lambda_{T+1}\leq z-1$ almost surely. Since $\hat \Lambda^N_{T+1}\rightarrow \hat \Lambda_{T+1}$ almost surely, for large enough $N$ it follows that $\hat \Lambda^N_{T+1}\leq z$. Thus, from the Strong Law of Large Numbers for Poisson random variables, almost surely:
$$\limsup_{N\rightarrow\infty} \hat \mu^N(\tau(x)\leq t,x_{-1}>K)\leq 2C\sqrt{\frac{2T}{\pi}}\exp\left(-\frac{(K-z)^2}{2(T+1)}\right).$$

Sending $K\rightarrow \infty$, we thus establish that $\limsup_{N\rightarrow\infty}\hat \mu^N(\lambda)_t\leq \hat \mu(\lambda)_t$ for any $t\in J$ almost surely.

To bound in the other direction, we note that almost surely, for any $t$ in $J$:
$$\liminf_{N\rightarrow\infty}\hat  \mu^N(\tau\leq t)\geq \liminf_{N\rightarrow\infty} \hat\mu^N_{S_K}(\tau\leq t)=\hat \mu_{S_K}(\tau\leq t).$$

Taking the limit as $K\uparrow\infty$, and applying the Monotone Convergence Theorem, we determine that
$$\liminf_{N\rightarrow\infty} \hat \mu^N(\tau\leq t)\geq \hat \mu(\tau \leq t).$$

Therefore, almost surely, for all $t$ in $J$:
$$\hat \mu^N(\lambda)_t\rightarrow \hat\mu(\lambda)_t\textrm{ as }N\rightarrow\infty.$$

The process $\hat \mu^N(\lambda)$ is non-decreasing, and converges almost everywhere to $\hat \mu(\Lambda)$. This implies the M1 convergence of $\hat \mu^N(\lambda)$ to $\hat \mu(\lambda)$. 
\end{proof}

Since $\hat \Lambda^N=\hat \mu^N(\lambda)$ for each $N$, we immediately obtain the following corollary.
\begin{corollary}
\label{corollaryonxi}
   $\hat \mu(\lambda)_t=\hat \Lambda_t$ for $t\leq T$ almost surely.
\end{corollary}

\begin{proposition}
\label{proplimit}
    Almost surely $(\hat \mu,\hat\Lambda)$ form a solution to the McKean-Vlasov equation \eqref{alternativetomvr}, up to time T.
\end{proposition}

\begin{proof}
     Recall from Corollary \ref{subseq} that  $\hat\mu=\tilde h(\hat \upsilon,\hat \Lambda)$ almost surely. Further, from Proposition \ref{convergence of Hn}, it holds that $\hat \upsilon=\theta_*(G,\mathcal{L}(B))$ almost surely. Therefore, $\hat \mu=\tilde h(\theta_*(G,\mathcal{L}(B)),\Lambda)$.
   
We consider the following system 
\begin{align*}   
&W_0^i \textrm{ are placed via a Poisson point process with density g}, \\
&B^i\textrm{ are independent extended Brownian motions, independent of the }W_0^i,\\
&{} X_t^i=X_{0}^i+B_t^i-\Lambda_t,\\
&\tau_i=\inf\{t\geq 0:X_t^i\leq 0\}.
\end{align*}

It is clear from the definition of the maps $\tilde h,\theta$ that for any $A\in\mathcal{B}(\tilde D)$,\\ $\hat \mu(A)=\tilde h(\theta_*(G,\mathcal{L}(B)),\hat \Lambda)(A)=\mathbb{E}(\sum_{i=1}^\infty\mathbbm{1}_{X^i\in A})$ . Further, from Corollary \ref{corollaryonxi}, it follows that $\hat \Lambda_t=\hat\mu(\lambda)_t$ for $t<T$. By definition of the maps $\tilde h,\theta$ it is then clear that $\hat\mu(\lambda)_t=\mathbb{E}(\sum_{i=1}^\infty\mathbbm{1}_{\tau_i\leq t})$. Therefore, the pair $(\hat\mu,\hat\Lambda)$ may be seen to form a solution to \eqref{aligned:MVR}. Since the system \eqref{aligned:MVR}, and \eqref{alternativetomvr} are equivalent this gives the desired result.

\end{proof}

\subsection{Physicality of the limit}\label{physical section}

To conclude the proof of Theorem \ref{maintheorem}, we must further show that almost surely any limit $\hat \mu$ is also a physical solution to the McKean-Vlasov equation \eqref{alternativetomvr}.

Recall $\nu^N$ is the empirical measure of the re-centred particles which have not yet been absorbed:
$$\nu^N_{t-}=\frac{1}{N}\sum_{i=1}^\infty a_{X_{t-}^{i,N}}\mathbbm{1}_{t\geq \tau_{i}} $$

We consider an equivalent lemma to \cite[Lemma B8.2]{minimal}. 
\begin{lemma}
Define a set $A^{N,\epsilon}$ by:
$$A^{N,\epsilon}:=\{\nu^N_{t^-}([0, x+3\epsilon^{1/3}])\geq x\quad\forall x\leq (\Lambda^N_{t+\epsilon}-\Lambda^N_{t^-})-2\epsilon^{1/3}\}$$
There exists a constant $\tilde L>0$ such that for all sufficiently small $\epsilon>0$, and sufficiently large $N>N_\epsilon$:
    $$\mathbb{P}\left(A^{N,\epsilon}\right)\geq 1-\tilde L\epsilon^{1/6}$$
    \label{lem:physparticleestimate}
\end{lemma}

The proof follows from a modification of the arguments in \cite[Lemma B8.2]{minimal} and is given in section \ref{physicalappen} in the appendix.

\begin{proposition}
\label{physi}
    Almost surely $\hat\Lambda$ satisfies the physicality condition \eqref{phys3}.
\end{proposition}

This follows from adapting the proof of \cite[Theorem 6.4]{minimal} to the case of a non integrable density. This is done by considering the restriction of measures to sets $S_K$, then taking a limit over $K$, analogously to the proof of Lemma \ref{lemmac}. We provide a full proof in Section \ref{physicalappen} of the appendix.

Combined with Proposition \ref{proplimit}, this concludes the proof of Theorem \ref{maintheorem}.

\section*{Funding}

This publication is based on work supported by the EPSRC Centre for Doctoral
Training in Mathematics of Random Systems: Analysis, Modelling and Simulation (EP/S023925/1). D.G.M. Flynn is funded by the Charles Coulson Scholarship. 
For the purpose of open access, the author has applied a CC BY public
copyright licence to any author accepted manuscript arising from this
submission.

\bibliography{bible}

\begin{thebibliography}{10}

\bibitem{BHJ2025}
{\sc G.~Baker, B.~Hambly, and P.~Jettkant}, {\em Particle systems and {M}c{K}ean-{V}lasov dynamics with singular interaction through local times}.
\newblock ArXiv:2503.08837, 2025.

\bibitem{zeroundercool}
{\sc G.~Baker and M.~Shkolnikov}, {\em {Zero kinetic undercooling limit in the supercooled Stefan problem}}, Ann. Inst. Henri Poincar\'e Probab. Stat., 58 (2022), pp.~861 -- 871.

\bibitem{PDEresult}
{\sc E.~Bayraktar, G.~Guo, W.~Tang, and Y.~P. Zhang}, {\em {M}c{K}ean–{V}lasov equations involving hitting times: Blow-ups and global solvability.}, Ann. Appl. Probab., 34 (2024), pp.~1600--1622.

\bibitem{Calvez2012}
{\sc V.~Calvez, R.~J. Hawkins, N.~Meunier, and R.~Voituriez}, {\em Analysis of a nonlocal model for spontaneous cell polarization}, SIAM J. Appl. Math., 72 (2012), pp.~594--622.

\bibitem{chadam}
{\sc J.~Chadam and P.~Ortoleva}, {\em The stabilizing effect of surface tension on the development of the free boundary in a planar, one-dimensional, {C}auchy-{S}tefan problem}, IMA J. Appl. Math., 30 (1983), pp.~57--66.

\bibitem{dlachayesswindle}
{\sc L.~Chayes and G.~Swindle}, {\em Hydrodynamic limits for one-dimensional particle systems with moving boundaries}, Ann. Probab., 24 (1996), pp.~559--598.

\bibitem{minimal}
{\sc C.~Cuchiero, S.~Rigger, and S.~Svaluto{-}Ferro}, {\em Propagation of minimality in the supercooled {S}tefan problem}, Ann. Appl. Probab., 33 (2023), pp.~1588--1618.

\bibitem{delarue:hal-00870991}
{\sc F.~Delarue, J.~Inglis, S.~Rubenthaler, and E.~Tanr{\'e}}, {\em {First hitting times for general non-homogeneous 1d diffusion processes: density estimates in small time}}.
\newblock hal-00870991, 2013.

\bibitem{DIRT1}
{\sc F.~Delarue, J.~Inglis, S.~Rubenthaler, and E.~Tanr{\'e}}, {\em {Global solvability of a networked integrate-and-fire model of McKean–Vlasov type}}, Ann. Appl. Probab., 25 (2015), pp.~2096 -- 2133.

\bibitem{delarue2015particle}
{\sc F.~Delarue, J.~Inglis, S.~Rubenthaler, and E.~Tanré}, {\em Particle systems with a singular mean-field self-excitation. application to neuronal networks}, Stoch. Proc. Applic., 125 (2015), pp.~2451--2492.

\bibitem{3}
{\sc F.~Delarue, S.~Nadtochiy, and M.~Shkolnikov}, {\em Global solutions to the supercooled {S}tefan problem with blow-ups: regularity and uniqueness}, Probab. Math. Phys., 3 (2022), p.~171–213.

\bibitem{Dembo_2019}
{\sc A.~Dembo and L.-C. Tsai}, {\em Criticality of a randomly-driven front}, Arch. Ration. Mech. Anal., 233 (2019), p.~643–699.

\bibitem{4}
{\sc J.~N. Dewynne, S.~D. Howison, J.~R. Ockendon, and W.~Q. Xie}, {\em Asymptotic behavior of solutions to the {S}tefan problem with a kinetic condition at the free boundary}, J. Austral. Math. Soc. Ser. B., 31 (1989), p.~81–96.

\bibitem{11}
{\sc S.~N. Ethier and T.~G. Kurtz}, {\em Markov Processes: Characterization and Convergence}, Wiley Series in Probability and Mathematical Statistics: Probability and Mathematical Statistics, John Wiley \& Sons, New York, 1986.

\bibitem{feinstein2022contagious}
{\sc Z.~Feinstein and A.~Søjmark}, {\em Contagious {M}c{K}ean-{V}lasov systems with heterogeneous impact and exposure}, Finance Stoch., 27 (2023), pp.~663--711.

\bibitem{NSsurfacetension}
{\sc Y.~Guo, S.~Nadtochiy, and M.~Shkolnikov}, {\em Stefan problem with surface tension: Uniqueness of physical solutions under radial symmetry}, Archive for Rational Mechanics and Analysis, 248 (2024), p.~90.

\bibitem{2}
{\sc B.~Hambly, S.~Ledger, and A.~Søjmark}, {\em A {M}c{K}ean--{V}lasov equation with positive feedback and blow-ups}, Ann. Appl. Probab., 29 (2019), pp.~2338 -- 2373.

\bibitem{kallenberg}
{\sc O.~Kallenberg}, {\em Random Measures, Theory and Applications}, Springer Cham, 2017.

\bibitem{5}
{\sc S.~Ledger and A.~Søjmark}, {\em Uniqueness for contagious {M}c{K}ean–{V}lasov systems in the weak feedback regime}, Bull. Lond. Math. Soc., 52 (2020), pp.~448--463.

\bibitem{ledger2021mercy}
{\sc S.~Ledger and A.~Søjmark}, {\em At the mercy of the common noise: Blow-ups in a conditional {M}c{K}ean--{V}lasov problem}, Electron. J. Probab., 29 (2021), pp.~1--39.

\bibitem{ledger2024supercooledstefanproblemtransport}
\leavevmode\vrule height 2pt depth -1.6pt width 23pt, {\em The supercooled {S}tefan problem with transport noise: weak solutions and blow-up}.
\newblock Arxiv:2409.20421, 2024.

\bibitem{7}
{\sc V.~A. Malyshev and A.~A. Zamyatin}, {\em Accumulation on the boundary for one-dimensional stochastic particle system}, Probl. Inf. Transm., 43 (2007), pp.~331--343.

\bibitem{mustapha2023wellposednesssupercooledstefanproblem}
{\sc S.~Mustapha and M.~Shkolnikov}, {\em Well-posedness of the supercooled {S}tefan problem with oscillatory initial conditions}, Electron. J. Probab., 29 (2024), pp.~Paper No. 193, 21.

\bibitem{NS19}
{\sc S.~Nadtochiy and M.~Shkolnikov}, {\em Particle systems with singular interaction through hitting times: Application in systemic risk modeling}, Ann. Appl. Probab., 29 (2019), pp.~89--129.

\bibitem{nadtochiy2022tension}
{\sc S.~Nadtochiy and M.~Shkolnikov}, {\em {S}tefan problem with surface tension: global existence of physical solutions under radial symmetry}, Probab. Theory Related Fields, 187 (2023), p.~385–422.

\bibitem{dlaNS}
{\sc S.~Nadtochiy, M.~Shkolnikov, and X.~Zhang}, {\em {Scaling limits of external multi-particle DLA on the plane and the supercooled Stefan problem}}, Ann. Inst. Henri Poincar\'e Probab. Stat., 60 (2024), pp.~658 -- 691.

\bibitem{Ricci_Weiqing_1991}
{\sc R.~Ricci and X.~Wei~Qing}, {\em On the stability of some solutions of the {S}tefan problem}, European J. Appl. Math., 2 (1991), pp.~1--15.

\bibitem{whitt}
{\sc W.~Whitt}, {\em An Introduction to Stochastic-Process Limits and Their Application to Queues}, Springer New York, 2011.

\end{thebibliography}
\bibliographystyle{siam}

\appendix
\section{Continuity of $\tilde h$}
\label{Continuity of h}
In this subsection we prove the continuity of the mapping $\tilde h$.

\begin{proof}
Let $\zeta^n\rightarrow \zeta\in M(\mathcal{C})$ and $l^n\rightarrow l\in \tilde D$. Let $f:\tilde D\rightarrow\mathbb{R}$ be bounded, 1-Lipschitz, and with bounded support.  Since $f$ has bounded support, it is supported on some set $A_1:=\{y\in \tilde D:||y||_\infty\leq M_1\}$. Since $l^n\rightarrow l$, the sequence $l^n$ is contained in some set $A_2:=\{y\in \tilde D:||y||_\infty\leq M_2\}$. Hence for every integer $n$, $f(x-l^n)=0$ 
is supported on the set $A_3$ where $A_3$ is the set $\{x=y\in\tilde D:||y||_\infty\leq M_1+M_2\}$. 

It follows that 
\begin{align*}
&\limsup_{n\rightarrow\infty}\left|\int f(x)d\tilde h(\zeta^n,l^n)-\int f(x)d\tilde h(\zeta,l)\right|\\&=\limsup_{n\rightarrow\infty}\left|\int_{A_3} f(x-l^n)d\zeta^n-\int_{A_3}f(x-l)d\zeta\right|\\
&\leq \limsup_{n\rightarrow\infty}d_{M1}(l^n,l)\limsup_{n\rightarrow\infty}\zeta^n(A_3)+\limsup_{n\rightarrow\infty}\left |\int f(x-l)d\zeta^n(x)-\int f(x-l)d\zeta(x)\right|\\
&=0
\end{align*}

The inequality follows since $f$ is 1-Lipschitz. The final equality follows by the vague convergence of $\zeta^n\rightarrow\zeta$ and the M1 convergence of $l^n\rightarrow l$, since $m(x)=f(x-l)$ is continuous and bounded with bounded support, and since $A_3$ is closed and bounded. 
\end{proof}
\section{Laplace transform}
\label{itoformula}

We prove the result of Lemma \ref{Ito}.

Recall that for a solution $\Lambda$ to \eqref{alternativetomvr}, we denote by $W_t^i$ the value $X_0^i+B_t^i=X_t^i+\Lambda_t$. Further, we denote by $ V(t,x)=\tilde V(t,x-\Lambda_{t})$ the density $\sum_{i=1}^\infty\mathbb{P}(W_t^i\in dx,\tau_i>t)$, and similarly we denote by $ V(t^-,x)=\tilde V(t^-,x-\Lambda_{t^-})$ the density $\sum_{i=1}^\infty\mathbb{P}(W_t^i\in dx,\tau_i\geq t)$. We must check that for an arbitrary $\lambda>0$ and $t>0$:
\begin{align*}&\int_{\Lambda_t}^\infty V(t,x)e^{-\lambda x-\frac{\lambda^2t}{2}}dx+\int_0^\infty (1-g(x))e^{-\lambda x}dx-\frac{1}{\lambda}e^{-\lambda \Lambda_t-\frac{\lambda^2t}{2}}\\&=\frac{\lambda}{2}\int_0^Te^{-\lambda\Lambda_{s^-}-\frac{\lambda^2s}{2}}ds+\sum_{s\leq t}e^{-\frac{\lambda^2s}{2}}\left(-\int_{\Lambda_{s^-}}^{\Lambda_s} V(s^-,x)e^{-\lambda x}dx-\frac{1}{\lambda}\Delta e^{-\lambda\Lambda_s}\right).
\end{align*}

We first apply Ito's formula to $e^{-\lambda W_t^i-\frac{\lambda^2t}{2}}\mathbbm{1}_{\{\tau_i\leq t\}}$. This yields:

\begin{align*}
    e^{-\lambda W_t^i-\frac{\lambda^2t}{2}}\mathbbm{1}_{\{\tau_i>t\}}-e^{-\lambda X_{0^-}^i}=-\lambda \int_0^t e^{-\lambda W_s^i-\frac{\lambda^2s}{2}}\mathbbm{1}_{\{\tau_i> s\}}dB_s^i-e^{-\lambda W_{\tau_i}^i-\frac{\lambda^2\tau_i}{2}}\mathbbm{1}_{\{\tau_i\leq t\}}.
\end{align*}

We take expectations, then sum over $i\in \mathbb{N}$. Checking integrability, we then obtain that

\begin{align}
\label{itoterm2}
\mathbb{E}\left(\sum_{i=1}^\infty e^{-\lambda W_t^i-\frac{\lambda^2t}{2}}\mathbbm{1}_{\{\tau_i\geq t\}}\right)-\mathbb{E}\left(\sum_{i=1}^\infty e^{-X_{0^-}^i}\right)=-\mathbb{E}\left(\sum_{i=1
}^\infty e^{-\lambda W_{\tau_i}^i}\mathbbm{1}_{\{\tau_i\leq t\}}\right).
\end{align}

Write $J$ for the set of discontinuities of $\Lambda_t$.
If $\tau_i\in J^c$, it holds that $W_{\tau_i}=\Lambda_{\tau_i}$. Replacing the $W_{\tau_i}$ terms by $\Lambda_{\tau_i}$ terms at all times $\tau_i\notin J$, the term in the right hand side of the above equation can then be rewritten as 
\begin{align}
\label{discsum}-\left(\mathbb{E}\left(\int_0^te^{-\lambda \Lambda_{s^-}-\frac{\lambda^2s}{2}}d\left(\sum_{i=1}^\infty \mathbbm{1}_{\{\tau_i\leq s\}}\right)\right)+\sum_{s\in J\cap[0,t]}\mathbb{E}\left(\sum_{i=1}^\infty e^{-\frac{\lambda^2s}{2}}\left(e^{-\lambda W_{s}^i}-e^{-\lambda \Lambda_{s^-}}\right)\mathbbm{1}_{\{\tau_i =s\}}\right)\right)\end{align}

As stated in the proof of Lemma \ref{lem:aplus}, the process $\sum_{i=1}^\infty \mathbbm{1}_{\{\tau_i\leq t\}}-\Lambda_t$ is a martingale. Therefore, the first term in \eqref{discsum} can be replaced by the integral
$$\int_0^te^{-\lambda\Lambda_s-\frac{\lambda^2s}{2}}d\Lambda_s.$$
Each term in the sum over $s\in J$ may be rewritten as 
\begin{align*}
    &\mathbb{E}\left(\sum_{i=1}^\infty e^{-\frac{\lambda^2s}{2}}\left(e^{-\lambda W_{s}^i}-e^{-\lambda \Lambda_{s^-}}\right)\mathbbm{1}_{\{\tau_i\geq s\}}\mathbbm{1}_{\{W_s^i\leq \Lambda_s\}}\right)\\
    &=\int_{\Lambda_{s^-}}^{\Lambda_s}e^{-\frac{\lambda^2s}{2}}\left(e^{-\lambda x}-e^{-\lambda \Lambda_{s^-}}\right)V(s^-,x)dx.
\end{align*}

This follows directly by the definition of $V(t^-,x)$.
Denote by $\hat V_t$ the Laplace transform $\int_{0}^\infty \tilde V(t,x)e^{-\lambda x}dx$. The equation \eqref{itoterm2} can now be rewritten as 

\begin{align*}
    \hat V_te^{-\lambda\Lambda_t-\frac{\lambda^2t}{2}}-\hat V_0&=-\int_0^te^{-\lambda \Lambda_{s^-}-\frac{\lambda^2s}{2}}d\Lambda_s-\sum_{s\leq t}\int_{\Lambda_{s^-}}^{\Lambda_s}e^{-\frac{\lambda^2s}{2}}\left(e^{-\lambda x}-e^{-\lambda \Lambda_{s^-}}\right)V(s^-,x)dx,\\
    &=-\int_0^te^{-\lambda \Lambda_{s^-}-\frac{\lambda^2s}{2}}d\Lambda_s-\sum_{s\leq t}e^{-\frac{\lambda^2s}{2}}\left(\int_{\Lambda_{s^-}}^{\Lambda_s}V(s^-,x)e^{-\lambda x}dx-e^{-\lambda \Lambda_{s^-}}\Delta \Lambda_s\right).
\end{align*}

The first term on the right hand side is given by 
$$\int_0^te^{-\frac{\lambda^2s}{2}}d\left(\frac{e^{-\lambda \Lambda_s}}{\lambda }\right)-\sum_{s\leq t}e^{-\frac{\lambda^2s}{2}}\left(e^{-\lambda \Lambda_{s^-}}\Delta \Lambda_s+\frac{\Delta e^{-\lambda \Lambda_s} }{\lambda}\right).$$
Therefore, applying integration by parts to this term, we obtain that

\begin{align*}
    \hat V_te^{-\lambda\Lambda_t-\frac{\lambda^2t}{2}}-\hat V_0&=\frac{\lambda }{2}\int_0^te^{-\lambda \Lambda_s-\frac{\lambda^2s}{2}}ds+\frac{e^{-\lambda \Lambda_t-\frac{\lambda^2t}{2}}-1}{\lambda}-\sum_{s\leq t}e^{-\frac{\lambda^2s}{2}}\left(e^{-\lambda \Lambda_{s^-}}\Delta\Lambda_s+\frac{\Delta e^{-\lambda \Lambda_s}}{\lambda}\right)\\
    &-\sum_{s\leq t}e^{-\frac{\lambda^2s}{2}}\left(\int_{\Lambda_{s^-}}^{\Lambda_s}e^{-\lambda x}V(s^-,x)dx-e^{-\lambda\Lambda_{s^-}}\Delta \Lambda_s\right).\end{align*}

Rearranging this finally yields:
\begin{align*}
    &\left(\hat V_t-\frac{1}{\lambda}\right)e^{-\lambda \Lambda_t-\frac{\lambda^2t}{2}}+\left(\frac{1}{\lambda}-\hat V_0\right)\\
    &=\frac{\lambda}{2}\int_0^te^{-\lambda\Lambda_{s^-}-\frac{\lambda^2s}{2}}ds+\sum_{s\leq t}e^{-\frac{\lambda^2s}{2}}\left(-\int_{\Lambda_{s^-}}^{\Lambda_s}V(s^-,x)e^{-\lambda x}dx-\frac{1}{\lambda}\Delta e^{-\lambda\Lambda_s}\right).
\end{align*}

\section{Finite explosion times}\label{finiteexplosionsec}

\begin{proof}[Proof of Proposition \ref{explosiontime}]
    We start with the first statement. We can consider $T^*<\infty$, otherwise the statement is immediate.

    For an arbitrary $M\in\mathbb{Q}^+$ we define $$\Lambda^{N,M}_t=\Lambda^N_{t\wedge T^{N,M}}\wedge M.$$ 
    We replace $X_t^{i,N}$ by $X_{M,t}^{i,N}=W_t^{i,N}-\Lambda^{N,M}_{t}$, which we view as a process on $[0,T^*+1]$, and extend to processes on $[-1,T^*+2]$ in the same manner as we extend $\Lambda^N$ in Section \ref{convergencestatementsection}. We consider the empirical measure of the $(X_{M,t}^{i,N})_{t\in[-1,T^*+2]}$ particles, and label this empirical measure as $\mu^{N,M}$.  It follows from the definition of $T^{N,M}$ that either $T^{N,M}=T^*+1$, or that $\nu^N_{T^{N,M}-}([0,M-\Lambda^N_{T^{N,M}-}])\geq M-\Lambda^N_{T^{N,M}-}$. In either case, it holds that $\mu^{N,M}(\tau\leq t)\geq \Lambda^{N,M}_t$ for all $t<T^*+1$. 
    Through an analogous argument to the proof of Theorem \ref{maintheorem}, we can show that the processes
    $(\mu^{N,M},\Lambda^{N,M},\vartheta^N,T^{N,M})$ are tight, and that any subsequential limit $(\mu^M,\nu^M,\Lambda^M,\vartheta,T^M)$ is such that

 $(\mu^M,\Lambda^M)$ forms a solution to the McKean-Vlasov type equation \eqref{alternativetomvr}, on the time interval $[0,T^M)$. Therefore, $\Lambda^M$ agrees with the unique solution $\Lambda$ on this time interval. 

Suppose that there existed some $\epsilon>0$ such that $T^M\geq T^*+\epsilon$. We would then obtain a solution to \eqref{alternativetomvr} on a time interval extending past $T^*$. Therefore, for any convergent subsequence $N_K$, for any $\epsilon>0$, $\liminf_{N\rightarrow\infty}\mathbb{P}(T^{N_K,M}<T^*+\epsilon)=1.$
Taking a limit over integers $M$, $M\uparrow\infty$, we obtain that $$\lim_{M\uparrow\infty} \liminf_{N_K\rightarrow\infty}\mathbb{P}(T^{N_K,M}<T^*+\epsilon)=1.$$
For any sequence of integers growing to infinity, we can find a subsequence satisfying the above. Therefore, the result of the first statement holds.

For the reverse direction, we define for any $q\in \mathbb{Q}^+$, $\Lambda^{N,q}_t=\Lambda^N_{t\wedge T^{N}}\wedge q$. We take $\mu^{N,q}$ to be the empirical measure of particles $(W_{t\wedge T^{N}}^{i,N}-\Lambda^{N,q}_t)_{t\in[-1,T^*+2]}$. This system satisfies that $\Lambda^{N,q}_t= \mu^{N,q}(\tau\leq t)\wedge q$.
Analogously to first half of the proof, we can determine that that the processes
$((\mu^{N,q},\Lambda^q)_{q\in \mathbb{Q}^+},T^N)$ are tight (when viewed under the product topology), and that any subsequential limit 
 $((\mu^q,\Lambda^q)_{q\in\mathbb{Q}},T^\infty)$ is such that for each $q$, $(\mu^q,\Lambda^q_t)$ is a solution to the McKean-Vlasov system, up to the time $T^q:=\inf\{t\geq 0:\Lambda^q_t\geq q\}\wedge T^*$.
Therefore, $\Lambda^q$ agrees with the unique solution $\Lambda_t$, on the interval $[0,T^q)$. Note that $T^q\leq T^\infty$ for each $q$. Suppose that $T^\infty<T^*-\epsilon$, for some $\epsilon>0$. It must then hold that for all $q$, $\Lambda_{T^*+2}^q\leq (\Lambda_{T^q-}\wedge q)\leq \Lambda_{T^*-\epsilon}.$ Therefore, for any $q$ such that $\Lambda_{T^*-}> q\geq \Lambda_{T^*-\epsilon}$, it holds that $T^q=\infty$, and $\Lambda_t^q=\Lambda_t$ on the entire interval. This is impossible since $\Lambda_t$ approaches the value $\Lambda_{T^*-}$, while $\Lambda^q_t$ is bounded above by $q$. We therefore obtain that for any $\epsilon>0$, $T^\infty>T^*-\epsilon$. Therefore, for any convergent subsequence $T^{N_K}$,
$$\liminf_{N_K\rightarrow\infty}\mathbb{P}(T^{N_K}\leq T^*-\epsilon)=1.$$

For any sequence of integers increasing to infinity, such a convergent subsequence may be found. The result of the second statement therefore follows.
\end{proof}

\begin{proof}[Proof of Proposition \ref{CDFalpha}]

    We observe that there exists $x^*>0$ and $\beta<1$ such that $\sup_{t\leq T}\sup_{x\leq x^*}V(t,x)\leq \beta$. If not, we can find a time $t\leq T$, where $\lim_{x\downarrow 0}V(t,x)\geq 1$, and hence at this time, $\Lambda$ must explode to infinity. Using this bound on the density $V(t,x)$ and arguing as in \cite[Proposition 2.2]{3}, we can determine that there exists a constant $C_T$ such that $\Lambda$ is 1/2-H\"older continuous on $[0,T]$ and with H\"older norm less than $C_T$. We can also find $c_T>0$ such that for any $t\leq T$, $$\mathbb{E}\left(\sup_{\{x\leq \sup_{s\leq c_T}B_s+2C_T\sqrt{c_T}\}}V(t,x)\right)\leq(1+\beta)/2. $$

    We fix $\beta'$ such that $\beta<\beta'<1$, and we define a sequence $(z_i)_{i\in\mathbb{N}}$ inductively in the following manner:

    $$z_1=1, \quad (1+\beta)\sum_{i=1}^jz_i=(1+\beta') z_j \textrm{ for }j\geq 2.$$

    We can argue analogously to the proof of Theorem \ref{maintheorem} that if $M>\Lambda_T$, $(\Lambda^N_{t})_{t\leq T}\wedge M$ converges in distribution to $(\Lambda_t)_{t\leq T}$. Since $\Lambda$ is continuous, and M$1$ convergence to a continuous function is equivalent to uniform convergence, we further obtain that $\Lambda^N$ converges uniformly in probability to $\Lambda$ on $[0,T]$.

    Let us consider functions of the form $f_t^d=\Lambda_t+d\left(\sum_{i=1}^{\lceil T/c_T\rceil}z_i\right)$, for $d\in\mathbb{R}$.
    We define, with some abuse of notation, $$X_t^{i,d}=X_0^{i,N}+B_t^{i,N}-\Lambda_t-d\left(\sum_{i=1}^{\lceil T/c_T\rceil}z_i\right),$$ and $$\tau^{i,d}=\inf\{t\geq 0:X_t^{i,d}\leq 0\}.$$ Suppose that $\Gamma_N(f_t^d)\leq f_t^d$ for $t< t_1$. Then, for it to hold that $\Gamma_N(f_t^d)\leq f_t^d$ for all $t\leq t_2$, it is sufficient that 
    $$\Gamma_N^{t_1,d}(\Lambda)_t=\frac{1}{N}\sum_{i=1}^\infty\mathbbm{1}_{\{X_{t_1-}^{i,d}\leq \sup_{s\leq t}(-B^{i,N}_{s+t_1}+B^{i,N}_{t_1}+\Lambda_{t_1+s}-\Lambda_{t_1}-dz_2), \tau^{i,d}\geq{t_1}\}}\leq \Lambda_{t_1+t}-\Lambda_{t_1}+dz_2\textrm{ for }t\leq t_2-t_1.$$
Therefore:
\begin{align*}&\mathbb{P}(\Gamma_N(f_t^d)\leq f_t^d\textrm{ for }t\leq t_2)\\
&\leq \mathbb{P}(\Gamma_N(f_t^d)\leq f_t^d \textrm{ for }t< t_1)+\mathbb{P}(\Gamma_N^{t_1,d}(\Lambda)\leq f_{t_1+t}^d-f_{t_1-}^d\textrm{ for }t\leq t_2-t_1).\end{align*}
  
  The points $(X_{t_1-}^{i,d})_{\tau^{i,d}\geq{t_1}}$ are a Poisson point process, with an intensity function $V^d(t_1-,x)$, which is bounded above by $V(t,x+dz_1)\mathbbm{1}_{[0,\infty)}(x)$. For $d(\sum_{i=1}^{\lceil T/c_T\rceil }z_i)\leq C_T\sqrt{c_T}$, $t\leq c_T$, we may bound \begin{align*}\mathbb{E}(\Gamma_N^{t_1,d}(\Lambda)_t)&\leq \int_0^\infty \mathbb{P}(x\leq (\sup_{s\leq t}B_s+(\Lambda_{t_1+s}-\Lambda_{t_1}-z_2d)))V(t_1,x+z_1d)dx,\\
&\leq \int_0^\infty \mathbb{P}(x\leq \sup_{s\leq t}(B_s+\Lambda_{t_1+s}-\Lambda_{t_1})+(z_1+z_2)d)V(t_1,x)dx,\\
&= \Lambda_t-\Lambda_{t_1}+\mathbb{E}\left(\int_{\sup_{s\leq t}(B_s+\Lambda_{t_1+s}-\Lambda_{t_1})}^{\sup_{s\leq t}(B_s+\Lambda_{t_1+s}-\Lambda_{t_1})+(z_1+z_2)d}V(t_1,x)\right),\\
&\leq \Lambda_t-\Lambda_{t_1}+(z_1+z_2)d\mathbb{E}\left(\sup_{x\leq \sup_{\{s\leq c_T}B_s+2C_T\sqrt{c_T}\}}V(t_1,x)\right),\\
&\leq \Lambda_{t}-\Lambda_{t_1}+dz_2(1+\beta')/2.\end{align*}

Arguing identically to the proof of Lemma \ref{lem:aplus}, we can determine that if $y$ satisfies $(\sum_{i=1}^{\lceil T/c_T \rceil }z_i)y/\sqrt{N}\leq C_T\sqrt{c_T}$, then

$$\mathbb{P}\left(\sup_{t\leq c_T}\sqrt{N}(\Gamma_N^{t_1,y/\sqrt{N}}(\Lambda)_t-\Lambda_{t_1+t}+\Lambda_{t_1}-yz_2/\sqrt{N})\geq 0\right)\leq 2e^{-\frac{(\beta'-\beta)^2z_2^2y^2}{(\Lambda_T+\alpha)+(\beta'-\beta)z_2y/\sqrt{N}}}.$$

A similar bound holds for $t_1=0$, and so we obtain that
$$\mathbb{P}(\Gamma_N(f^{y/\sqrt{N}})_t\leq f^{y/\sqrt{N}}\textrm{ for }t\leq 2c_T)\geq 1-2*2e^{-\frac{(\beta'-\beta)^2z_2^2y^2}{\Lambda_T+\alpha+(\beta'-\beta)z_2y/\sqrt{N}}}.$$

We consider a partition of $[0,T]$ given by $t_0=0,$ and $t_{i+1}-t_i=c_T\wedge (T-t_i)_+.$ Iterating our above argument, we obtain that for $(\sum_{i=1}^{\lceil T/c_T \rceil }z_i)y/\sqrt{N}\leq C_T\sqrt{c_T}$
$$\mathbb{P}\left(\Gamma_N(f^{y/\sqrt{N}})_t\leq f^{y/\sqrt{N}}_t\textrm{ for }t\leq T\right)\geq 1- 2\left\lceil\frac{T}{c_T}\right\rceil e^{-\frac{(\beta'-\beta)^2y^2}{\Lambda_T+\alpha+z_{\lceil T/c_T\rceil}(\beta'-\beta)y/\sqrt{N}}}.$$

Applying Lemma \ref{Alemmatoboundathing}, we can see that this implies 

$$\mathbb{P}\left(\sup_{t\leq T}\sqrt{N}(\Lambda^N_t-\Lambda_t)\geq y\left(\sum_{i=1}^{\lceil T/c_T\rceil }z_i\right)\right)\leq \left\lceil\frac{T}{c_T}\right\rceil e^{-\frac{(\beta'-\beta)^2y^2}{\Lambda_T+\alpha+z_{(T/c_T)}(\beta'-\beta)y/\sqrt{N}}},$$

 and further that 
$$\lim_{y\rightarrow\infty}\limsup_{N\rightarrow\infty}\mathbb{P}\left(\sup_{t\leq T}\sqrt{N}(\Lambda^N_t-\Lambda_t)\geq y\left(\sum_{i=1}^{\lfloor T/c_T\rfloor }z_i\right)\right)=0.$$

From analogous arguments, we may also obtain that 
$$\lim_{y\rightarrow\infty}\limsup_{N\rightarrow\infty}\mathbb{P}\left(\sup_{t\leq T}\sqrt{N}(\Lambda_t-\Lambda^N_t)\geq y\right)=0.$$

Therefore, we have determined that $\sqrt{N}\sup_{t\leq T}|\Lambda^N_t-\Lambda_t|$ is stochastically bounded.

\end{proof}

\begin{proof}[Proof of Proposition \ref{BHJextension}]
    We assume without loss of generality that $\alpha\geq 1$, otherwise the result of \cite[Theorem 1.14]{BHJ2025} applies. By the workings of \cite[Proposition 3.2]{BHJ2025}, there exists $\beta<1$, and $x^*>0$ such that $\sup_{t\leq T}\mathbb{P}(X_t\leq x^*)\leq \beta/\alpha.$ From this it can be checked that there exists a constant $C_T$, such that $l$ is 1/2 H\"older continuous on $[0,T]$ with H\"older norm less than $C_T$. 
    
    We choose some $c_T\leq (\frac{x^*}{2\alpha C_T})^2$, and such that $\alpha\left(1-\frac{\beta}{\alpha}\right)\left(1-\textrm{erf}\left(\frac{x^*}{6\sqrt{2c_T}}\right)\right)\leq \frac{1-\beta}{2}.$
    We define a partition of $[0,T]$ by $t_i=0,$ $t_i-t_{i-1}=c_T\wedge (T-t_{i-1})_+$. We write $$\hat X_t^{j,N}=X_0^{j,N}+B_t^{j,N}-\alpha l_t+\hat L_t^{j,N},$$ where $$\hat L_t^{j,N}=\sup_{s\leq t}(X_0^{j,N}+B_s^{j,N}-\alpha l_s)_-,$$ and recall the definition $$X_t^{j,N}=X_0^{j,N}+B_t^{j,N}-\alpha \bar L_t^N+L_t^{j,N}$$ given above Proposition \ref{BHJextension}. 
    We define a random variable $l^N$:
    $$l^N_t:=\frac{1}{N}\sum_{i=1}^N\hat L_t^{i,N}.$$ 
    It follows directly from the definitions that $|\hat X_t^{j,N}-X_t^{j,N}|\leq 2\alpha\sup_{s\leq T}|l_s-\bar L_s^N|.$
    Further, we observe that $$\hat L_t^{j,N}-\hat L_{t_i}^{j,N}=\sup_{s\leq t-t_i}(X_{t_i}^{j,N}+(B^{j,N}_s-B^{j,N}_{t_i})-\alpha(l_t-l_{t_i}))_-$$ and $$L_t^{j,N}-L_{t_i}^{j,N}=\sup_{s\leq t-t_i}(X_{t_i}^{j,N}+(B^j_s-B^j_{t_i})-\alpha (\bar L^N_t-\bar L^N_{t_i}))_-.$$
    These are both zero, unless $\hat X_{t_i}^{j,N}+\inf_{s\leq t-t_i}(B^j_s-B^j_{t_i})-\alpha (l_t-l_{t_i})-3\alpha \sup_{s\leq T}|\bar L_s^N-l_s|\leq 0.$

    Adding this observation to the argument of \cite[Theorem 1.14]{BHJ2025}, we obtain 
    that, for an arbitrary $j$,
    \begin{align*}
    &\sup_{t_{j}\leq s\leq t_{j+1}}|\bar L_s^N-l_s|
    \\&\quad\leq |\bar L^N_{t_j}-l_{t_j}|+\sup_{t_{j}\leq s\leq t_{j+1}}|l_s^N-l_s|\\&\quad+\sup_{t_{j}\leq s\leq t_{j+1}}\alpha |\bar L_s^N-l_s|\frac{1}{N}\sum_{i=1}^N\mathbbm{1}_{\{\hat X_{t_j}^{i,N}\leq \sup_{t_{j}\leq s\leq t_{j+1}}(B_s^{i,N}-B_{t_j}^{i,N})+\alpha(l_{t_{j+1}}-l_{t_j})+3\alpha \sup_{s\leq T}|\bar L_s^N-l_s|\}},\\
    &\implies \sup_{t_{j}\leq s\leq t_{j+1}}|\bar L_s^N-l_s|\\
    &\quad\quad\quad\leq \left(|\bar L^N_{t_j}-l_{t_j}|+\sup_{t_{j}\leq s\leq t_{j+1}}|l_s^N-l_s|\right)\\
    &\hspace{1.5cm}\times\left(1-\alpha\frac{1}{N}\sum_{i=1}^N\mathbbm{1}_{\{\hat X_{t_j}^{i,N}\leq \sup_{t_{j}\leq s\leq t_{j+1}}(B_s^{i,N}-B_{t_j}^{i,N})+\alpha (l_{t_{j+1}}-l_{t_j})+3\alpha \sup_{s\leq T}|\bar L_s^N-l_s|\}}\right)^{-1}.
    \end{align*}
    Therefore, for any $K>0$:
    \begin{align}
        &\mathbb{P}(\sqrt{N}\sup_{t_j\leq s\leq t_{j+1}}|\bar L_s^N-l_s|\geq M,\sup_{s\leq T}|\bar L_s^N-l_s|\leq x^*/9\alpha)\label{eqtnonl}\\
        &\leq \mathbb{P}\left(\sqrt{N}\sup_{s\leq T}|l_s^N-l_s|\geq MK/2\right)+\mathbb{P}\left(1-\alpha \frac{1}{N}\sum_{i=1}^N\mathbbm{1}_{\{\hat X_{t_j}^{i,N}\leq \sup_{s\leq c_T}(B_{s+t_j}^{i,N}-B_{t_j}^{i,N})+5x^*/6\}}\leq K\right)\\&\hspace{0.4cm}+\mathbb{P}(\sqrt{N}|\bar L^N_{t_j}-l_{t_j}|\geq MK/2,\sup_{s\leq T}|\bar L_s^N-l_s|\leq x^*/9\alpha).\nonumber
\end{align}

We bound 
\begin{align*}
    \mathbb{P}(X_{t_j}^{i,N}\leq \sup_{s\leq c_T}B_s+5x^*/6)&=\int_0^\infty\left( \left(1-\textrm{erf}\left(\frac{x-5x^*/6}{\sqrt{2c_T}}\right)\right)\wedge 1 \right)\mathcal{L}(X_{t_j})(dx),\\
    &\leq \int_0^\infty \left( \left(1-\textrm{erf}\left(\frac{x-5x^*/6}{\sqrt{2c_T}}\right)\right)\wedge 1 \right) \left(\frac{\beta}{\alpha}\delta_{0}+\left(1-\frac{\beta}{\alpha}\right)\delta_{x^*}\right)(dx),\\
    &\leq \frac{\beta}{\alpha}+\left(1-\frac{\beta}{\alpha}\right)\left(1-\textrm{erf}\left(\frac{x^*}{6\sqrt{2c_T}}\right)\right),\\
    &\leq \beta/\alpha+(1-\beta)/(2\alpha),\\
    &\leq (1+\beta)/2\alpha.
\end{align*}
Choosing $K=(1-\beta)/4$, the second term on the right hand side of the inequality \eqref{eqtnonl} is bounded above by
$$\mathbb{P}\left(\frac{N(3+\beta)}{4\alpha}\leq Y\right),$$

where Y is a Binomial random variable, with parameters $N$, and $\frac{1+\beta}{2\alpha}$. The above probability converges to 0, as $N\rightarrow\infty$.

Using that $l$ is 1/2-H\"older continuous on $[0,T]$, from an identical argument to the proof of \cite[Theorem 1.14]{BHJ2025} we may determine that $\sqrt{N}\sup_{s\leq T}|l_s^N-l_s|$ is stochastically bounded. Therefore, 
$$\lim_{M\uparrow\infty}\limsup_{N\rightarrow\infty}\mathbb{P}\left(\sqrt{N}\sup_{s\leq T}|l_s^N-l_s|\geq \frac{(1-\beta)M}{4}\right)=0.$$

In total, from the inequality \eqref{eqtnonl}, we obtain that 
\begin{align*}&\lim_{M\uparrow\infty}\limsup_{N\rightarrow\infty}\mathbb{P}(\sup_{t_j\leq s\leq t_{j+1}}\sqrt{N}|\bar L_s^N-l_s|\geq M,\sup_{s\leq T}|\bar L_s^N-l_s|\leq x^*/9\alpha)\\
&=\lim_{M\uparrow\infty}\limsup_{N\rightarrow\infty}\mathbb{P}(\sqrt{N}|\bar L_{t_j}^N-l_{t_j}|\geq M,\sup_{s\leq T}|\bar L_s^N-l_s|\leq x^*/9\alpha)\end{align*}
Arguing inductively, starting with $T=t_{j_{max}}$, we see that
\begin{align*}&\lim_{M\uparrow\infty}\limsup_{N\rightarrow\infty}\mathbb{P}(\sup_{s\leq T}\sqrt{N}|\bar L_s^N-l_s|\geq M,\sup_{s\leq T}|\bar L_s^N-l_s|\leq x^*/9\alpha)\\
&=\lim_{M\uparrow\infty}\limsup_{N\rightarrow\infty}\mathbb{P}(\sqrt{N}|\bar L_0^N-l_0|\geq M,\sup_{s\leq T}|\bar L_s^N-l_s|\leq x^*/9\alpha)=0.\end{align*}
By \cite[Theorem 1.14]{BHJ2025}, $\liminf_{N\rightarrow\infty}\mathbb{P}(\sup_{s\leq T}|\bar L_s^N-l_s|\leq x^*/9\alpha)=1$. Therefore: 

$$\lim_{M\uparrow\infty}\limsup_{N\rightarrow\infty}\mathbb{P}(\sup_{s\leq T}\sqrt{N}|\bar L_s^N-l_s|\geq M)=0.$$
\end{proof}

\section{Proofs deferred from Section \ref{critsec}}
\label{critfull}
We first prove the scaling result given in Theorem \ref{crit}.
\begin{proof}[Proof of Theorem \ref{crit}]

\underline{Lower bound:}

    Recall that $\xi_{N^2}/N\stackrel{}{=}\Lambda^N_1$. The lower bound on $\xi_t$ in the theorem is thus equivalent to 
    $$\lim_{K\rightarrow 0}\liminf_{N\rightarrow\infty }\mathbb{P}(\Lambda^{N}_1\geq KN^{1/3})=1.$$ Consider for a fixed $N$, the function $\Gamma_N(l^{c,-1})$. From an analogous argument to the proof of Lemma \ref{lem:aplus} it follows that 
    $$\mathbb{P}(\sup_{t\leq 1}(\Gamma(l^{c,-1})_t-\Gamma_N(l^{c,-1})_t)\geq x)\leq 2e^{-\frac{Nx^2}{2(\Gamma(l^{c,-1})_1+x/3)}}.$$

    We choose $c=KN^{1/3}$, and $x=1/(2c)=\frac{1}{2KN^{1/3}}$. With this choice the above bound becomes $$\mathbb{P}(\sup_{t\leq 1}\Gamma(l^{KN^{1/3},-1})_t-\Gamma_N(l^{KN^{1/3},-1})_t\geq x)\leq 2e^{-\frac{1}{8(K^3+K^2/(6N^{1/3}))}}.$$  

    Further, applying Lemma \ref{critlem} equation \eqref{critlemcd} to lower bound $\Gamma(l^{c,-1})$, we obtain 
    $$\Gamma(l^{KN^{1/3},-1})_t\geq l^{KN^{1/3},-1}_t+\frac{1}{2KN^{1/3}},$$
    so that
    \begin{align*}\mathbb{P}(\sup_{t\leq 1}(l^{KN^{1/3},-1}_t-\Gamma_N(l^{KN^{1/3},-1})_t)\geq 0)&\leq \mathbb{P}(\sup_{t\leq 1}(\Gamma(l^{KN^{1/3},-1})_t-\Gamma_N(l^{KN^{1/3},-1})_t)\geq x)\\&\leq  2e^{-\frac{1}{8(K^3+K^2/(6N^{1/3}))}}.\end{align*}   
    Consider the event
    $$\Gamma_N(KN^{1/3}t-1)_t\geq KN^{1/3}t-1\textrm{ for }t\leq 1.$$
    Applying Lemma \ref{Alemmatoboundathing}, on this event it follows that $$\Lambda^{N}_1\geq N^{1/3}K-1.$$

    In total we obtain
    $$\mathbb{P}(\Lambda^N_1\geq KN^{1/3}-1)\geq 1- 2e^{-\frac{1}{8(K^3+K^2/(6N^{1/3}))}}.$$
    Sending $N\rightarrow\infty$, and $K\rightarrow0$ then gives the desired bound
    $$\lim_{K\rightarrow 0}\liminf_{N\rightarrow\infty }\mathbb{P}(\Lambda^N_1\geq KN^{1/3})=1.$$

\underline{Upper bound:}

As above, using the rescaling of $\xi_t$, the upper bound on $\xi_t$ in the theorem is equivalent to $$\lim_{K\rightarrow\infty}\limsup_{N\rightarrow\infty}\mathbb{P}(\Lambda^N_1\leq KN^{1/3})=1.$$

Recall from Section \ref{setupsection}, that $W_0^i$ are the points of a Poisson process of intensity $N$.
We define random CDF type functions $F^N,F$ by
\begin{align*}
&F^N(x)=\frac{1}{N}\sum_{i=1}^\infty\mathbbm{1}_{\{W_0^{i,N}\leq x\}},\\
&F(x)=NF^N(x/N).
\end{align*}
The process $(F(x))_{x>0}$ is a standard Poisson process. 

We consider a random function $f_t^{c,m}$, defined by 
\begin{align*}
    &f_t^{c,m}=ct+\varsigma_m,\\
    &\varsigma_m=\inf\{x:F^N(x)-x\leq -m\}.
\end{align*}

Since $(F^N(x+\varsigma_m)-F^N(\varsigma_m))_{x\geq 0}$ is independent of $(F^N(x))_{x\leq \varsigma_m}$, it holds that 
$$\Gamma_N({f_t^{c,m}})_t\stackrel{d}{=}\tilde \Gamma_N(l^c)_t+\varsigma_m-m,$$ where $\tilde \Gamma_N(l^c)_t=\frac{1}{N}\sum_{i=0}^\infty\mathbbm{1}_{\{W_0^{i,N}>\varsigma_m,W_0^{i,N}-\varsigma_m+\inf_{s\leq t}(B_s^{i,N}-cs)\leq 0\}}$ is identically distributed to $\Gamma_N(l^c)$, and is independent of $\varsigma_m$. We choose $c=N^{1/3}$ and $m=(K_1+1/2)N^{-1/3}$, where $K_1$ will be chosen later.

As in the previous step, we may determine 
$$\mathbb{P}(\sup_{t\leq 1}(\tilde \Gamma_N(l^c)_t-\Gamma(l^c)_t)>x)\leq 2e^{-\frac{Nx^2}{2(c+1/2c+x/3)}}.$$
Recalling from Lemma \ref{critlem} equation \eqref{critlemc} that $\Gamma({ct})_t\leq ct+1/2c$, we then obtain that 
\begin{align*}
\mathbb{P}(\sup_{t\leq 1}(\tilde \Gamma_N(l^{N^{1/3}})_t-N^{1/3}t)\geq m)&\leq\mathbb{P}(\sup_{t\leq 1}(\tilde \Gamma_N(l^{N^{1/3}})_t-\Gamma(l^{N^{1/3}})_t)\geq K_1N^{-1/3})
\\
&\leq 2e^{-\frac{K_1^2}{2(1+N^{-1/3}+3K_1N^{-1/3})}}
\end{align*}

Consider the event $\{\tilde \Gamma_N(l^{N^{1/3}})_t\leq N^{1/3}t+m$ for $t\leq 1\}$. On this event it follows that $\Gamma_N{(f_t^{c,m})}_t\leq N^{1/3}t+\varsigma_m=f_t^{c,m}$. Further, applying Lemma \ref{Alemmatoboundathing} it also follows that $\Lambda^{
N}_1\leq \varsigma_m+N^{1/3}$.
Therefore, 
\begin{equation}\label{eqtncrit1}\mathbb{P}(\Lambda^{N}_1\leq \varsigma_m+N^{1/3})\geq \mathbb{P}(\sup_{t\leq 1}(\tilde \Gamma_N(l^{N^{1/3}})_t-N^{1/3}t)<m)\geq 1- 2e^{-\frac{K_1^2}{2(1+N^{-1/3}+3K_1N^{-1/3})}}\end{equation}

To bound $\varsigma_m$, we first note:
\begin{align*}
    \varsigma_m&=\inf\{x:F^N(x)-x\leq -m\},\\
    &=\inf\{x:F(Nx)-Nx\leq -Nm\},\\
    &=\inf\{x:F(Nx)-Nx\leq -N^{2/3}(K_1+1/2)\},\\
    &=N^{1/3}\inf\left\{x:\frac{F(N^{4/3}x)-N^{4/3}x}{N^{2/3}}\leq -(K_1+1/2)\right\}.
\end{align*}

By Lemma \ref{lemmabm}, ${B^N_0(x)=(F(N^{4/3}x)-N^{4/3}x)}/{N^{2/3}}$ converges weakly to a standard Brownian motion. Applying \cite[Lemma 5.6]{DIRT1}, we obtain that for any $y\geq 0$:
\begin{align*}\lim_{N\rightarrow\infty}\mathbb{P}(\varsigma_m\leq N^{1/3}y)&=\lim_{N\rightarrow\infty}\mathbb{P}\left(\inf\left\{x:\frac{F(N^{4/3}x)-N^{4/3}x}{N^{2/3}}\leq(K_1+1/2)\right\}\leq y\right),\\
&=\mathbb{P}(\inf_{x\leq y}B_x\leq -(K_1+1/2)),\\
&= 2\left(1-\Phi\left(\frac{K_1+1/2}{\sqrt{y}}\right)\right),\\
&\geq 1-\frac{K_1+1/2}{\sqrt{2\pi y}},
\end{align*}
where $B$ denotes a standard Brownian motion.

Combining this with the inequality \eqref{eqtncrit1} yields
\begin{align*}
&\liminf_{N\rightarrow\infty }\mathbb{P}(\Lambda^N_1\leq (y+1)N^{1/3})\\
&\geq \liminf_{N\rightarrow\infty} \mathbb{P}(\Lambda^N_1\leq \varsigma_m+N^{1/3},\varsigma_m\leq yN^{1/3}),\\
&\geq 1-\liminf_{N\rightarrow\infty }\mathbb{P}(\Lambda^N_1\leq \varsigma_m+N^{1/3})-\liminf_{N\rightarrow\infty}\mathbb{P}(\varsigma_m>N^{1/3}y),\\
&\geq 1-2e^{-\frac{K_1^2}{2}}-\frac{K_1+1/2}{\sqrt{2\pi y}}.
\end{align*}

Taking $y=K-1$, and setting $K_1=K^{1/3}$, we therefore obtain that

$$\lim_{K\rightarrow\infty}\liminf_{N\rightarrow\infty }\mathbb{P}(\Lambda^N_1\leq KN^{1/3})=1.$$
\end{proof} 

We now prove Lemma \ref{lemmabm2}.
\begin{proof}[Proof of Lemma \ref{lemmabm2}]
    We shall show this result holds when $k=1$. It is simple to extend to the case of finite $k$. It is then also simple to extend to the case of $((B_t^N(x))_{x\in\mathbb{R}^+})_{t\in\mathbb{R^+}}$ by noting that any subsequence of this process has a convergent subsequence, and that the finite dimensional distributions converge appropriately. 
    
    From Lemma \ref{lemmabm}, for any $t\leq T$, the process $(B_0^N(x),B_{t}^N(x))_{x\in\mathbb{R^+}}$ converges weakly to a 2 dimensional Brownian motion $(B_0(x),B_t(x))_{x\in\mathbb{R^+}}$, with $(B_0(0),B_t(0))=0$. We claim that for any fixed $x>0$, 
    $$\mathbb{P}(|B_0(x)-B_t(x)|>0)=0.$$ Taking a union over $x \in \mathbb{Q}$, and then applying the continuity of the processes, this is sufficient to establish that $\mathbb{P}(B_t(x)=B_0(x)\hspace{0.2cm} \forall x\geq 0)=1.$

    To prove our claim, we first calculate exactly the first two moments of \\$\left(F^N_0(N^{1/3}x)-F^N_t(N^{1/3}x)\right)$. These are given by:
    \begin{align*}
    &\mathbb{E}\left(N^2\left(F^N_0(N^{1/3}x)-F^N_t(N^{1/3}x)\right)^2\right)=N^2H_1^2+NH_1+2NH_2,\\
    &\mathbb{E}\left(N\left(F^N_0(N^{1/3}x)-F^N_t(N^{1/3}x)\right)\right)=-NH_1,\\
    &H_1=\frac{\sqrt{t}}{\sqrt{2\pi}}e^{-\frac{N^{2/3}x^2}{2t}}-\frac{N^{1/3}x}{2}\left(1-\textrm{erf}\left(\frac{N^{1/3}x}{\sqrt{2t}}\right)\right),\\
    &H_2=\frac{N^{1/3}x}{2}\left(1-\textrm{erf}\left(\frac{N^{1/3}x}{\sqrt{2t}}\right)\right)+\frac{\sqrt{t}}{\sqrt{2\pi}}\left(1-e^{-\frac{N^{2/3}x}{2t}}\right).
    \end{align*}

    We may therefore bound 
    \begin{align*}
        \mathbb{E}\left(\left(B^N_0(x)-B^N_t(x)\right)^2\right)=\frac{H_1+2H_2}{N^{1/3}}\leq \frac{C_t}{N^{1/3}}.
    \end{align*}
     Above, $C_t$ is a constant depending only on $t$, and which is uniform in $N$.

    Fix $\epsilon>0$. Applying the Portmanteau theorem, and then Markov's inequality, we determine:
    \begin{align*}
        \mathbb{P}(|B_0(x)-B_t(x)|>\epsilon)\leq \liminf_{N\rightarrow\infty}  \frac{\mathbb{E}\left(\left(B^N_0(x)-B^N_t(x)\right)^2\right)}{\epsilon^2}=0.
    \end{align*}

    Sending $\epsilon\downarrow 0$ then proves the claim that $B_0(x)=B_t(x)$ almost surely.  
\end{proof}

\begin{proof}[Proof of Lemma \ref{lemma1incritlabel}]
We first give a Chernoff bound on $M$.
\begin{lemma}
\label{MNbd}
 There exists a constant $C_l$ such that, for arbitrary $x\geq N^{-1/3}$, and $z>0$, we may bound via a Chernoff bound
\begin{align}
&\mathbb{P}\left(NM_x^{N,y}-\mathbb{E}(NM_x^{N,y})\geq N^{2/3}z\right)\leq \exp\left(-N^{1/6}C_lz^2\right).
\end{align}
\end{lemma}
This is proven directly below the current proof.

Through analogous arguments to those of Lemma \ref{lem:aplus}, it holds that $G_x^{N,y}-\mathbb{E}(G_x^{N,y})$, and $H_x^{N,y}-\mathbb{E}(H_x^{N,y})$ are martingales under their natural filtrations. Further, for arbitrary $0\leq s\leq m$,

\begin{align*}&\mathbb{P}(\sup_{s\leq x\leq m}\left|G_x^{N,y}-G_s^{N,y}-\mathbb{E}(G_x^{N,y}-G_s^{N,y})\right|>z)\leq 2e^{-\frac{Nz^2}{2\left(\mathbb{E}\left(G_m^{N,y}-G_s^{N,y}\right)+z\right)}},\\
&\mathbb{P}(\sup_{s\leq x\leq m}\left|H_x^{N,y}-H_s^{N,y}-\mathbb{E}(H_x^{N,y}-H_s^{N,y})\right|>z)\leq 2e^{-\frac{Nz^2}{2\left(\mathbb{E}\left(H_m^{N,y}-H_s^{N,y}\right)+z\right)}}.
\end{align*}

It is simple to determine that $\mathbb{E}\left(H_m^{N,y}-H_s^{N,y}\right)\leq \frac{N^{1/3}(m-s)}{2a}$, while, 

\begin{align*}\mathbb{E}\left(G_m^{N,y}-G_s^{N,y}\right)&=\mathbb{E}\left(\int_{y+\sup_{{x\leq s}}\left(\frac{N^{1/3}x}{2a}+\hat B_x\right)}^{y+\sup_{x\leq m}(\frac{N^{1/3}x}{2a}+\hat B_x)}\Phi\left(\frac{z}{\sqrt{t}}\right)dz\right)\leq \frac{N^{1/3}(m-s)}{2a}+\frac{\sqrt{2(m-s)}}{\sqrt{\pi}}.\end{align*}

Therefore, the previous bound becomes

\begin{align}
&\mathbb{P}(\sup_{s\leq x\leq m}\left|G_x^{N,y}-G_s^{N,y}-\mathbb{E}(G_x^{N,y}-G_s^{N,y})\right|>z)\leq 2e^{-\frac{Nz^2}{2\left(N^{1/3}(m-s)/(2a)+\sqrt{m-s}+z\right)}},\label{Geqtn}\\
&\mathbb{P}(\sup_{s\leq x\leq m}\left|H_x^{N,y}-H_s^{N,y}-\mathbb{E}(H_x^{N,y}-H_s^{N,y})\right|>z)\leq 2e^{-\frac{Nz^2}{2\left(N^{1/3}(m-s)/(2a)+z\right)}}.\label{Heqtn}
\end{align}

For any $y,l\in \mathbb{R}^+$ we bound
\begin{align*}
\mathbb{P}&\left(\sup_{x\leq l}M^{N,y}_x-\mathbb{E}(M^{N,y}_x)>\frac{\epsilon}{3N^{1/3}}\right)\\
     \leq \sum_{i=0}^{\lceil N^{1/3}l\rceil }&\mathbb{P}\left(M_{iN^{-1/3}}^{N,y}-\mathbb{E}(M_{iN^{-1/3}}^{N,y})\geq \frac{\epsilon}{6N^{1/3}}\right)\\
    &+\mathbb{P}\left(\sup_{iN^{-1/3}\leq x\leq ({i+1})N^{-1/3}}(M^{N,y}_x-\mathbb{E}(M_x^{N,y}))-(M^{N,y}_{iN^{-1/3}}-\mathbb{E}(M_{iN^{-1/3}}^{N,y}))\geq \frac{\epsilon}{6N^{1/3}}\right),\\
    \leq \sum_{i=0}^{N^{1/3}l}&\mathbb{P}\left(M_{iN^{-1/3}}^{N,y}-\mathbb{E}(M_{iN^{-1/3}}^{N,y})\geq \frac{\epsilon}{6N^{1/3}}\right)\\
    &+\mathbb{P}\left(\sup_{iN^{-1/3}\leq x\leq (i+1)N^{-1/3}}| G_x^{N,y}-G_{iN^{-1/3}}^{N,y}-\mathbb{E}( G_x^{N,y}-G_{iN^{-1/3}}^{N,y})|\geq \frac{\epsilon}{12N^{1/3}}\right)\\
    &+\mathbb{P}\left(\sup_{iN^{-1/3}\leq x\leq (i+1)N^{-1/3}}|H_x^{N,y}-H_{iN^{-1/3}}^{N,y}-\mathbb{E}(H_x^{N,y}-H_{iN^{-1/3}}^{N,y})|\geq \frac{\epsilon}{12N^{1/3}}\right).
\end{align*}

 Applying the bounds from the inequalities Lemma \ref{MNbd}, \eqref{Geqtn}, and \eqref{Heqtn}, we obtain that for $\epsilon\leq 1$ this is bounded above by  

\begin{align*}
    &lN^{1/3} \left(e^{-{N^{1/6}C_l\frac{\epsilon^2}{36}}}+2e^{-\frac{N^{1/3}\epsilon^2}{4a^{-1}+16}}+2e^{-\frac{N^{1/3}\epsilon^2}{a^{-1}+2}}\right)\\
    &\leq 5lN^{1/3} e^{-\tilde CN^{1/6}\epsilon^2} \textrm{ for all sufficiently large }N.
\end{align*}
Here $\tilde C$ is some positive constant. Finally, taking a union bound over $y\in \frac{1}{N}\mathbb{Z}$ such that $C_2N^{1/3}\leq y\leq C_1N^{1/3}$ yields  

\begin{align*}&\mathbb{P}\left(\sup_{\{y\leq C_1N^{1/3},\textrm{ }y\in \frac{1}{N}\mathbb{N}\}}\sup_{\{x\leq l\}}\left(M_x^{N,y}-\mathbb{E}(M_x^{N.y})\right)>\frac{\epsilon}{3N^{1/3}}\right)\\
&\leq (C_1-C_2)N^{5/3}e^{-\tilde CN^{1/6}\epsilon^2}.
\end{align*}
This converges to $0$, as $N\rightarrow\infty$.

\end{proof}

\begin{proof}[Proof of Lemma \ref{MNbd}]

We define 
\begin{align*}
    &\tilde G_{x}^{N,y}={\frac{1}{N}\sum_{i=1}^\infty\mathbbm{1}_{\{W_t^i\leq y+\sup_{s\leq x}(\frac{sN^{1/3}}{2a}+\hat B_s^{i,N})+\beta N^{1/3},W_t^{i,N}>y+\beta N^{1/3}+\frac{N^{1/3}x}{2a}\}}}_,\\
&\tilde H_x^{N,y}={\frac{1}{N}\sum_{i=1}^\infty \mathbbm{1}_{\{W_t^{i,N}\leq y+\beta N^{1/3}+\frac{N^{1/3}x}{2a},W_t^{i,N}> y+\sup_{s\leq x}(\frac{sN^{1/3}}{2a}+\hat B_s^{i,N})+\beta N^{1/3}\}}}_.\\
\end{align*}

It holds that $N\tilde G_x^{N,y}$, and $N\tilde H_x^{N,y}$ are independent Poisson random variables, with
$$M^{N,y}_x=\tilde G_x^{N,y}-\tilde H^{N,y}_x.$$

We define constants $c_x^{N,y},d_x^{N,y}$ by:
\begin{align*}
&p_z=\mathbb{P}\left(z\leq y+\sup_{s\leq x}\left(\frac{N^{1/3}s}{2a}+\hat B_s+\beta N^{1/3}\right)\right),\\
    &c_x^{N,y}=\mathbb{E}(\tilde G_x^{N,y})=\int_{y+\beta N^{1/3}+\frac{N^{1/3x}}{2a}}^\infty p_z\Phi\left(\frac{z}{\sqrt{t}}\right)dz,\\
    &\quad\quad= \frac{\sqrt{x}}{\sqrt{2\pi}}+\frac{a}{2N^{1/3}}+\frac{a}{2N^{1/3}}e^{\frac{N^{2/3}x}{2a^2}}\left(1-\textrm{erf}\left(\frac{N^{1/3}\sqrt{x}}{\sqrt{2}a}\right)\right) +L_{x}^{N,y},\\
    &|L_x^{N,y}|\leq \frac{\sqrt{t}}{\sqrt{2\pi}}e^{-\frac{N^{2/3}\beta^2}{2t}},\\
    &d_x^{N,y}=\mathbb{E}(\tilde H_x^{N,y})=\int_{y+\beta N^{1/3}}^{y+\beta N^{1/3}+\frac{N^{1/3}x}{2a}}(1-p_z)\Phi\left(\frac{z}{\sqrt{t}}\right)dz,\\
    &\quad \quad =c_x^{y,N}+\frac{N^{1/3}x}{4a}\left(1-\textrm{erf}\left(\frac{N^{1/3}\sqrt{x}}{\sqrt{8}a}\right)\right)-\frac{\sqrt{x}}{\sqrt{2\pi}}e^{-\frac{N^{2/3}x}{2}}-\frac{a}{N^{1/3}}\textrm{erf}\left(\frac{N^{1/3}\sqrt{x}}{\sqrt{8}a}\right)+\hat L_x^{N,y},\\
    &|\hat L_x^{N,y}|\leq \frac{\sqrt{t}}{\sqrt{2\pi}}e^{-\frac{N^{2/3}\beta^2}{2t}}.
\end{align*}
Here $\hat B$ denotes a standard Brownian motion.
It is simple to calculate that 
\begin{align*}
    \mathbb{E}(M_x^{N,y})&=c_x^{N,y}-d_x^{N,y},\\
    &=\int_{y+\beta N^{1/3}}^\infty p_z\Phi\left(\frac{z}{t}\right)dz-\int_{y+\beta N^{1/3}}^{y+\beta N^{1/3}+\frac{N^{1/3}x}{2a}}\Phi\left(\frac{z}{t}\right)dz,\\
    &\geq \Phi\left(\frac{\beta N^{1/3}}{\sqrt{t}}\right)\left(\Gamma\left(\frac{Nt^{1/3}}{2a}\right)_x-\frac{N^{1/3}x}{2a}\right),\\
    &\geq 0.
\end{align*}

We may also calculate

\begin{align*}
    c_x^{N,y}-d_x^{N,y}&\leq \left(\Gamma\left(\frac{N^{1/3}t}{2a}\right)_x-\frac{N^{1/3}x}{2a}\right)+ \frac{N^{1/3}\beta}{\sqrt{2\pi}}e^{-\frac{N^{1/3}\beta^2}{2t}},\\
    &\leq \frac{a}{N^{1/3}}+\frac{N^{1/3}\beta}{\sqrt{2\pi}}e^{-\frac{N^{1/3}\beta^2}{2t}}.
\end{align*}

The final line follows via Lemma \ref{critlem} equation \eqref{critlemc}.

We compute \begin{align*}
    \varnothing_x^{N,y}(u)=\mathbb{E}(\exp(uN(M_x^{N,y})))=\exp(Nc_x^{N,y}(e^u-1)+Nd_x^{N,y}(e^{-u}-1)).
\end{align*}

We denote $c=Nc^{
N,y}_x$, $d=Nd^{N,y}_x$, and $b=N^{2/3}z+c-d$.

By Markov's inequality we bound 
\begin{align*}
&\mathbb{P}\left(N(M_x^{N,y})-\mathbb{E}(N(M_x^{N,y}))\geq N^{2/3}z\right),\\
&\leq \min_{u\geq 0}\varnothing_x^{N,y}(u)\exp\left(-\mathbb{E}\left(u\left(N(M_x^{N,y})-N^{2/3}z\right)\right)\right).
\end{align*}
The minimum of 
$$\varnothing_x^{N,y}(u)\exp(-\mathbb{E}\left(u(N(M_x^{N,y})-N^{2/3}z)\right)$$
occurs when  
$$\tilde x=e^u=\frac{b}{2c}+\sqrt{1+\frac{b^2-4c(c-d)}{4c^2}}.$$

For this value of $u$, the function takes value 
$$\exp((\tilde x-1)(c-d/\tilde x)-\log(\tilde x)b).$$

It holds that $\log(\tilde x)\geq \frac{\tilde x-1}{\tilde x}$, hence we may further bound this by 
$$\exp((\tilde x-1)(c(\tilde x-1)-N^{2/3}z)/\tilde x).$$

We may bound, 
\begin{align*}
    \frac{N^{2/3}z}{2c}-\frac{(c-d)^2}{2c^2}\leq \tilde x-1\leq \frac{N^{2/3}z}{2c}+\frac{b^2}{8c^2}.
\end{align*}

For any $\epsilon>0$, and sufficiently large $N$, $c-d\leq (a+\epsilon)N^{2/3}$. As long as $N^{2/3}z/2\geq (N^{2/3}z+c-d)^2/8c$, we may further bound

\begin{align*}&\exp((\tilde x-1)(c(\tilde x-1)-N^{2/3}z)/\tilde x)\\
&\leq \exp\left(-\frac{cN^{2/3}z-(a+\epsilon)^2N^{4/3}}{2c^2+cN^{2/3}z+(N^{2/3}z+(a+\epsilon)N^{2/3})^2}\left(\frac{N^{2/3}z}{2}-\frac{(N^{2/3}z+(a+\epsilon)N^{2/3})^2}{8c}\right)\right)\end{align*}

 For any $\gamma>0$, and $x\geq N^{-2/3+\gamma}$, we may also bound $2N(\sqrt{l}+1)\geq c\geq N^{2/3+\gamma/2}/\sqrt{2\pi}$.
 
 The above bound, can then be bounded by
\begin{align*}
\exp\left(-\frac{N^{\gamma/2}(z-2N^{-\gamma/2}(a^2+\epsilon^2))}{C_lN^{2/3}}\left(\frac{N^{2/3}(z-N^{-\gamma/2}(a^2+z^2+\epsilon^2))}{2}\right)\right),\end{align*}

for some constant $C_l$.

Therefore, there exists some constant, $C_l>0$, such that for arbitrary $x\geq N^{-1/3}$, $z>0$, and all sufficiently large $N$,
\begin{align}
&\mathbb{P}\left(N(M_x^{N,y})-\mathbb{E}((N(M_x^{N,y}))\geq N^{2/3}z\right),\\
&\leq \exp\left(-N^{1/6}C_lz\right).\nonumber
\end{align}
\end{proof}

\begin{proof}[Proof of Lemma \ref{lemmaonMmean}]
    
We compute:
\begin{align*}
    \mathbb{E}(G_x^{N,y})&=\int_{-\infty}^\infty \mathbb{P}\left(z\leq y+\sup_{s\leq x}\left(\frac{N^{1/3}s}{2a}+B_s+\beta N^{1/3}\right)\right)\Phi\left(\frac{z}{\sqrt{t}}\right)dz,\\
    &\leq \int_0^\infty \mathbb{P}\left(z\leq\sup_{s\leq x}\left(\frac{N^{1/3}s}{2a}+B_s\right)\right)dz+\int^{\beta N^{1/3}+y}_{-\infty}\Phi\left(\frac{z}{\sqrt{t}}\right)dz,\\
    &\leq \frac{N^{1/3}x}{2a}+\beta N^{1/3}+\frac{a}{N^{1/3}}+y+e^{-\frac{N^{2/3}\beta^2}{2t}}\text{, applying the result of eqtn. \eqref{critlemc}}.
\end{align*}
We also compute:
\begin{align*}\mathbb{E}(H^{N,y}_x)&=\int_{-\infty}^{y+\frac{xN^{1/3}}{2a}+\beta N^{1/3}}\Phi\left(\frac{z}{\sqrt{t}}\right)dz,\\
&=y+\frac{xN^{1/3}}{2a}+\beta N^{1/3}+\int_{y+\frac{xN^{1/3}}{2a}+\beta N^{1/3}}^\infty \Phi\left(\frac{-z}{\sqrt{t}}\right) dz,\\
&\geq y+\frac{xN^{1/3}}{2a}+\beta N^{1/3}.
\end{align*}

Subtracting these from each other yields the result of the lemma.
\end{proof}
\begin{proof}[Proof of Lemma \ref{lemmaz}]
    
$NZ^{N,y}$ is Poisson distributed, with mean given by:
$$\mathbb{E}(Z^{N,y})=\int_0^\infty 1-\Phi\left(\frac{x+N^{1/3}\beta}{\sqrt{t}}\right)dx+\int_0^y 1-\Phi\left(\frac{x+N^{1/3}\beta}{\sqrt{t}}\right)dx\leq \frac{\sqrt{2t}}{\sqrt{\pi}}e^{-\frac{N^{2/3}\beta^2}{2t}}.$$

Applying standard Poisson tail bounds, followed by a union bound over $y$, we obtain that
$$\mathbb{P}\left(\sup_{\{y\leq C_1N^{1/3},\textrm{ }y\in \frac{1}{N}\mathbb{N}\}}Z^{N,y}>\frac{\epsilon}{2N^{1/3}}\right)\leq BN^{4/3}e^{-\frac{N^{1/3}\epsilon^2}{2(\sqrt{2t}e^{-\beta^2N^{2/3}/(2t)}+\epsilon N^{-1/3})}}.$$

 This converges to 0 exponentially quickly as $N\rightarrow \infty$.
\end{proof}

We prove Corollary \ref{corupper}, which follows from Proposition \ref{propcritupper} by elementary arguments.
\begin{proof}[Proof of Corollary \ref{corupper}]

We introduce subsets of $\{(r,(b_t)_{t\in \hat J}):r\in \tilde D,\textrm{ }(b_t)_{t\in\hat J}\in D(\mathbb{R}^+)^\infty\}.$ For a fixed $t\in\hat J$, and fixed $a,l,\beta,\epsilon>0$ and $0<4C_1\leq C_2$, we define sets:
\begin{align*}&A_1=\{\sup_{s\leq l}(r_{t+s}-r_t-\beta-s/2a)>0\},\\
&A_2=\{\sup_{x\leq l/2a+\beta}b_t(r_t+x)>a+\epsilon\},\\
&A_3=\{r_t>2C_2\}\cap \{r_T<C_2/2\},\\
&A=A_1\cap A_2\cap A_3.
\end{align*}
Each of these sets are open. By the Portmanteau theorem we obtain:
\begin{align*}
    \mathbb{P}((R,(B_s)_{s\in\hat J})\in A)&\leq \liminf_{N\rightarrow\infty}   \mathbb{P}((R^N,(B^N_s)_{s\in\hat J})\in A)\\
    &\leq \liminf_{N\rightarrow\infty}\mathbb{P}(A^N)=0.
\end{align*}
The set $A^N$, is exactly the set given in Lemma \ref{propcritupper}.

Taking a union over $t\in\hat J$, and over rational $a$, $\beta$, $l$, $\epsilon,$ $C_1$ and $C_2$, we determine that almost surely, for all $t\in\hat J$, and all $a,l>0$

$$B_0(R_t+x)\geq a\textrm{ for }x\leq l/2a $$
$$\implies R_{t+s}\leq R_t+\frac{s}{2a}\textrm{ for }s\leq l. $$

For $t\notin \hat J$, we may establish that the result still holds by taking limits over $t\in \hat J$ and using that $R$ is cadlag. Suppose that $B_0(R_t+x)\geq a$ for $x\leq l/2a$.
For any $s\in\hat J$, with $0\leq R_s-R_t\leq l/2a$ it holds that $B_0(R_s+x)\geq a$ for $x\leq (l-2a(R_s-R_t))/2a$, from which it follows that 
$$R_{s+x}\leq R_s+\frac{x}{2a}\textrm{ for }x\leq l-2a(R_s-R_t),$$
hence $$R_{t+x}\leq R_t+(R_s-R_t)+\frac{x}{2a}\textrm{ for }x\leq l-2a(R_t-R_s).$$
Taking a sequence $s\in\hat J$ which decreases to $t$, we obtain that 
$$R_{t+x}\leq R_t+\frac{x}{2a}\textrm{ for }x\leq l.$$

\end{proof}

\begin{proof}[Proof of Lemma \ref{deriv}]
    The first statement follows by taking limits over $l$ and $a$ appropriately in Corollary \ref{corlower} and Corollary \ref{corupper}. To prove the second statement, we first consider an arbitrary $\epsilon>0$, so that $B_0({R_l})>0$ for $l\in [s+\epsilon,t-\epsilon]$. We obtain that 
    $$\int_{R_s}^{R_l}{2(B_0({R_s+x}))_+}dx=(l-s)+c$$ for a constant $c$. The constant $c$ is immediately identified as 0. Since $R$ is non-decreasing, and $(B_0)_+$ is non-negative, this identifies that 
    \begin{align*}R_l-R_{s+\epsilon}&=\inf\left\{x:\int_0^x(2B_0(R_{s+\epsilon}+x))_+dx>l-s-\epsilon\right\},\\
    &=\inf\left\{x:\int_0^x(2B_0(R_{s+\epsilon}+x))_+dx=l-s-\epsilon\right\} \textrm{ for }   l\in[s+\epsilon,t-\epsilon].
    \end{align*}
    We send $\epsilon\downarrow 0$, and apply the right-continuity of $R$ to obtain that for $s< l<t$
     \begin{align*}
     R_l-R_{s}&=\lim_{\epsilon\downarrow 0} \inf\left\{x:\int_{ R_{s+\epsilon}}^x2(B_0(y))_+dy>t-s-\epsilon\right\},\\
    &=\inf\left\{x:\int_{ R_{s}}^x2(B_0(y))_+dy\geq t-s\right\},\\
   &=\inf\left\{x:\int_{ R_{s}}^x2(B_0(y))_+dy> t-s\right\}.\end{align*}
    The final line follows since $R$ must be continuous on $[s,t)$.
\end{proof}
\begin{proof}[Proof of Lemma \ref{jump}]

To prove the first statement, we take a sequence $t_n\uparrow t$ so that \\$\sup_{x\leq (R_{t^-}-R_{t_n})} B_0({R_{t_n}+x})\downarrow 0$. Fix an arbitrary $\epsilon>0$. For all sufficiently large $n$, $\sup_{x\leq R_{t^-}-R_{t_n}+\varpi}B_0({R_{t_n}+x})\leq \epsilon.$
By Corollary \ref{corlower}, it follows that
$$R_{t_n+s}\geq R_{t_n}+s/2\epsilon\textrm{ for }s\leq 2\epsilon(R_{t^-}-R_{t_n}+\varpi).$$
Sending $n\rightarrow \infty$ in the above equation, we obtain that
$$R_{t^-+2\epsilon\varpi}\geq R_{t^-}+\varpi.$$
Since $\epsilon$ was arbitrary, sending $\epsilon\downarrow0$, and applying the right continuity of $R$ yields that
$$R_t\geq R_{t^-}+\omega.$$

The second statement holds through a slight adaptation of Proposition \ref{propcritupper}, and the use of this Proposition in Corollary \ref{corupper}. In this proposition, the set $A_2$ may be replaced by $$\{B^N_t(\beta+x)>a+\epsilon \textrm{ for }x\leq l/2a \},$$ and the result of the proposition follows through identical arguments. Keeping $\beta>0$, rather than sending it to zero as in Corollary \ref{corupper}, we then obtain that almost surely, for any time $t$,
\begin{align*}
&B_0(R_t+\beta+x)\geq a\textrm{ for }x\leq l/2a,\\
&\implies R_{t+x}\leq R_t+\beta +x/2a\textrm{ for }x\leq l.
\end{align*}

Fix an arbitrary $\gamma>0$ and choose $\epsilon_\gamma$ such that $$B_0(R_{t^-}+\inf\{x:B_0({R_{t^-}+x})>0\}+\epsilon_\gamma)\geq \gamma.$$ For a sequence $\gamma\downarrow 0$, $\epsilon_\gamma$ can be chosen such that $\epsilon_\gamma\rightarrow 0$. We choose a further $\epsilon'>0$ such that $$B_0(R_{t^-}+\inf\{x:B_0({R_{t^-}+x})>0\}+\epsilon_\gamma+y)\geq \gamma/2\textrm{ for }y\leq \epsilon'.$$ We take a sequence of times $t_n\uparrow t$, and a sequence $$\beta_{n}=\inf\{x:B_0({R_{t^-}+x})>0\}+\epsilon_\gamma+R_t-R_{t_n},$$
so that for each $n$:
$$B_0(R_{t_n}+\beta_n+y)\geq \gamma/2\textrm{ for }y\leq \epsilon'.$$
We therefore obtain that 
$$R_{t_n+x}\leq R_{t_n}+x/2\gamma+\beta_n \textrm{ for }x\leq 2\gamma \epsilon'.$$ 
Taking a limit over $n\rightarrow \infty$, then sending $\epsilon'\downarrow 0$, we determine that
$$R_{t}\leq R_{t^-}+\inf\{x:B_0({R_{t^-}+x})>0\}+\epsilon_\gamma.$$
Further sending $\gamma\downarrow 0$, we then obtain the upper bound:
$$\Delta R_t\leq \inf\{x:B_0({R_{t^-}+x})>0\}.$$

\end{proof}

\begin{proof}[Proof of Lemma \ref{denseness}]
    
     The process $R$ is strictly increasing and cadlag. It is differentiable almost everywhere, and by Lemma \ref{deriv}
      the set $\{t\leq T:R_t \textrm{ is differentiable at }t\}$ is identical to the set $\{t\leq T:B_0({R_t})>0\}$.
 
    The set $\{x\leq R_T:B_0(x)>0\}$ is open, hence is a countable union of disjoint open intervals $\cup_{i\in\mathbb{N}}(c_i,d_i)$. Let us fix some $i\in\mathbb{N}$. By Lemma \ref{deriv}, for $R_t\in(c_i,d_i)$, $R$ is continuous at $t$. Therefore, $R^{-1}((c_i,d_i))$ is an open interval $(a_i,b_i)$, on which $R$ is continuous. It follows also that $R_{a_i}=c_i$ by applying right continuity. Since $R$ is strictly increasing, the set $\{t\leq T:R_t=c_i\}$ is at most a singleton, and therefore it is exactly $\{a_i\}$.
    Since $i$ was arbitrary, we obtain that $$\{t\leq T:B_{R_t}>0\}=R^{-1}(\{t\leq R_T:B_t>0\})=\cup_{i\in\mathbb{N}}(a_i,b_i)$$ is an open set, and that $R^{-1}(\cup_{i\in\mathbb{N}}\{c_i\})=\cup_{i\in\mathbb{N}}\{a_i\}$.

    The set $\{x\leq R_T:B_0(x)<0\}$ is open, and may be written as $\cup_{i\in\mathbb{N}}(e_i,f_i)$ for disjoint open sets $(e_i,b_i)$. By Lemma \ref{jump}, the preimage of this set, $\{t\leq T:B_0({R_t})<0\}$, must be empty. Consider the preimage of the set $\{e_i\}$ for arbitrary $i$. If, for some $t\leq T$ $R_t=e_i$, it must hold by right continuity that $R_{s}\in(e_i,f_i)$ for some time $s$. However, by Lemma \ref{jump} this is impossible, hence
    $R^{-1}(\cup_{i\in\mathbb{N}}[e_i,f_i))=\phi$. 

    All $t\leq R_T$ are either in a set of form $[c_i,d_i)$, $[e_i,f_i)$ for some $i$, or in $\{R_T\}$. Therefore, all $t<T$ are in the preimage of such a set, hence are in a set of form $[a_i,b_i)$ for some $i$.

\end{proof}
\section{Bounds on hitting linear barriers}
\label{hitting a line}
For non-negative positive constants $c $ and $d$ the value of $\Gamma(l^{c,-d})_t-(ct-d)$ is given explicitly by:
\begin{align}
\label{difference}
&\frac{\sqrt{t}}{\sqrt{2\pi}}e^{-\frac{(ct-d)^2}{2t}}-\frac{(ct-d)}{2}\left(1-\textrm{erf}\left(\frac{ct-d}{\sqrt{2t}}\right)\right)\\
&+\frac{1}{4c}\left(1+\textrm{erf}\left(\frac{ct-d}{\sqrt{2t}}\right)-\left(1-\textrm{erf}\left(\frac{ct+d}{\sqrt{2t}}\right)\right)e^{2cd}\right).\nonumber
\end{align}

The above has derivative given by:
$$-\frac{c}{2}\left(1-\textrm{erf}\left(\frac{ct-d}{\sqrt{2t}}\right)\right)+\frac{e^{-\frac{(ct-d)^2}{2t}}}{\sqrt{2\pi t}},$$

and second derivative given by
$$\frac{((cd-1)t+d^2)}{\sqrt{2^3\pi t^5}}e^{-\frac{(ct-d)^2}{2t}}.$$

First consider the case $cd\geq 1$. In this case, the derivative is initially negative, and the second derivative is always positive. It is clear from the explicit expression, that $\Gamma(l^{c,-d})_t-(ct-d)$ is bounded, and so the derivative must remain non-negative. Therefore, the minimum occurs as $t\rightarrow \infty$, yielding the lower bound
$$\inf_{t\geq 0}\Gamma(l^{c,-d})_t-(ct-d)\geq \frac{1}{2c}.$$
The maximum occurs at time 0, yielding the upper bound of $d$.

From the expression \eqref{difference}, and standard Gaussian tail bounds, it can also be seen that for any $d\geq 0$, $$\Gamma(l^{c,-d})_t-(ct-d)-\frac{1}{2c}\leq \tilde C(1+ct+\sqrt{t})e^{-\frac{c^2t}{2}}.$$
Therefore, for time $t$ greater than $1/c,$ $\Gamma(ct-d)-(ct-d)$ is within $3\tilde Ce^{-c/2}$ of $1/(2c)$.

In the case of $d=0$, the derivative is initially negative, while the second derivative is always positive. Hence the maximum is met at some time $t^*$ where the derivative is zero. Substituting in the condition for zero derivative into the expression for $\Gamma(l^c)_t-ct$ yields
$$\sup_{t\geq 0}\Gamma(l^c)_{t}-ct=\frac{\textrm{erf}\left(\frac{c\sqrt{t^*}}{\sqrt{2}}\right)}{2c}\leq \frac{1}{2c}.$$

From the expression \eqref{difference}, and standard Gaussian tail bounds, it is immediate that 
$$\frac{1}{2c}-(\Gamma(l^c)_t-ct)\leq \tilde Ce^{-\frac{cx^2}{2}}$$ for a constant $\tilde C$. Therefore, for time $t\geq 1/c$, $\Gamma(l^c)_t-ct$ is within $2\tilde Ce^{-c/2}$ of $1/(2c)$.

\section{Blow ups for all time}
\label{infiniteblowup}

We take constants $L\geq 10$, and $\alpha=1/4$.
Consider the density $h$ given by
\begin{align*}
&h(x)=\begin{cases}
    0\quad x\in \cup_{i\in\mathbb{N}_0}[\hat x_{2i-1},\hat x_{2i}]  , \\
    (1-\alpha)L \quad x\in\cup_{i\in\mathbb{N}_0} [\hat x_{2i},\hat x_{2i+1}],
\end{cases}\\
&\hat x_i=\sum_{j=0}^iL^j=\frac{L^{i+1}-1}{L-1},\quad i\in\mathbb{N}_0,\\
&\hat x_{-1}=0.
\end{align*}

This density satisfies condition \hyperlink{A3}{(A3)}. Indeed we compute that
\begin{align*}
&\int_0^{ \hat x_i}h(x)dx=L(1-\alpha)\sum_{j=0\textrm{, } j\textrm{ odd}}^iL^j=(1-\alpha)(L^2+L^4+...+L^{\lfloor \frac{i+1}{2}\rfloor} )\\
&=\begin{cases}
    (1-\alpha )L^2\frac{L^{i+1}-1}{L^2-1} \textrm{ for odd }i,\\
    (1-\alpha )L^2\frac{L^{i}-1}{L^2-1} \textrm{ for even }i,
\end{cases}\\
&=\begin{cases}
    (1-\alpha )\frac{L^2 \hat x_i}{L+1} \textrm{ for odd }i,\\
    (1-\alpha )\frac{L^2(L^i-1) \hat x_i}{(L+1)(L^{i+1}-1)}\textrm{ for even }i,
\end{cases}\\
&\quad\begin{cases}
    \geq \frac{L \hat x_i}{2} \textrm{ for odd }i,\\
    \leq (1-\alpha)\hat x_i\textrm{ for even }i,.
\end{cases}
\end{align*}

We may take $\hat x_{2n}$ to be $x_n$ in condition \hyperlink{A3}{(A3)}, and $y_n=(1-\alpha)\hat x_{2n}=(1-\alpha)x_n$. We thus see that condition \hyperlink{A3}{(A3)} is satisfied, and so physical solutions to \eqref{alternativetomvr} with initial density $g$ exist. From now on fix we such a physical solution $\Lambda_t$.

It is simple to determine that for any physical solution $\Lambda_t$ of \eqref{aligned:MVR}, $\Lambda_t\leq \hat x_i$ up to a time 
$$t=\frac{\pi\alpha^2}{2(1-\alpha)^2 L^2} \hat x_i^2.$$
It is also simple to determine that for even $i$,
$$\Lambda_t\geq \hat x_i\textrm{ for }t=\hat x_i^2.$$
To determine this lower bound we consider a function $f_t=-\epsilon$ on $(0,\hat x_i^2)$, and $f_t=\hat x_i-\epsilon$ on $(\hat x_i^2,\infty)$. Considering $\Gamma_N(f)$ as defined above Lemma \ref{Alemmatoboundathing}, we may argue as in the proof of Lemma \ref{lemma on bounding} that if $E(\Lambda^{N,f_t})\geq f_t+\epsilon$, then $\Lambda^{N}_t\geq f_t$ for all large $N$, and so $\Lambda_{\hat x_i^2}\geq \hat x^2_i$.

To determine the expectation, we bound the density $h$ below by $\tilde h$ which takes value $(1-\alpha)L$ only on $[\hat x_i,L\hat x_{i}]$. It is simple then to bound $E(\Gamma_N(f)_{\hat x_i^2 s})$ below by
$$(1-\alpha)L\hat x_i\left( \frac{L}{2}\left(1-\textrm{erf}\left(\frac{L}{\sqrt{2s}}\right)\right)-\frac{1}{2}\left(1-\textrm{erf}\left(\frac{1}{\sqrt{2s}}\right)\right)+\frac{\sqrt{s}}{\sqrt{2\pi}}(e^{-1/2s}-e^{-L^2/2s})\right)+O(\epsilon).$$

Taking $\epsilon$ sufficiently small, then the above is greater than $\hat x_i=f_{(\hat x^2_i)s}+\epsilon$ for $s\geq 1$. Therefore, we obtain the lower bound on $\Lambda_{\hat \xi^2_i}$.

We now prove Theorem \ref{blowuptheorem}.
\begin{proof}
 We suppose for a contradiction that there is some time $\zeta$ after which there are no further blow-ups. We take a suitably large even $i$ such that $\tau^*=\inf\{t>0:\Lambda_t\geq \hat x_i\}\geq \zeta$. To determine if the McKean-Vlasov type system \eqref{aligned:MVR} has jumps after this point, we restart the system at the time $\tau^*$, with initial density $\tilde V({\tau^*},.)$ corresponding to the measure of non-killed particles at this time. To determine existence of a blow-up for some time after $\tau^*$, we may then apply the criteria of Theorem \ref{blowuptheorem} by giving lower bounds on the density $\tilde V({\tau^*},.)$.

 To obtain a lower bound, we consider the density corresponding to killing particles at $\hat x_i$ while $t\leq \tau^*$, rather than at $\Lambda_t$. Re-centring this at $\Lambda_{\tau^*}=\hat x_i$, we obtain the density corresponding to the measure $$\vartheta(A-\hat x_i)=\int_{\hat x_i}^\infty\mathbb{P}(x+B_{\tau^*}\in A, \inf_{s\leq \tau^*}x+B_s>\hat x_i)h(x)dx.$$
This can be further bounded below by replacing $h$ in the above expression by $(1-\alpha)L\mathbbm{1}_{[\hat x_i,\hat x_{i+1}]}$. In total we obtain the lower bound 
$$\tilde V({\tau^*},x)\geq \frac{(1-\alpha)L}{2}\left(2\textrm{erf}\left(\frac{x}{\sqrt{2\tau^*}}\right)+\textrm{erf}\left(\frac{\hat x_{i+1}-\hat x_{i}-x}{\sqrt{2\tau^*}}\right)-\textrm{erf}\left(\frac{x+\hat x_{i+1}-\hat x_{i}}{\sqrt{2\tau^*}}\right)\right).$$

We recall $\hat x_{i+1}-\hat x_i=L^{i+1}$.  The above expression is decreasing in $\tau^*$ for $x\leq L^{i+1}$. Further $\textrm{erf}\left(\frac{x+L^{i+1}}{\sqrt{2\tau^*}}\right)-\textrm{erf}\left(\frac{L^{i+1}-x}{\sqrt{2\tau^*}}\right)$ may be bounded above by $2\textrm{erf}\left(\frac{x}{\sqrt{2\tau^*}}\right)e^{-\frac{L^{2i+2}-2L^{i+1}x}{2\tau^*}}$. Since $\hat x_i^2\geq\tau^*$ we therefore further decrease the lower bound to
$$\tilde V({\tau^*},x)\geq (1-\alpha)L\textrm{erf}\left(\frac{x}{\sqrt{2}\hat x_i}\right)\left(1-e^{-{\frac{L^{2i+2}}{8\hat x_i^2}}}\right) \textrm{ for }x\leq \frac{3L^{i+1}}{8}.$$

Taking $M=\frac{3L^{i+1}}{8}$ and integrating this over $[0,M]$, we obtain that 
$$\int_0^{M}\tilde V({\tau^*},x)dx\geq (1-\alpha)L \left(1-e^{-{\frac{L^{2i+2}}{8\hat x_i^2}}}\right)\left(M\textrm{erf}\left(\frac{M}{\sqrt{2}\hat x_i}\right)-\frac{\sqrt{2}\hat x_i}{\sqrt{\pi}}\left(1-e^{-\frac{M^2}{2\hat x_i^2}}\right)\right)$$

Since $\hat x_i\leq \frac{L^{i+1}}{L-1}$, this is bounded below by

$$(1-\alpha)L\left(1-e^{-\frac{(L-1)^2}{8}}\right)M\left(\textrm{erf}\left(\frac{3(L-1)}{8\sqrt{2}}\right)-\frac{8\sqrt{2}}{3(L-1)\sqrt{\pi}}\left(1-e^{-\frac{(L-1)^2}{16}}\right)\right).$$

For $L>6$ the coefficient multiplying $M$ is strictly bigger than $2$, and therefore the condition of Theorem \ref{blowuptheorem} applies. It follows that there must be a time greater than $\tau^*$ such that $\Lambda$ is discontinuous at this time, contradicting the assumption that there are no such discontinuities after time $\zeta$.
\end{proof}

\section{Proof deferred from Section \ref{uniquenessappendix}}
\label{uniquenessappendix2}
\begin{proof}[Proof of Proposition \ref{uniquecont}]
Let $\Lambda_t$ be a physical solution to \eqref{alternativetomvr} on $[0,T)$, which shall be compared to the minimal solution $\underline{\Lambda}_t$. Suppose that $t^*:=\inf\{t\geq 0:\underline{\Lambda}_{t}\neq \Lambda_{t}\}\wedge T$ is strictly less than $T$. Without loss of generality we may consider $t^*=0$. We can do this by considering the system started with initial density $\tilde V(t^{*-},x)$, which is the density corresponding to $\int_0^\infty \mathbb{P}(y+B_t-\underline{\Lambda}_{t^*-}\in dx,\tau\geq t^*)g(y)dy$. Further, since the physicality condition \eqref{alternativephysical} uniquely determines the jump size of any physical solution at time zero, we may assume that $\underline{\Lambda}_0=\Lambda_0=0$ by restarting the system with initial density $\tilde V(t^*,x)$. 

By assumption, there exists $\epsilon>0$ such that $\underline{\Lambda}$ is continuous on $[0,\epsilon)$. Since $t^*=0$, and $\Lambda,$ $\underline{\Lambda}$ are right continuous, we can find $\epsilon'>0$, and $0<t_1<t_2$ such that $\Lambda_{t}>\underline{\Lambda_t}+\epsilon'$ for $t\in(t_1,t_2)$.

We consider the expression given in Lemma \ref{Ito} for $\Lambda_t$ and $\underline{\Lambda}_t$, on the interval $[0,t_2]$.
Taking the difference of the expressions, then multiplying both sides by $e^{\lambda \underline{\Lambda}_{t_2}+\frac{\lambda^2t_2}{2}}$ we obtain:
\begin{align*}
    &\int_{\underline{\Lambda}_{t_2}}^\infty \underline{V}(t_2,x)e^{-\lambda x+\lambda \underline{\Lambda}_{t_2}}dx-\int_{\Lambda_{t_2}}^\infty V(t,x)e^{-\lambda x+\lambda\underline{\Lambda}_{t_2}}dx-\frac{1-e^{-\lambda(\Lambda_{t_2}-\underline{\Lambda}_{t_2})}}{\lambda}\\
    &=\frac{\lambda}{2}\int_0^{t_2}(e^{-\lambda \underline{\Lambda}_s}-e^{-\lambda \Lambda_{s}})e^{\lambda \underline{\Lambda}_{t_2}+\frac{\lambda^2(t_2-s)}{2}}ds\\
    &\hspace{0.4cm}-e^{\underline{\lambda\Lambda}_{t_2}+\frac{\lambda^2t_2}{2}}\sum_{s\leq t_2}e^{-\frac{\lambda^2s}{2}}\left(-\int_{\Lambda_{s^-}}^{\Lambda_s} V(s^-,x)e^{-\lambda x}dx-\frac{1}{\lambda}\Delta e^{-\lambda\Lambda_s}\right).
\end{align*}

We claim that the left hand side converges to zero as $\lambda\rightarrow\infty$. From Proposition \ref{6.1}, we obtain the bound $C\geq V(t_2,x),\underline{V}(t_2,x)$. Therefore the first two terms on the right are bounded above by $\frac{C}{\lambda}$. The final term is bounded above by $\frac{1}{\lambda}$, and the validity of the claim is clear.

As stated in the proof of Theorem \ref{conda4}, each term in the sum is negative. By definition of the minimal solution $\underline{\Lambda}_t\leq \Lambda_t$ for all $t$, and so the integral term is non-negative. Therefore, each term on the right hand side of the equation is non-negative. It must then hold that both the sum term, and the integral term converge to zero as $\lambda \rightarrow \infty$.

Consider only the integral term. Applying Fatou's lemma, it holds that 
$$0=\int_0^{t_2}\liminf_{\lambda\uparrow\infty}\frac{\lambda}{2}(e^{-\lambda \underline{\Lambda}_s}-e^{-\lambda \Lambda_{s}})e^{\lambda \underline{\Lambda}_{t_2}+\frac{\lambda^2(t_2-s)}{2}}ds.$$
Therefore:
$$\liminf_{\lambda\uparrow\infty}\frac{\lambda}{2}(e^{-\lambda \underline{\Lambda}_s}-e^{-\lambda \Lambda_{s}})e^{\lambda \underline{\Lambda}_{t_2}+\frac{\lambda^2(t_2-s)}{2}}=0 \textrm{ for }a.e. \hspace{0.1cm}s\in[0,t_2].$$ However, for $s\in (t_1,t_2)$, we can bound this by 
\begin{align*}
    \liminf_{\lambda \uparrow\infty }\frac{\lambda}{2}(1-e^{-\lambda \epsilon'})e^{\lambda (\underline{\Lambda}_{t_2}-\underline{\Lambda}_s)+\frac{\lambda^2(t_2-s)}{2}}\geq \liminf_{\lambda\uparrow\infty}\frac{\lambda}{2}(1-e^{-\lambda \epsilon'})e^{\frac{\lambda^2(t_2-s)}{2}}=\infty.
\end{align*}

This is a contradiction, and therefore it must hold that $\underline{\Lambda}_t=\Lambda_t$. 
\end{proof}

\begin{proof}[Proof of Corollary \ref{contunique3}]
    Let $\Lambda_t$ be any physical solution.
    We bound  \begin{align*}
    &\int_0^\infty \mathbb{P}(X_t^y\leq x,\tau^y>t)g(y)dy\\
    &\leq \int_0^\infty \mathbb{P}(\Lambda_t<y+B_t\leq \Lambda_t+x)g(y)dy,\\
    &\leq\int_0^\epsilon \mathbb{P}(\Lambda_t<y+B_t\leq \Lambda_t+x)dy+C\int_{\epsilon}^\infty\mathbb{P}(\Lambda_t<y+B_t\leq \Lambda_t+x)dy,\\
    &=\int_0^x\left(\frac{1}{2}\left(\textrm{erf}\left(\frac{\epsilon-(y-\Lambda_t)}{\sqrt{2t}}\right)+\textrm{erf}\left(\frac{y-\Lambda_t}{\sqrt{2t}}\right)\right)+\frac{C}{2}\left(1-\textrm{erf}\left(\frac{\epsilon-(y-\Lambda_t)}{\sqrt{2t}}\right)\right) \right)dy.\end{align*}

By the second assumption on $g$, $\Lambda_0=0$. The integrand in the final integral is smooth, and for sufficiently small $t$, the integrand is strictly less than $1$ at $y=0$. Therefore the integrand is strictly less than $1$ in a neighbourhood of zero, and so $\inf\{x:\int_0^\infty \mathbb{P}(X_t^y\leq x,\tau^y>t)g(y)dy<x\}=0$. By the physical jump condition $\Lambda_t$ cannot jump at time $t$. Therefore there is some $T'>0$ on which any physical solution to \eqref{alternativetomvr} is continuous. By Proposition \ref{uniquecont}, this implies physical solutions are unique on $[0,T']$.
\end{proof}
\begin{proof}[Proof of Proposition \ref{contunique2}]
 We approximate the density $g$ to allow application of the results of \cite{PDEresult}. Denote $K^*:=\inf\{x:\int_0^xg(y)dy>0\}$, and consider integers $n>K^*$. We denote $g_n(x)=g(x)\mathbbm{1}_{[K^*+\frac{1}{n},n]}(x)$.
 Fix an arbitrary $n\in \mathbb{N}$. It is simple to construct a sequence of piecewise continuous, non-negative functions $g_n^{k}$ which are supported on $[\frac{1}{n},n]$, and such that $\int_{\frac{1}{n}}^n|g_n(x)- g_n^k(x)|dx<\frac{1}{k}$ for all $k\in\mathbb{N}$. There exists some $k_n$ such that $\frac{1}{k_n}<\frac{1}{2}\int_0^{K^*+\frac{1}{n}}g(y)dy$. Therefore, for any $k\geq k_n$ it holds that $\int_0^xg_n^{k}(x)dx<x$ for all $x>0$, and that
 $$\int_0^x g_n^{k}(y)dy\leq \left(\int_{K^*+\frac{1}{n}}^xg(y)dy\right)\vee 0+\frac{1}{2}\int_0^{K^*+\frac{1}{n}}g(y)dy<\int_0^xg(y)dy \textrm{ for all }x>0.$$ 

From this, it follows that the minimal solution to \eqref{alternativetomvr} with initial density $g_R^{n,k}$ for $k\geq k_n$, must be bounded above at all times by the minimal solution to \eqref{alternativetomvr} with initial density $g$. This can be seen by noting as in \cite{minimal} that the minimal solution to \eqref{alternativetomvr} is given by $\lim_{n\rightarrow\infty}\Gamma^n(0)_t$ where $\Gamma^n$ is the $n$th iterate of the map $\Gamma$, and that $$\int_0^\infty \mathbb{P}(\inf_{s\leq t}(x+B_s-f_s))g(x)dx\geq \int_0^\infty\mathbb{P}(\inf_{s\leq t}(x+B_s-f_s))g_n^k(x)dx$$ for any cadlag function $f$, since $\mathbb{P}(\inf_{s\leq t}(x+B_s-f_s))$ is a decreasing function in $x$.  

Since $\limsup_{x\downarrow 0}g_n^{k}(x)=0$, and $\int_0^xg_n^{k}(y)dy<x$ for all $x>0$, the densities $g_n^k$ for $k\geq k_n$ satisfy the conditions of 
\cite[Theorem 1.1]{PDEresult}. Therefore, for all $k\geq k_n$ there is a global in time classical solution to the supercooled Stefan problem \eqref{pde} with initial data $g_n^{k}$. We denote this solution $\Lambda^{n,k}_t$. It is simple to check that $\Lambda^{n,k}_t$ is a solution to \eqref{alternativetomvr} with initial density $g_n^{k}$. Recall that we denote the minimal solution to \eqref{alternativetomvr} with initial density $g$ by $\underline{\Lambda}_t$. As stated above, for each $n\in\mathbb{N}$, and $k\geq k_n$ it holds that $\Lambda_t^{n,k}\leq \underline{\Lambda}_t$ for all $t\geq 0$.

We extend the processes $\Lambda^{n,k}_t$ to $[-1,\infty)$ by setting $\Lambda^{n,k}_t=0$ for $t\leq 0$. From simple modifications to Lemma 
\ref{Rtight}, we can determine that the sequence $(\Lambda^{n,k}_t)_{k\geq k_n}$ possesses a convergent subsequence under the M1 Skorokhod topology. From further simple modifications to the arguments of \cite[Proposition 2.1]{minimal} we can then determine that any subsequential limit is a solution to \eqref{alternativetomvr} with initial density $g_n$, which we denote $\Lambda^n_t$. This argument can be repeated to obtain a further subsequence of $\Lambda_t^n$, which converges to a solution to \eqref{alternativetomvr} with initial density $g$. Since each process in the sequence $(\Lambda^{n,k}_t)_{n\in\mathbb{N},k\geq k_n}$ is bounded above by $\underline{\Lambda}_t$, it therefore holds that any subsequential limit must be the minimal solution $\underline{\Lambda}_t$, and we determine that
$$\lim_{n\rightarrow\infty }\lim_{k\rightarrow\infty}\Lambda_t^{n,k}=\underline{\Lambda}_t \textrm{ for Lebesgue }a.e. \hspace{0.15cm} t>0.$$

Applying Lemma \ref{Ito}, we determine that for any $n\in \mathbb{N}$, and $k\geq k_n$,
\begin{align}\label{ito3}\frac{1}{\lambda}-\int_0^\infty g_n(x)e^{-\lambda x}dx=\frac{\lambda}{2}\int_0^\infty e^{-\lambda \Lambda^{n,k}_s-\frac{\lambda^2s}{2}}ds.\end{align}

Applying the Dominated Convergence Theorem, we can take the limit over $k$, then over $n$ in the above expression to obtain that

$$\frac{1}{\lambda}-\int_0^\infty g(x)e^{-\lambda x}dx=\frac{\lambda}{2}\int_0^\infty e^{-\lambda \underline{\Lambda}_s-\frac{\lambda^2s}{2}}ds.$$

As in the proof of Proposition \ref{uniquecont}, we let $\Lambda_t$ be a physical solution to \eqref{alternativetomvr}. Subtracting the expression above from the expression given in \ref{Ito} and sending $t\rightarrow\infty$, we obtain that

\begin{align*}
0&=\frac{\lambda}{2}\int_0^\infty (e^{-\lambda \underline{\Lambda}_s}-e^{-\Lambda_s})e^{-\frac{\lambda^2s}{2}}ds-\sum_{s<\infty}e^{-\frac{\lambda^2s}{2}}\left(\int_{\Lambda_{s^-}}^{\Lambda_s}V(s^-,x)e^{-\lambda x}dx-\frac{\Delta e^{-\lambda \Lambda_s}}{\lambda}\right),
\\&\geq \frac{\lambda}{2}\int_0^\infty (e^{-\lambda \underline{\Lambda}_s}-e^{-\Lambda_s})e^{-\frac{\lambda^2s}{2}}ds.
\end{align*}

In the inequality we have used that each term in the sum is negative, as shown in the proof of Theorem \ref{conda4}. Since $\underline{\Lambda}_t$ is the minimal solution to \eqref{alternativetomvr}, the term inside the integral is non-negative, and therefore must be zero almost everywhere. Therefore $\Lambda_t=\underline{\Lambda}_t$ for almost every $t>0$, and thus for all $t>0$ since both processes are right continuous.
\end{proof}

\section{Physicality of solutions}
\label{physicalappen}
We first prove the prerequisite Lemma~\ref{lem:physparticleestimate}. 
\begin{proof}
First note that ${N}(\Lambda_{t+\e}^{N} - \Lambda_{t-}^{N})$ is equal to the number 
of particles absorbed in the time interval $[t,t+\e]$. 
Define events $E_1^{i,k}$ by:
\begin{align*}
E_1^{i,k} := \big\{ X_{t-}^{i,N} - \frac{k}{N} -\sup_{h \leq \e} |B_{t+h}^{i,N} - B_{t}^{i,N}|\leq0,~ \tau_i \geq t \big\}.
\end{align*}
This is similar to the term in \cite[Lemma 3.10]{ledger2021mercy}, but without an $\epsilon(1+\sup_{s\leq t}|X^{i,N}_s|)$ term, since we do not include a Lipschitz drift in the particle movement.
Due to the minimal jump condition, for any $k\leq {N}(\Lambda_{t+\e}^{N} - \Lambda_{t-}^{N})$, it must hold that 
\begin{align}
\sum_{i=1}^\infty \mathbbm{1}_{E_1^{i,k}}\geq k.\label{eq:kEeventestimate}
\end{align}
If not then we could bound $\Lambda_{t+s}-\Lambda_{t^-}$, by $\frac{k-1}{N}$, by using Lemma \ref{Alemmatoboundathing}.

Fix $x\in\mathbb{R}$ such that $x \leq (\Lambda_{t+\e}^{N} - \Lambda_{t-}^{N}) - 2 \e^{1/3}$ and set 
$k_0 := \lfloor x + 2\e ^{1/3} \rfloor \leq N(\Lambda_{t+\e}^{N} - \Lambda_{t-}^{N}).$ This choice implies that $x \geq \frac{k_0}{N} - 2 \e^{1/3}$ as well as $\frac{k_0}{N} \geq x + 2 \e^{1/3} - \frac{1}{N}$. It also holds that \eqref{eq:kEeventestimate} holds for $k=k_0$.
Define an event $E_2^i$ by:
\begin{align*}
E_2^i := \big\{\sup_{h \leq \e} |B_{t+h}^{i,N} - B_{t}^{i,N}| \geq \e^{1/3}\big\}.
\end{align*}
On the event $E_1^{i,k_0} \cap (E_2^{i})^{c},$ it holds that $X_{t-}^i -  \frac{k_0}{N} \leq \e^{1/3},$ and hence
\begin{align*}
X_{t-}^{i,N} - x \leq X_{t-}^{i,N} -  \left(\frac{k_0}{N} - 2 \e^{1/3}\right) \leq 3\e^{1/3}.
\end{align*}
It follows that
\begin{align*}
\lawstopped_{t-}^{N}([0, x + 3\e^{1/3}]) \geq \frac{1}{N} \sum_{i=1}^{\infty} \mathbbm{1}_{E_1^{i,k_0}} \mathbbm{1}_{(E_2^{i})^c}.
\end{align*}
We set $A_1^M=\{\frac{1}{N}\sum_{i=1}^\infty \mathbbm{1}_{E_2^i}\mathbbm{1}_{\{|X_{-1}^{i,N}|\leq M\}}<\frac{1}{2}\epsilon^{1/3}\}$, and set \\ $A_2^M=\{\frac{1}{N}\sum_{i=1}^\infty\mathbbm{1}_{E^{i,k_0}\mathbbm{1}_{\{X^{i,N}_{-1}>M\}}}<\frac{1}{2}\epsilon^{1/3}\}$. We then set $E^M=A_1^M\cap A_2^M$. We find that on the event $E^M$: 
\begin{align*}
\lawstopped_{t-}^{N}([0, x + 3\e^{1/3}]) &\geq \frac{1}{N}\sum_{i=1}^\infty\mathbbm{1}_{E^{i,k_0}}\mathbbm{1}_{(E_2^i)^c}\mathbbm{1}_{\{X^{i,N}_{-1}\leq M\}},\\&\geq \frac{1}{N} \sum_{i=1}^{\infty} \mathbbm{1}_{E_1^{i,k_0}}-\frac{1}{N} \sum_{i=1}^{\infty} \mathbbm{1}_{E_1^{i,k_0}}\mathbbm{1}_{\{|X^{i,N}_{-1}|>M\}}- \frac{1}{N}\sum_{i=1}^{\infty} \mathbbm{1}_{E_2^i}\mathbbm{1}_{\{X^{i,N}_{-1}\leq M\}}, \\&\geq \frac{k_0}{N} - \e^{1/3} \geq (x + 2\e^{1/3} - {1}/{N}) - \e^{1/3} = x + \e^{1/3} - {1}/{N}.
\end{align*}
Taking $N \geq \e^{-1/3}$ implies $\e^{1/3} - 1/N \geq 0$. Therefore, we conclude that on $E^M$
\begin{align*}
\lawstopped_{t-}^{N}([0, x + 3\e^{1/3}]) \geq x \textrm{ for any }x \leq (\Lambda_{t+\e}^N - \Lambda_{t-}^{N}) - 2 \e^{1/3},
\end{align*}

and so $E^M\subset A^{N,\epsilon}$.
To prove the result of the lemma, it is then sufficient to choose $M$ appropriately such that $\mathbb{P}((E^M)^c)\leq 1-\tilde L\epsilon^{1/6}$.

We first bound $\mathbb{P}((A_2^M)^c)$.
By Markov's inequality, and the definition of $E_1^{i,k_0}$ \begin{align*}
\mathbb{P}((A_2^M)^c)&\leq \mathbb{P}(\Lambda^N_T>z)+\mathbb{P}((A_2^M)^c,\Lambda_T^N\leq z),\\
&\leq \mathbb{P}(\Lambda^N_T>z)+\frac{2\epsilon^{-1/3}}{N}\sum_{i=1}^\infty\mathbb{P}(E_1^{i,k_0},X^{i,N}_{-1}>M,\Lambda^N_T\leq z),\\
&\leq \mathbb{P}(\Lambda^N_T>z)\\
&+\frac{2\epsilon^{-1/3}}{N}\sum_{i=1}^\infty \mathbb{P}(X^i_{-1}>M,X_{-1}^{i,N}-\sup_{s\leq t}|B_{s}^{i,N}|-\sup_{h\leq \epsilon}|B_{t+h}^{i,N}-B_t^{i,N}|\leq z+\epsilon).
\end{align*}

We may take $z$ to be one greater than the $z$ given in Lemma \ref{lemma on bounding}. Thus, for all large enough $N$, the first term here is less than $\epsilon^{1/3}$. Using the bound \hyperlink{A1}{(A1)} on the density $g$, and the independence of the Brownian motions from the initial Poisson point process, the second term may be bounded above by
$$2C\epsilon^{-1/3}\int_M^\infty \mathbb{P}(x-z\leq \sup_{s\leq t}|B_s^{1}|+\sup_{h\leq \epsilon}|B_{t+h}^{1}-B^{1}_t|)dx.$$
Using the strong Markov property of the Brownian motion, and the reflection principle this is then bounded above by
$$2C\epsilon^{-1/3}\int_M^\infty \int_0^\infty\mathbb{P}(x-z-y\leq \sup_{s\leq t}|B_s^1|)\frac{4}{\sqrt{2\pi\epsilon}}e^{-\frac{y^2}{8\epsilon}}dydx.$$
From standard Gaussian tail bounds, the reflection principle, this may be bounded above by

\begin{align*}
&2C\epsilon^{-1/3} \left(\frac{c_1\sqrt{\epsilon}}{\sqrt{\pi}}+c_2\Phi\left(-\frac{(M-z)}{c_3 T}\right)+c_4\Phi\left(-\frac{M-z}{c_5\epsilon}\right)\right)\\
&\leq  C_1\epsilon^{1/6}+C_2\epsilon^{-1/3}\exp\left(-\frac{(M-z)^2}{2c_3T}\right).
\end{align*}

We take $M$ to be of size $2c_3T\sqrt{-\log(\epsilon)}+z$. The above is then bounded by $L_1\epsilon^{1/6}$ for some constant $L_1$.

We must also bound $\mathbb{P}\left(\left(A_1^{M}\right)^c\right)$.
This may be done using Markov's inequality in the same manner as above. From Markov's inequality we bound $\mathbb{P}((A_1^{M})^c)$ by 

$$\mathbb{P}(\Lambda^N_t>z)+\frac{2\epsilon^{-1/3}}{N}\sum_{i=1}^\infty\mathbb{P}(E_2^{i,k_0},X^i_{-1}\leq M,\Lambda^N_T\leq z).$$
As above, from Lemma \ref{lemma on bounding}, the first term may be made arbitrarily small by taking $N$ suitably large.
The sum term may be bounded by
$$\frac{2\epsilon^{-1/3}}{N}\sum_{i=1}^\infty\mathbb{P}(\sup_{h\leq \e}|B_{t+h}^{i,N}-B^{i,N}_t|\geq \epsilon^{1/3},X^i_{-1}\leq M).$$

 The above sum is the expected number of particles, initially started below $M$, such that their Brownian motions satisfy the appropriate conditions described above. Using the independence of the Brownian motions from the initial Poisson point process, and the bound on the density $g$ from condition \hyperlink{A1}{(A1)} we may then bound the above by
$${2\epsilon^{-1/3}}C(2c_3T+z)\sqrt{-\log(\epsilon)}\mathbb{P}(\sup_{h\leq \e}|B_{t+h}^1-B_t^1|\geq \epsilon^{1/3}).$$
From Markov's inequality, this may be bounded by $2C\epsilon^{1/3}\sqrt{-\log(\epsilon)}$. This may then be further bounded by $L_2\epsilon^{1/6}$ for all small enough $\epsilon$. 

In total then we have obtained that for all small enough $\epsilon$, and large enough $N$, choosing $M=2c_3T\sqrt{-\log(\epsilon)}+z$,
$$\mathbb{P}((E^M)^c)\leq \epsilon^{1/6}+L_1\epsilon^{1/6}+L_2\epsilon^{1/6}\leq \tilde L\epsilon^{1/6}.$$
Recalling that $E^M\subset A^{N,\epsilon} $, we obtain that 
$$\mathbb{P}(A^{N,\epsilon})\geq 1-\tilde L\epsilon^{1/6}$$
\end{proof}

We can now prove Proposition \ref{physi}.
\begin{proof}
In this proof we omit the $\hat{}$ hat notation when working with the representations here, to simplify notation. We write out the proof in steps which are equivalent to the steps in the proof of \cite[Theorem 6.4]{minimal}, and with similar notation where possible.

\underline{Step 1:} We first show convergence of the $\nu^N_{t^-}$ measures.
For an integer $K$, we define $\nu^{N,K}_t$,$\nu^{K}_t$ in a similar manner to $\nu^N_t$, defining for the re-centred particle system \eqref{aligned:Particlealpha}
\begin{align*}
\nu^{N,K}_{t}:=\frac{1}{N}\sum_{i=1}^\infty \delta_{X^{i,N}_t}\mathbbm{1}_{\{\tau_i^N>t\}}\mathbbm{1}_{\{X^{i,N}_{-1}\leq K\}},
\end{align*}
 while analogously for the McKean-Vlasov system \eqref{aligned:MVR}:
 \begin{align*}
\nu^K_{t}:=\sum_{i=1}^\infty \mathbb{E}(\delta_{X^{i}_t}\mathbbm{1}_{\{\tau_i>t\}}\mathbbm{1}_{\{X^{i}_{-1}\leq K\}}).
\end{align*}

The measures $\nu^{N,K}_{t^-},\nu^K_{t^-}$ are defined analogously, with $X_{t-}^{i,N},X_{t^-}^{i}$ and $\tau_i^N\geq t$, or $\tau_i\geq t$ respectively.

 For $t \geq 0,$ let $\pi_{t^-}(x)$ denote the map $x \rightarrow x_{t^-}$ and define the transformation $S:\tilde D \rightarrow \mathbb{R}$ through 
$S_{t-}(x) = \pi_{t-}(x)(1-\ellfunc_{t-}(x))$. Here $\lambda_{t^-}(x)=\lim_{s\uparrow t}\lambda_s(x)=\mathbbm{1}_{\{\tau(x)<t\}}$.
For $f \in C_b(\mathbb{R})$ and $V \in M(\tilde D)$ it holds that
\begin{align*}
(S_{t-})_*V(f) = \int_{}f(x_{t-}\mathbbm{1}_{\{\tau_0(x)\geq t\}})~\textrm{d}V(x) = \int_{}f(x_{t-})\mathbbm{1}_{\{\tau_0(x)\geq t\}}~\textrm{d}V(x) + f(0) V(\ellfunc_{t-}).
\end{align*}
It then follows that
\begin{align*}
\nu_{t-}^{N,K}( f)&= (S_{t-})_*(\mu^N_{S_K})( f) - {f(0)}\mu^N_{S_K}(\lambda_{t^-}), \quad
 \nu_{t-}^K( f) = (S_{t-})_*(\mu_{S_K})(f) - {f(0)} \mu_{S_K}(\lambda_{t^-}). 
\end{align*}
Recall from the proof of Lemma \ref{lemmac}, that there exists a co-countable set $J$ on which $\Lambda_t$ is continuous in $t$ almost surely. We claim that almost surely for $t\in J$ that $S_{t^-}$ is continuous at $\mu_{S_K}$ a.e. point in $\tilde D$. First note that $\pi_{t-}$ is
continuous at all paths $x \in \tilde D$ such that $t$ is a continuity point of $x$. Hence for $t\in J$, this is almost surely continuous under almost every $\mu$. We recall that the crossing property holds almost surely under almost every realisation of $\mu_{S_K}=\tilde h(\nu_{S_K},\Lambda)$. Arguing as in the manner of \cite[Proposition 5.8]{delarue2015particle}, we can determine that $\lambda_{t^-}$ is continuous at $t\in J$, almost surely under almost every $\mu_{S_K}$. Applying the Portmanteau Theorem, we may verify that $\nu_{t^-}^{N,K}\rightarrow \nu_{t^-}^K$ weakly for $t\in J$ almost surely. 

From this result, we can then determine that $\nu_{t^-}^N\rightarrow \nu_{t^-}$ vaguely. Indeed, if we consider a continuous, bounded function $f$ of bounded support, then $f$ has support contained certainly in some set of form $\{x_{t^-}\leq K'\}$ for some constant $K'$. It holds that $$\nu_{t^-}^N(f)=\nu_{t^-}^{N,K}(f)+\frac{1}{N}\sum_{i=1}^\infty f(X_{t^-}^{i,N})\mathbbm{1}_{\{\tau_i^N>t\}}\mathbbm{1}_{\{X_{-1}^{i,N}\geq K\}}.$$
The second term may be bounded by 
$$||f||_{\infty}\frac{1}{N}\sum_{i=1}^N\mathbbm{1}_{\{X_{-1}^{i,N}\geq K,X_{t^-}^{i,N}\leq K'\}}.$$

Through similar arguments to those used in bounding $\sum_{i=1}^\infty \mathbbm{1}_{A^{i,N}_K}$ in Lemma \ref{lemmac}, we obtain that the limsup of this term as $N\rightarrow\infty$ is almost surely bounded by a term of the form $C'\exp(-\frac{(K-K'-z)^2}{2T+2})$ for some constant $C'$. From the weak convergence of $\nu_{t^-}^{N,K}$ we then obtain that there are constants $C_1,C_2$ such that
$$\limsup_{N\rightarrow\infty}\nu_{t^-}^N(f)\leq \nu_{t^-}^K(f)+C_1\exp(-C_2 K^2).$$

We can also obtain that
$$\liminf_{N\rightarrow\infty}\nu_{t^-}^N(f)\geq\nu_{t^-}^K(f).$$
Sending $K\rightarrow\infty$, and applying the Dominated Convergence Theorem, we thus obtain that $\nu^{N}_{t^-}(f)\rightarrow \nu_{t^-}(f)$. Since $f$ was arbitrary this determines that $\nu^N_{t^-}\rightarrow \nu_{t^-}(f)$ vaguely. 

\underline{Step 2:} Fix a sufficiently small $\e>0$. We take the limit as $N\to\infty$ of the inequality given in Lemma \ref{lem:physparticleestimate}.
Recall that this inequality determines:
$$\mathbb{P}\left(\nu^N_{t^-}([0, x+3\epsilon^{1/3}])\geq x\quad\forall x\leq (\Lambda^N_{t+\epsilon}-\Lambda^N_{t^-})-L\epsilon^{1/3}\right)\geq 1-\tilde L\epsilon^{1/6}.$$

Fix $t, t+\e \in [0,T)\cap J$ and take $\beta>0$.
We introduce the events
\begin{align*}
A_\beta^{N}:= \left\{\Lambda_{t+\e}-\Lambda_{t^-}-\beta \leq \Lambda_{t+\e}^{N} - \Lambda_{t-}^{N}\right\}.
\end{align*}
By the assumption that $t,t+\epsilon\in J$, we know that $\Lambda_{t+\e}^{N} - \Lambda_{t-}^{N} \to \Lambda_{t+\e} - \Lambda_{t-}$, and therefore it holds that
$ \lim_{N\to\infty} \mathbb{P}({A_\beta^N}) = 1.$

We note that $\nu_{t^-}$ has no atoms almost surely. This may be argued in the same manner as \cite[Proposition 2.1]{2}. Recalling the vague convergence from Step 1, it holds that almost surely $\limsup_{N\rightarrow\infty}\nu^N_{t^-}[0,x+3\epsilon^{1/3}]= \nu_{t^-}[0,x+3\epsilon^{1/3}]$ since this set is bounded, and $\nu_{t^-}$ has no atoms almost surely.
Therefore, using the Reverse Fatou Lemma,
\begin{align*}
    &\mathbb{P}(\nu_{t^-}[0,x+3\epsilon^{1/3}])\geq x\quad\forall x\leq \Lambda_{t+\e}-\Lambda_{t^-}-2\e^{1/3}-\beta),\\
    &=\mathbb{P}(\limsup_{N\rightarrow\infty}\nu_{t^-}^N[0,x+3\epsilon^{1/3}])\geq x\quad\forall x\leq \Lambda_{t+\e}-\Lambda_{t^-}-2\e^{1/3}-\beta)\\
    &\geq \limsup_{N\rightarrow\infty}\left( \mathbb{P}(\nu_{t^-}^N([0,x+3\epsilon^{1/3}])\geq x\quad\forall x \leq (\Lambda_{t+\e}^{N} - \Lambda_{t-}^{N}) - 2\e^{1/3} )-\mathbb{P}((A_\beta^N)^c)\right)\\
    &\geq 1-\tilde L\epsilon^{1/6}.
\end{align*}

Since this holds for any $\beta>0$, sending $\beta\rightarrow 0$ yields that 
\begin{align}\mathbb{P}(\nu_{t^-}([0,x+3\epsilon^{1/3}])\geq x\quad\forall x< \Lambda_{t+\e}-\Lambda_{t^-}-2\e^{1/3})\geq 1-\tilde L\epsilon^{1/6}.\label{refertonow}
\end{align}

\underline{Step 3:}
By application of the Dominated Convergence Theorem, it can be checked that the map $t\rightarrow \nu_{t^-}$ is left continuous. Consider an arbitrary fixed $t<T$. Consider $t_n\leq  t\leq t_n+\epsilon_n$ with $t_n\uparrow t$, $\epsilon_n\downarrow t$ and $t_n,t_n+\epsilon_n\in J$. We may then determine from taking limits in \eqref{refertonow} that 
$$P(\nu_{t^-}[0,x]\geq x\quad \forall x<\Delta \Lambda_t)=1.$$

This shows then that almost surely $\Delta\Lambda_t\leq \inf\{x:\nu_{t^-}([0,x])<x\}$. 
Further the arguments of \cite[Proposition 1.2]{2} may be directly applied to the case of non-integrable density $g$, by changing $P(\tau\leq t)$ to $\sum_{i=1}^\infty P(\tau_i\leq t)$. 
This argument determines that $\Delta\Lambda_t\geq \inf\{x:\nu_{t^-}([0,x])<x\}$.
Therefore, for any subsequential limit $\Lambda$ of $\Lambda^N$, almost surely
$$\Delta\Lambda_t=\inf\{x:\nu_{t^-}([0,x])<x\}.$$
Taking a countable union of $t\notin J$, we then determine that the physicality condition holds for $\Lambda$ almost surely.
\end{proof}

\end{document}